\documentclass{article}

\usepackage{amsmath,amssymb,amscd,amsthm}
\usepackage{ascmac}
\usepackage[all]{xy}
\usepackage{enumerate}
\usepackage{cancel}
\usepackage[dvips]{graphicx}
\usepackage{authblk}
\usepackage{comment}
\usepackage{mathrsfs}
\usepackage{array,arydshln}
\usepackage{geometry}
\usepackage{url}
\theoremstyle{plain}
\newtheorem{thm}{Theorem}[section]
\newtheorem{pro}[thm]{Proposition}
\newtheorem{`thm'}[thm]{``Theorem''}

\newtheorem{lem}[thm]{Lemma}
\newtheorem{conj}[thm]{Conjecture}
\newtheorem{dfn-thm}[thm]{Definition-Theorem}
\newtheorem{dfn-pro}[thm]{Definition-Proposition}

\theoremstyle{definition}
\newtheorem{prob}[thm]{Problem}
\newtheorem{dfn}[thm]{Definition}
\newtheorem{fact}[thm]{Fact}
\newtheorem{cau}[thm]{Caution}
\newtheorem{asm}[thm]{Assumption}
\theoremstyle{remark}
\newtheorem{rmk}[thm]{Remark}
\newtheorem{rmks}[thm]{Remarks}

\newtheorem{exs}[thm]{Examples}

\newcommand{\bb}[1]{\mathbb{#1}}
\newcommand{\fra}[1]{\mathfrak{#1}}
\newcommand{\ca}[1]{\mathcal{#1}}
\newcommand{\scr}[1]{\mathscr{#1}}
\newcommand{\Cl}{\mathit{Cliff}}
\newcommand{\innpro}[3]{\left<#1,#2\right>_{#3}}
\newcommand{\Innpro}[3]{\left<\left<#1,#2\right>\right>_{#3}}
\newcommand{\inpr}[3]{\left(#1\middle|#2\right)_{#3}}
\newcommand{\sine}[1]{\frac{\sin#1\theta}{\sqrt{#1\pi}}}
\newcommand{\cosi}[1]{\frac{\cos#1\theta}{\sqrt{#1\pi}}}
\newcommand{\rot}{{\rm rot}}

\newcommand{\alg}{{\rm alg}}
\newcommand{\ind}{{\rm ind}}
\newcommand{\id}{{\rm id}}

\newcommand{\End}{{\rm End}}
\newcommand{\dom}{{\rm dom}}
\newcommand{\op}{{\rm op}}
\newcommand{\Lie}{{\rm Lie}}
\newcommand{\fin}{{\rm fin}}

\newcommand{\red}{{\rm red}}

\newcommand{\Vac}{{\rm Vac}}
\newcommand{\vac}{{\rm vac}}

\newcommand{\Ad}{{\rm Ad}}

\newcommand{\CAR}{{\rm CAR}}
\newcommand{\bra}[1]{\left(#1\right)}
\newcommand{\bbra}[1]{\left\{#1\right\}}
\newcommand{\bbbra}[1]{\left[#1\right]}
\newcommand{\Dirac}{\cancel{\partial}}
\newcommand{\SC}{\mathscr{SC}}
\newcommand{\ud}[1]{\underline{#1}}

\begin{document}
\title{An infinite-dimensional index theorem and the Higson-Kasparov-Trout algebra}
\author{Doman Takata \\
The University of Tokyo}

\date{\today}

\maketitle
\begin{abstract}

We have been studying the index theory for some special infinite-dimensional manifolds with a ``proper cocompact'' actions of the loop group $LT$ of the circle $T$, from the viewpoint of the noncommutative geometry. In this paper, we will introduce the $LT$-equivariant $KK$-theory and we will construct three $KK$-elements: the index element, the Clifford symbol element and the Dirac element. These elements satisfy a certain relation, which should be called the ($KK$-theoretical) index theorem, or the $KK$-theoretical Poincar\'e duality for infinite-dimensional manifolds. We will also discuss the assembly maps. 
\end{abstract}

\tableofcontents

\section{Introduction}

The Atiyah-Singer index theorem is one of the monumental work in differential topology \cite{ASi1,ASi2}. The original theorem was stated for closed manifolds,
but it has been generalized to many cases, for example: closed manifolds with compact group action \cite{ASe1}, families of closed manifolds \cite{ASi4}, operator algebras \cite{Con85,Con94}, and complete Riemannian manifolds with isometric, proper and cocompact group actions \cite{Kas84,Kas15}. Our dream is to formulate and prove an infinite-dimensional version of the Atiyah-Singer index theorem. In particular, we would like to study an infinite-dimensional version of \cite{Kas84,Kas15}.

As a first step of this dream, we have been studying the index problem for some special infinite-dimensional manifolds in \cite{T1,T2,T3}. \cite{T3} is the PhD thesis of the author, which contains \cite{T2}, most of \cite{T1} and some detailed study on proper $LT$-spaces which is to be explained soon. The following is a rough version of our problem of the present paper.
Let $T$ be the circle group, and let $LT$ be its loop group $C^\infty(S^1,T)$. Precise definitions will be explained later.


\begin{prob}\label{Main problem}
For an infinite-dimensional proper $LT$-space $\ca{M}$ equipped with a ``$\tau$-twisted'' $LT$-equivariant line bundle $\ca{L}\to \ca{M}$, and an $LT$-equivariant Spinor bundle $\ca{S}\to \ca{M}$, study an index problem of the ``Dirac operator'' $\ca{D}$ on $\ca{L}\otimes \ca{S}$, from the viewpoint of noncommutative geometry.
\end{prob}

In order to explain the results of the present paper, we need to recall the Kasparov index theorem. Let $X$ be a finite-dimensional complete Riemannian $Spin^c$-manifold, and let $\Gamma$ be a locally compact, second countable and Hausdorff group. Let $Cl_\tau(X):=C_0(X,\Cl(T^*X))$ be the Clifford algebra-valued function algebra. Suppose that $\Gamma$ acts on $X$ isometrically, properly and cocompactly. Then, the following two fundamental $KK$-elements are defined: the Dirac element $[d_X]\in KK_\Gamma(Cl_\tau(X),\bb{C})$ and the Mishchenko line bundle $[c]\in KK(\bb{C},C_0(X)\rtimes \Gamma)$.
We also suppose that the Spinor bundle $S$ of $X$ is $\Gamma$-equivariant, and a $\Gamma$-equivariant Dirac operator $D$ is given. Since $X$ is not supposed to be compact, $D$ can be non-Fredholm. However, in this situation, $D$ is a ``$C^*(\Gamma)$-Fredholm operator'' as follows.

\begin{fact}[\cite{Kas84,Kas15}]\label{Intro Kas index thm}
$(1)$ An analytic index $\ind(D)\in KK(\bb{C},C^*(\Gamma))$ can be defined.

$(2)$ By the given data, the following two $KK$-elements are defined: the index element $[D]\in KK_\Gamma(C_0(X),\bb{C})$ and the Clifford symbol element $[\sigma_D^{Cl}]\in KK_\Gamma(C_0(X),Cl_\tau(X))$. These two elements are related as follows: $[D]=[\sigma_D^{Cl}]\otimes_{Cl_\tau(X)}[d_X]$.

$(3)$ The analytic index $\ind(D)$ is completely determined by the $K$-homology class $[D]$ by the formula $[c]\otimes_{C_0(X)\rtimes \Gamma} j_\Gamma([D])$. Consequently, $\ind(D)$ is determined by the topological data $[\sigma_D^{Cl}]$.
\end{fact}

The main result of the present paper is an infinite-dimensional version of $(2)$ above. In order to formulate it, we need a ``function algebra of $\ca{M}$''. The Higson-Kasparov-Trout algebra (HKT algebra for short) plays the role of such algebra \cite{HKT,Tro,DT}, which will be explained soon. 
The HKT algebra for $\ca{M}$ is denoted by $\SC(\ca{M})$ in the present paper. The main result, Definition-Theorem \ref{dfn-thm index element}, \ref{dfn-thm Dirac element}, \ref{dfn-thm symbol element}, and Theorem \ref{Main theorem}, is summarized as follows.

\begin{thm}
Under the assumption of Problem \ref{Main problem}, we can define three $LT$-equivariant Kasparov modules corresponding to the index element, the Clifford symbol element and the Dirac element. The Kasparov module corresponding to the index element is a Kasparov product of the others. Formally,
$$[\widetilde{\ca{D}}]=[\widetilde{\sigma_{\ca{D}}^{Cl}}]\otimes_{\SC(\ca{M})}[\widetilde{d_{\ca{M}}}].$$
\end{thm}

Moreover, we also study an ``assembly map'', or $(3)$ of Fact \ref{Intro Kas index thm}. Unfortunately, it is not of satisfactory form. The aim of Section $6$ is, not only to describe the following result, but also to leave some ideas and observations, which can be useful.

\begin{thm}[Theorem \ref{thm assembly map}]
The index element $[\widetilde{\ca{D}}]$ is ``assemblable''. The value of the ``total assembly map'' coincides with the analytic index computed in \cite{T1,T2,T3}.
\end{thm}

Let us explain several details. We must begin with the explanations of the difficulties of the global analysis of infinite-dimensional manifolds and several observation on these difficulties. After that, we will talk about previous researches and the outline of our constructions of this paper.

Recall the Kasparov index theorem, \cite{Kas84,Kas15} or Fact \ref{Intro Kas index thm}. There, we need the following ingredients: $C_c^\infty(X,S)$, $D$ and $C^*(\Gamma)$ for $(1)$, $C_0(X)$, $Cl_\tau(X)$, $L^2(X,S)$, $C_0(X,S)$, $\sigma_D^{Cl}$ and $[d_X]$ for $(2)$, and $C_0(X)\rtimes \Gamma$, $j_\Gamma$ and $[c]$ for $(3)$. See also Section 2.4 for details.
All of them cannot be defined for infinite-dimensional cases in classical ways. For example, our manifold is non-locally compact, and hence a compactly supported continuous function on $\ca{M}$ is automatically trivial. Therefore, the classical $C_0(\ca{M})$ is trivial. According to a well-known theorem by Gelfand (see \cite{Mur} for example), the category of locally compact Hausdorff spaces is contravariantly equivalent to that of commutative $C^*$-algebras. However, this is another reason why commutative $C^*$-algebras are useless for the study of non-locally compact spaces\footnote{Even if a commutative $C^*$-algebra can be defined from $\ca{M}$, it must be the $C_0$-algebra of another locally compact Hausdorff space. Such an algebra does not carry information on the original manifold $\ca{M}$}. For the same reason, we can not define ``$Cl_\tau(\ca{M})$'' or ``$C_0(\ca{M},\ca{L}\otimes \ca{S})$''. Moreover, an infinite-dimensional manifold does not have a good measure, unlike the Lebesgue measure on a Euclidean spaces. Therefore, we can not define ``$C^*(LT)$'', ``$L^2(\ca{M},\ca{L}\otimes \ca{S})$'', ``$C_0(\ca{M})\rtimes LT$'' or ``$j_{LT}$''.

However, there are many noncommutative $C^*$-algebras which are ``equivalent'' to the $C_0$-algebras in some sense. It can be possible to study an infinite-dimensional version of such noncommutative $C^*$-algebras. In fact, the HKT algebra is a $C^*$-algebra which plays a role of ``$C_0(\bb{R})\widehat{\otimes} C_0(\bb{R}^\infty,\Cl(\bb{R}^\infty))$''. Once we define a $C^*$-algebra which carries some topological information on $\ca{M}$, we may forget $\ca{M}$ itself. In order to study the index theory of $\ca{M}$, it is enough to study the $LT$-equivariant $KK$-theory of this algebra.

By this observation, our abstract and naive problem becomes slightly concrete: {\it Find a $C^*$-algebra which carries some information on $\ca{M}$, and consider some $KK$-elements of this algebra}.

Then, how can we define the Hilbert space which plays a role of ``$L^2$-space''? In other words, how will we define the $K$-homology class?
The answer seems to lie on the representation theory. Recall the Peter-Weyl theorem \cite{KO,Kna}: For a compact group $G$, there is an isomorphism $L^2(G)\cong \oplus_{\lambda\in \widehat{G}} V_\lambda\otimes V_\lambda^*$ as representation spaces of $G\times G$, where $\widehat{G}$ is the set of isomorphism classes of irreducible unitary representations of $G$, and $V_\lambda$ is the corresponding representation space. In order to define the left hand side, we need a Haar measure on $G$ , but, for the right hand side, we do not need it. Since our manifold is, up to some finite-dimensional stuffs, an infinite-dimensional Lie group, such observation can be useful.

In order to carry out such ideas, we need to recall several previous researches: The HKT algebras and the representation theory of loop groups. Then, we review some related researches on  Hamiltonian loop group spaces, in order to explain the geometrical background.

Let us begin with the HKT algebra, which is one of the main objects in the present paper.
Recall the Bott periodicity $K_0(Cl_\tau(\bb{R}^n))= K^{-n}(\bb{R}^n)\cong \bb{Z}$, and the Bott periodicity map $K_0(Cl_\tau(\bb{R}^n))\to K_0(Cl_\tau(\bb{R}^{n+1}))$. Seeing these formulas, everyone thinks that ``$K^{-\infty}(\bb{R}^\infty)$'' must be isomorphic to $\bb{Z}$. This naive idea was justified in \cite{HKT}. N. Higson, G. Kasparov and J. Trout rewrote the Bott periodicity map as the induced map of a $*$-homomorphism between the graded-suspended algebras $C_0(\bb{R})\widehat{\otimes} Cl_\tau(\bb{R}^n)\to C_0(\bb{R})\widehat{\otimes} Cl_\tau(\bb{R}^{n+1})$, and defined the HKT algebra $\SC(\bb{R}^\infty)$ by the $C^*$-algebra inductive limit of this sequence. The HKT algebra was used to solve the Baum-Connes conjecture for a-T-menable groups in \cite{HK}, and generalized to Hilbert bundles \cite{Tro}, and to Fredholm manifolds \cite{DT}. Recently, the HKT algebra appeared in the context of the gauge theory \cite{Kat}.
We will extensively study the HKT algebras in Section 4.

About the loop groups, we should explain something about the representations.
Loop groups have the rotation symmetry: $\theta_0.l(\theta)=l(\theta+\theta_0)$. We
would like to deal with representations of $LG$ satisfying the following conditions: they are infinite-dimensional; they reflect the rotation symmetry; and they satisfy a certain finiteness condition. Positive energy representations (PERs for short) are such ones \cite{PS}. It is known that PERs must be projective. Let $LG^\tau$ be an appropriate $U(1)$-central extension of $LG$. On an irreducible PER, the added center $U(1)\subseteq LG^\tau$ acts on the representation space with a fixed weight, which is called the level of the representation.
Surprisingly, there are only finitely many equivalence classes of irreducible PERs at fixed level. The Wess-Zumino-Witten Hilbert space (WZW Hilbert space) is defined by $\oplus_{\lambda\in \widehat{LG}_\tau}V_\lambda\otimes V_\lambda^*$, where $\widehat{LG}_\tau$ is the set of isomorphism classes of irreducible PERs of $LG$ at level $\tau$. By an imitation of the Peter-Weyl theorem \cite{KO,Kna}, the WZW Hilbert space is regarded as the ``$L^2$-space of $LG$''.
It is used in the context of the conformal field theory \cite{Gaw}.

D. Freed, M. Hopkins and C. Teleman defined a ``Dirac operator'' as an operator acting on the tensor product of the WZW Hilbert space and the ``Spinor space'' which is defined by the complex structure in \cite{FHTII}.\footnote{In fact, what they defined is an operator on $V\otimes S$, wehre $V$ is a PER and $S$ is the Spinor. However, that operator clearly modeled on the Dirac operator on the $L^2$-space. See Section 2 in \cite{FHTII} also.}
They used this operator in order to define an isomorphism (called the FHT isomorphism here) between the representation group of $LG$ at level $\tau$ and the $\tau$-twisted $G$-equivariant $K$-theory of $G$. What they tried to define was a map $R^\tau(LG)\to K_{LG}^{\tau+\dim(LG)}(L\fra{g}^*)$, although $K_{LG}^{\tau+\dim(LG)}(L\fra{g}^*)$ does not make sense. There is a ``local equivalence'' between $L\fra{g}^*//LG$ and $G//G$ as groupoids. They used this equivalence to define the FHT isomorphism. However, what they actually defined from a PER of $LG$ is an element of $K_G^{\tau +\dim(G)}(G)$.

This result is important for us at least for the following two reasons. Firstly, the construction of such Dirac operator is essential. The Dirac operator is an infinite sum, but this infinite sum converges for the algebraic reason. For details on this technique, see also \cite{T1,T3}. The most important technical issues on this paper is in some sense a quantitative version of theirs.
Secondly, the ``FHT isomorphism'' $R^\tau(LG)\to KK_{LG}^{\tau+\dim(LG)}(L\fra{g}^*)\cong K^{LG}_{\tau+0}(L\fra{g}^*)$ reminds us of the Baum-Connes assembly map $RK^\Gamma_0(\ud{\ca{E}}\Gamma)\to K_0(C^*_\red(\Gamma))$ \cite{Val,GHT}, where $\ud{\ca{E}}\Gamma$ is the universal example for proper actions of $\Gamma$. The ``isomorphism $KK_{LG}^{\tau+\dim(LG)}(L\fra{g}^*)\cong K^{LG}_{\tau+0}(L\fra{g}^*)$'' is given by the Poincar\'e duality.
We expect the following things: {\it The $\tau$-twisted $LG$-equivarint $K$-homology group $K_{\tau+0}^{LG}(L\fra{g}^*)$ makes sense in terms of noncommutative geometry; $L\fra{g}^*$ is one of the universal examples for the proper actions of $LG$ in the sense of \cite{BCH}; and the FHT isomorphism can be regarded as the inverse of the Baum-Connes assembly map.} This conjecture is a reformulation of the result of \cite{Loi}: In this paper, Y. Loizides proved that the inverse of the FHT isomorphism is given by a kind of assembly map, where the twisted $K$-theory is identified with the twisted $K$-homology.
We hope our project provides a good tool to describe his result in the purely infinite-dimensional framework, and in a direct way. We will not study further this topic.

Let us explain Hamiltonian loop group spaces from the viewpoint of the geometric quantization, before going to the explanations of the outline of the constructions of the present paper. A Hamiltonian $LG$-space is an infinite-dimensional symplectic manifold equipped with an $LG$-action and a proper moment map taking values in $L\fra{g}^*$. It was introduced in \cite{MW}. 
There is a one to one correspondence between (infinite-dimensional symplectic) Hamiltonian $LG$-spaces and (compact, possibly non-$Spin^c$) quasi-Hamiltonian $G$-spaces \cite{AMM}. E. Meinrenken studied the geometric quantization problem of Hamiltonian $LG$-spaces using this correspondence and the FHT isomorphism in \cite{Mei}. Y. Song studied the same problem in a direct way: He define a Hilbert space which can be regarded as an ``$L^2$-space'' of the original Hamiltonian $LG$-space. In fact, he defined it by the $L^2$-space of a Hilbert bundle over the corresponding quasi-Hamiltonian $G$-space. Then, the method of abelianization was introduced in \cite{LMS}. Our $\Phi^{-1}(\fra{t})$ is a special case of the abelianization. Using this method, the quantization problem of Hamiltonian $LG$-space was studied in \cite{LS}. 

Our problem was very inspired by these researches, but {\bf the method is quite different.} Our method is more noncommutative geometrical, and more direct than theirs. However, we have not studied the noncomutative group cases, and it seems to be difficult to generalize our work to such general cases without much efforts.

Let us outline the present paper. We can precisely set the problem here. A Hamiltonian $LT$-space with a proper moment map automatically satisfies the following conditions. One can find the proof of this fact in \cite{T3}.

\begin{dfn}
An infinite-dimensional manifold $\ca{M}$ is a ``proper $LT$-space'' if it has an $LT$-action and a proper, equivariant and smooth map $\Phi:\ca{M}\to L\fra{t}^*$. 
The $LT$-action on $L\fra{t}^*$ is not the coadjoint action but the gauge action $l.A=A-l^{-1}dl$ for $l\in LT$ and $A\in L\fra{t}^*$.
We also suppose that $\Phi^{-1}(\fra{t})$ is even-dimensional and $T\times \Pi_T$-equivariantly $Spin^c$, that is, $\ca{M}$ is ``even-dimensional and $Spin^c$''.
\end{dfn}
\begin{prob}
For a proper $LT$-space $\ca{M}$ with a $\tau$-twisted $LT$-equivariant line bundle $\ca{L}$, construct an $LT$-equivariant Spinor bundle $\ca{S}$, and study the $KK$-theoretical $LT$-equivariant index theory. More precisely, construct three Kasparov modules corresponding to the index element, the Clifford symbol element and the Dirac element, and prove the $KK$-theoretical index theorem equality. Then, study the assembly map.
\end{prob}
\begin{rmks}
$(1)$ The phrase ``infinite-dimensional and even-dimensional'' sounds too strange. However, everyone probably thinks that a complex manifold or a symplectic manifold must be ``even-dimensional'', even if the manifold is infinite-dimensional. In our case, $\ca{M}$ is in fact $\Phi^{-1}(\fra{t}) \times U$ as observed in \cite{T1,T3}, where $U$ is an infinite-dimensional complex vector space. In this sense, our manifold is ``even-dimensional''.

$(2)$ The $Spin^c$-condition for infinite-dimensional manifolds is not clear. We adopt an easy condition: Our manifold is of the form ``a finite-dimensional $Spin^c$-manifold $\times$ an infinite-dimensional almost complex manifold''.

\end{rmks}

The core idea of the present paper is to use the HKT algebra as the function algebra of $\ca{M}$. Unfortunately, the HKT algebra is the substitute of, not the $C_0$-algebra, but the $Cl_\tau$-algebra. However, for an even-dimensional $G$-equivariantly $Spin^c$-manifold $X$, $C_0(X)$ is $G$-equivariantly Morita-Rieffel equivalent to $Cl_\tau(X)$, as explained in Lemma \ref{C0 is KK-equiv to Cltau}. Using this equivalence and the graded suspension $C_0(\bb{R})\widehat{\otimes}-$, we will reformulate the Kasparov index theorem in Section 2.5. 
In this reformulated version, only $C_0(\bb{R})\widehat{\otimes} Cl_\tau(X)$ appears as function algebras of the manifold. This $C^*$-algebra is the model of the HKT 
algebra.

In Section 3, we will introduce the $LT$-equivariant $KK$-theory. We will define the Kasparov product in the $LT$-equivariant setting at the level of modules there, but we will not prove that the product is well-defined at the level of $KK$-theory. However, the following phrase still makes sense: ``{\it An $LT$-equivariant Kasparov module\footnote{not the $KK$-element!} is a Kasparov product of other two Kasparov modules}''. Then, we will set the problem precisely, and we will divide it into two parts just like \cite{T1,T2,T3}. Consequently, we will find that we may concentrate on one infinite-dimensional manifold $U$ with the standard $U$-action.

In Section 4, we will study the HKT algebra: its definition, the group action and a representation on a Hilbert $\scr{S}$-module.
The original HKT algebra inherits a locally compact group (affine) action on the underling space, as explained in \cite{HKT,HK}. We will prove that the algebra also inherits the $U$-action.

Thanks to this result, we can consider $U$-equivariant $KK$-theory of $\SC(U)$. We will also introduce a Hilbert $\scr{S}$-module on which $\SC(U)$ acts, using the WZW model. Although the both of the HKT algebra and the WZW Hilbert space are not new objects, the connection between them is probably a new result.

The main part of the present paper is Section 5. We will define three $U$-equivariant Kasparov modules there.
In order to define the index element $[\widetilde{\Dirac}]\in KK_{U}^\tau(\SC(U),\scr{S})$ and the Dirac element $[\widetilde{d_{U}}]\in KK_{U}(\SC(U),\scr{S})$, we need to define a Hilbert $\scr{S}$-module which is equipped with a left $\SC(U)$-action, and appropriate operators. In fact, we will use the same module and the same operator. The operator is based on the Dirac operator in \cite{FHTII} or \cite{Son}. Our operator is not actually equivariant, but ``almost equivariant'', which can be adopted as an operator defining a Kasparov module.
Once a pair of a Hilbert module and an operator which looks like an index element is defined, we only have to check several conditions to be an equivariant Kasparov module. These are purely functional analytical tasks.
The story is quite simple, but the calculation is very complicated. We must frequently verify some infinite sums converge, because the infinite sum defining the Dirac operator becomes an actual infinite sum, after multiplying an element of the HKT algebra with an element of an ``$L^2$-section'', unlike \cite{FHTII}. {\bf This is the biggest difference from the previous researches.}
We will carry out these programs in Section 5.1. 

For the Dirac element, we use the same operator and the same Hilbert module, but the different $U$-action. It will be an easy task to modify the group action, and it will be studied in Section 5.2. in this subsection, we will also explain why the model of our Dirac element is appropriate.

For the Clifford symbol element, we need a $\tau$-twisted $U$-equivariant $(\SC(U),\SC(U))$-bimodule. For this aim, we modify the $U$-action on $\SC(U)$, which is $\tau$-twisted.
Once these Kasparov modules are defined, it is not difficult to prove the main result: Theorem \ref{Main theorem}.

In the final section, we will deal with ``assembly maps'' for our case. In locally compact cases, assembly maps played an important role in \cite{Per} to study the cohomological formula of the Kasparov index theorem. The value of our assembly map coincides with the analytic index in \cite{T2,T3}. In this sense, our construction is ``correct''.
Our assembly map is well-defined at the level of $KK$-theory, but we will not be able to introduce the crossed product and the descent homomorphism. This is the reason why our assembly map is not very satisfying for us. In the final subsection, we will write some observations on the crossed products and descent homomorphisms.

Finally, we add some notational remarks.

\begin{rmks}
$(1)$ From now on, we use the graded language, without any special notice. We summarize it here, although we believe that they are standard.
\begin{itemize}
\item For a $\bb{Z}_2$-graded vector space (Hilbert space, algebra or Hilbert module) $A$ $A=A_0\widehat{\oplus} A_1$, the grading is denoted by $\partial$: for $a\in A_i$, $\partial a=i$. The grading homomorphism is always denoted by $\epsilon$ (or $\epsilon_A$): $\epsilon(a)=(-1)^{\partial a}a$ for $a\in A_0\cup A_1$.
\item If two algebras (vector spaces, Hilbert modules and so on) $A=A_0\widehat{\oplus} A_1$ and $B=B_0\widehat{\oplus}  B_1$ are given, $A\otimes B$ is graded by, as usual, $[(A_0\otimes B_0)\oplus (A_1\otimes B_1)] \widehat{\oplus}[(A_0\otimes B_1)\oplus (A_1\otimes B_0)]$. The product (or the ring-action on the module) is also twisted: $[a\otimes b]\cdot [a'\otimes b']=(-1)^{\partial b \partial a'}aa'\otimes bb'$, if $a'$ and $b$ are homogeneous.
Such tensor product is often written as $A\widehat{\otimes}B$, but we omit the ``hat''.
We sometimes remind the reader of this rule to emphasize it.
\item The commutator is always graded $[F_1,F_2]=F_1F_2-(-1)^{\partial F_1\partial F_2}F_2F_1$, where $F_1$ and $F_2$ can be just operators, not elements of some algebra. They can be unbounded operators. If the space on which an operator $F$ acts is $\bb{Z}_2$-graded, we say that $F$ is odd or even, if $F$ reverses or preserves the grading, respectively.
\end{itemize}

$(2)$ We sometimes use the $C^*$-algebra $C_0(\bb{R})$ equipped with a non-trivial $\bb{Z}_2$-grading defined by $\epsilon(f)(t):=f(-t)$. $\scr{S}$ denotes this graded algebra.

$(3)$ We frequently use some objects which are substitutes of something that we can not define in classical methods. Such substitute are denoted by the standard symbol with the underline.
For example, our WZW model for $LT$ is a substitute of ``$L^2(LT)$''. In such a case, $\ud{L^2(LT)}$ denotes the WZW model. Here is another example: For a Fredholm operator $T:H_1\to H_2$, $\ind(T)$ is a substitute of ``$\dim(H_1)-\dim(H_2)$'', and so we may write $\ind(T)$ as $\ud{\dim(H_1)-\dim(H_2)}$.

$(4)$ We use the following notations: For a Hilbert $B$-module $\ca{E}$, $\ca{L}_B(\ca{E})$ or simply $\ca{L}(\ca{E})$ is the $C^*$-algebra consisting of adjointable operators on $\ca{E}$, and $\ca{K}_B(\ca{E})$ or simply $\ca{K}(\ca{E})$ is the $C^*$-algebra consisting of $B$-compact operators, following \cite{JT}.

$(5)$ A group action is denoted by $g.v$, for a group element $g$ and an element $v$ of a set on which the group acts. A product in an algebra or an algebra action on a module is denoted by $a\cdot \phi$, for an element $a$ of an algebra and an element $\phi$ of a module or the algebra on which the algebra acts.

$(6)$ Throughout this paper, parentheses $\inpr{\bullet}{\bullet}{}$ means an inner product which is linear in the second variables and anti-linear in the forst one. An inner product takes values in some $C^*$-algebra (including $\bb{C}$). On the other hand, angle brackets $\innpro{\bullet}{\bullet}{}$ or $\Innpro{\bullet}{\bullet}{}$ is a bilinear form (or pairing) taking values in $\bb{R}$ or $\bb{C}$.

\end{rmks}

\section{The finite-dimensional case}

In this section, we review the Kasparov index theorem, in order to clarify notations. Then, we will reformulate it in an appropriate form for our problem. This section does not contain many new results: Proposition \ref{our condition is OK} and related things are probably new; Section 2.5 is maybe also new, but it is just an exercise for the equivariant $KK$-theory.

\subsection{Unbounded equivariant Kasparov modules}

The bounded picture is convenient for the general study of the $KK$-theory. This picture works very well, even for equivariant cases.
On the other hand, although natural $KK$-elements appearing in geometry are often unbounded, general theory for unbounded cycles are quite difficult because of functional analytical issues on Hilbert modules. 

Unbounded Kasparov modules was introduced in \cite{BJ}, and the criterion to be a Kasparov product was studied by \cite{Kuc} and \cite{Mes}. 
In this subsection, we prepare some issues on unbounded equivariant Kasparov modules.
First of all, let us recall several definitions. Let $A$ and $B$ be separable $C^*$-algebras.

\begin{dfn}
Let $E$ be a countably generated Hilbert $B$-module equipped with a $*$-homomorphism $A\to \ca{L}_B(E)$, and let $D$ be a (possibly) unbounded, densely defined, regular\footnote{A densely defined adjointable operator $D$ on $E$ is said to be regular if and only if $1+D^2$ has dense range. This condition is not automatic.}, odd and self-adjoint operator.
The pair $(E,D)$ is called an {\bf unbounded Kasparov $(A,B)$-module}, if the following conditions are satisfied:
\begin{itemize}
\item The set of $a\in A$ such that $[a,D]$ extends to a bounded operator, is dense.
\item $a(1+D^2)^{-1}$ belongs to $\ca{K}_B(E)$ for each $a\in A$.
\end{itemize}
\end{dfn}

The first condition means that ``$D$ is of first order'', and the second one means that ``$D$ is an elliptic operator''.

Under the following bounded transformation, an unbounded Kasparov module is transformed into a bounded one \cite{BJ,Bla}.

\begin{dfn}
Let $\fra{b}$ be a function on $\bb{R}$ defined by $\fra{b}(x):=\frac{x}{\sqrt{1+x^2}}$. For an unbounded Kasparov module $(E,D)$, the {\bf bounded transformation} of this module is defined by
$$\fra{b}(E,D):=\bra{E,\fra{b}(D)}=\bra{E,\frac{D}{\sqrt{1+D^2}}}.$$

\end{dfn}

Let us move to the equivariant $KK$-theory. For details, see \cite{Kas88,Bla}.
Let $G$ be a second countable, locally compact and Hausdorff group. Suppose that the $C^*$-algebras $A$ and $B$ are $G$-$C^*$-algebras, that is, $G$ acts there and the map $g\mapsto g(a)$ is continuous for any $a\in A$ or $B$.

\begin{dfn}\label{dfn of equivariant KK}
Let $E$ be a $\bb{Z}_2$-graded, countably generated, $G$-equivariant Hilbert $B$-module equipped with a $G$-equivariant $*$-homomorphism $A\to \ca{L}(E)$, and $F\in \ca{L}(E)$ be an odd and self-adjoint operator. The pair $(E,F)$ is called a {\bf bounded $G$-equivariant Kasparov $(A,B)$-module}, if the following conditions are satisfied:
\begin{itemize}
\item $[a,F]$, $a(1-F^2)$, $a(g(F)-F)$ are $B$-compact operators for any $a\in A$ and $g\in G$.
\item The map $G\ni g\mapsto g(aF)\in \ca{L}(E)$ is norm-continuous for any $a\in A$.
\end{itemize}
The set of homotopy classes of $G$-equivariant Kasparov $(A,B)$-modules turns out to be an abelian group with respect to the direct sum of modules. This group is denoted by $KK_G(A,B)$. For a $G$-equivariant Kasparov $(A,B)$-module $(E,F)$, the corresponding $KK$-element (or the homotopy class represented by $(E,F)$) is denoted by $[(E,F)]$.

\end{dfn}

The condition to be ``equivariant'' is not very simple, but the small failure to be ``actually equivariant'' is useful. See also Section 5.2. On the other hand, that failure makes it difficult to translate the definition to the unbounded picture.
The aim of this subsection is to give a sufficient condition so that an unbounded Kasparov module with a group action is transformed into a bounded equivariant Kasparov module by the bounded transformation.
One of the easiest sufficient condition is probably to assume the operator to be ``actually equivariant''. Obviously such unbounded ``actually equivariant'' Kasparov module is transformed into a bounded ``actually equivariant''  Kasparov module by the bounded transformation. It means that the assumption is too strong. The following is our current answer.

\begin{dfn}
An unbounded Kasparov $(A,B)$-module $(E,D)$ is an {\bf unbounded $G$-equivariant Kasparov $(A,B)$-module} if the following conditions are satisfied:
\begin{itemize}
\item $E$ is $G$-equivariant $(A,B)$-bimodule.
\item $G$ preserves $\dom(D)$.
\item $g(D)-D$ and $[D,D-g(D)]\cdot \frac{D}{1+\lambda+D^2}$ are bounded for all $g\in G$ and $\lambda\geq 0$.
\item The map $g\mapsto g(D)-D$ is norm-continuous.
\end{itemize}

\end{dfn}

\begin{rmk}
This assumption is very relaxed in comparison with being ``actually equivariant'', but still too strong. We hope that some better sufficient condition will be found.

However, our conditions have geometrical sense. Suppose that $A=C_0(M)$ for some smooth manifold $M$, $B=\bb{C}$, and $D$ is a first order differential operator on $M$. If $g(D)-D$ is bounded, the symbol must be invariant, and the potential is at most linear along the direction of the group action. A sufficient condition so that the map $g\mapsto g(D)-D$ is continuous, is the following: The potential function is uniformly continuous. The rest condition is less clear. However, if the manifold is compact, this condition is almost automatic, thanks to the elliptic regularity.

\end{rmk}

Let us prove that our unbounded equivariant Kasparov modules are transformed to bounded ones by the bounded transformation.

%




\begin{pro}\label{our condition is OK}
For an unbounded $G$-equivariant Kasparov $(A,B)$-module $(E,D)$, the bounded transformation
$\fra{b}(E,D)$ is a bounded $G$-equivariant Kasparov $(A,B)$-module.
\end{pro}
\begin{proof}
As explained in \cite{Bla}, $\fra{b}(E,D)$ is actually a Kasparov module. We need to check the equivariance condition and the $G$-continuity condition.

The $G$-continuity is obvious because $g(\fra{b}(D))=\fra{b}(g(D))$. For the equivariance, we use the usual technique: $\fra{b}$ can be rewritten as
$$\fra{b}(x)=\frac{x}{\sqrt{1+x^2}}=\frac{1}{\pi}\int_0^\infty \frac{1}{\sqrt{\lambda}}\frac{x}{1+x^2+\lambda}d\lambda.$$
Let us calculate $a(\fra{b}(D)-g(\fra{b}(D)))$ using this formula.

\begin{align*}
a\bra{\fra{b}(D)-g(\fra{b}(D))} &=a\bbbra{(1+D^2)^{-\frac{1}{2}}D-g(D)(1+g(D)^2)^{-\frac{1}{2}}} \\
&= a\cdot\bbbra{\frac{1}{\pi}\int_0^\infty \frac{1}{\sqrt{\lambda}}\cdot\frac{1}{1+D^2+\lambda}\cdot Dd\lambda- \frac{1}{\pi}\int_0^\infty \frac{1}{\sqrt{\lambda}}\cdot g(D)\cdot\frac{1}{1+g(D)^2+\lambda}d\lambda} \\
&= a\cdot \frac{1}{\pi}\int_0^\infty\frac{1}{\sqrt{\lambda}} \frac{1}{1+D^2+\lambda}
\Bigl[D(1+g(D)^2+\lambda)-(1+D^2+\lambda)g(D)\Bigr]\frac{1}{1+g(D)^2+\lambda}d\lambda \\
&=a\cdot \frac{1}{\pi}\int_0^\infty \frac{1}{\sqrt{\lambda}}\frac{1}{1+D^2+\lambda}
\Bigl[(1+\lambda)(D-g(D))+D[g(D)-D]g(D)\Bigr]\frac{1}{1+g(D)^2+\lambda}d\lambda \\
&=\frac{1}{\pi}
\int_0^\infty \frac{1}{\sqrt{\lambda}}a\cdot \frac{1}{1+D^2+\lambda}\cdot (D-g(D))\cdot\bra{\frac{1+\lambda}{1+g(D)^2+\lambda}}
d\lambda \\
&\ \ \ \ \ \  \ \ \ +\frac{1}{\pi}
\int_0^\infty \frac{1}{\sqrt{\lambda}}a\cdot\frac{1}{1+D^2+\lambda}\cdot
D[g(D)-D]g(D)\cdot\bra{ \frac{1}{1+g(D)^2+\lambda}}d\lambda.
\end{align*}

The first term is compact, because $D-g(D)$ is bounded, $\frac{1+\lambda}{1+g(D)^2+\lambda}$ is uniformly bounded in $\lambda$ and $a \cdot (1+D^2+\lambda)^{-1}$ is a compact operator whose norm is at most $\frac{\|a\|}{1+\lambda}$. For the second one, we notice that 
$$D[g(D)-D]g(D)=-[g(D)-D]^2g(D)+[g(D),g(D)-D]g(D)-[g(D)-D]g(D)^2.$$
Multiplying by $\frac{1}{1+g(D)^2+\lambda}$ from the right, we get bounded operators whose norms are uniformly bounded in $\lambda$.
More precisely, for the second term,
$$[g(D),g(D)-D]g(D)\frac{1}{1+g(D)^2+\lambda}=g\bra{[D,D-g^{-1}(D)]\frac{1}{1+D^2+\lambda}}$$
is uniformly bounded, by the assumption. The other terms are easier to handle.

\end{proof}

We have not mentioned the equivalence relations on unbounded Kasparov modules. In fact, even for non-equivariant $KK$-theory, we do not know what relation in the unbounded picture completely corresponds to the homotopy in the bounded picture. See \cite{DGM} for the current state. We only verify that homotopy equivalence is a sufficient condition. $G$, $A$ and $B$ are as usual.

\begin{dfn}
Two unbounded $G$-equivariant Kasparov $(A,B)$-modules $(E_0,D_0)$ and $(E_1,D_1)$ are {\bf homotopic} if there exists an unbounded $G$-equivariant Kasparov $(A,BI)$-module $(E',D')$ such that $\pi_{i*}(E',D')\cong (E_i,D_i)$ for $i=0,1$.
\end{dfn}

If two unbounded $G$-equivariant Kasparov $(A,B)$-modules are homotopic by $(E',D')$, their bounded transformations are homotopic, because the bounded transformation of $(E',D')$ gives a homotopy between these transformed Kasparov modules.

Let us introduce a criterion to be a Kasparov product in the unbounded picture. The following is immediately proved thanks to \cite{Kuc}.

\begin{pro}[\cite{Kuc}]\label{Kucs criterion}
Let $(E_1,D_1)$ be an unbounded $G$-equivariant Kasparov $(A,C)$-module, and $(E_2,D_2)$ be an unbounded $G$-equivariant Kasparov $(C,B)$-module. $x$ and $y$ are the corresponding $KK$-elements of $(E_1,D_1)$ and $(E_2,D_2)$, respectively.
A $G$-equivariant Kasparov $(A,B)$-module $(E_1\otimes E_2,D)$ is a representative of the Kasparov product of $x$ and $y$, if the following conditions are fulfilled:
\begin{itemize}
\item For all $v$ in some dense subset of $AE_1$, the (graded) commutator
$$\bbbra{
\begin{pmatrix} D & 0 \\ 0 & D_2 \end{pmatrix},
\begin{pmatrix} 0 & T_v \\ T_v^* & 0 \end{pmatrix}
}$$
is bounded on $\dom(D\oplus D_2)$, where $T_v(w):=v\otimes w:E_2\to E_1\otimes E_2$.
\item $\dom(D)$ is contained in $\dom(D_1\otimes \id)$.
\item The (graded) commutator $[D_1\otimes \id,D]$ is bounded below.
\end{itemize}

\end{pro}

\begin{proof}
In the bounded picture of $KK_G$, the following holds: $z=x\otimes y$ in $KK_G$ if and only if $z=x\otimes y$ in $KK$ and $z$ is $G$-equivariant. The conditions in the statement guarantee that $(E_1\otimes E_2,D)$ gives a Kasparov product of $x$ and $y$, thanks to \cite{Kuc}. Since we suppose that $(E,D)$ itself is $G$-equivariant, the statement has been proved.

\end{proof}

Let us define the descent homomorphism in the unbounded picture. Note that this homomorphism has already been defined in the bounded picture in \cite{Kas88}. 
For simplicity, we assume the following, from now on. For details on the amenability, see \cite{Pie}.

\begin{asm}
$G$ is amenable and unimodular.
\end{asm}

Thus, we do not have to pay attention to the norm of crossed products or modular functions.

\begin{dfn-pro}
For an unbounded $G$-equivariant Kasparov $(A,B)$-module $(E,D)$, consider $C_c(G,E)$ and the operator $[\widetilde{D}(f)](g):=D[f(g)]$ for $f\in C_c(G,E)$. $C_c(G,E)$ is a pre-Hilbert $C_c(G,B)$-module by the following operations:
$$f*b(g):=\int_Gf(h)h.[b(h^{-1}g)]dh,\ \inpr{f_1}{f_2}{B\rtimes G}(g):= \int_Gh^{-1}\inpr{f_1(h)}{f_2(hg)}{B}dh$$
for $f,f_1,f_2\in C_c(G,E)$ and $b\in C_c(G,B)$.
This module admits a left $C_c(G,A)$-action defined by
$$a*f(g):=\int a(h)h.[f(h^{-1}g)]dh.$$

The completion of this module with respect to the above inner product is denoted by $E\rtimes G$. The pair $(E\rtimes G,\widetilde{D})$ is an unbounded Kasparov $(A\rtimes G,B\rtimes G)$-module, and denoted by $j^G(E,D)$. 

The bounded transformation of the value of the new descent homomorphism coincides with the value of the descent homomorphism of the bounded transformed module: $j_G(\fra{b}(E,D))=\fra{b}(j_G(E,D))$.

\end{dfn-pro}
\begin{proof}
Firstly, we note that $\fra{b}(\widetilde{D})=\widetilde{\fra{b}(D)}$. Therefore, the pair $(E\rtimes G,\fra{b}(\widetilde{D}))$ coincides with $j_G(E,\fra{b}(D))$.

In order to prove that $j_G(E,D)$ is a Kasparov module, we must check two conditions: $(1)$ $[a,\widetilde{D}]$ is bounded for dense $a\in A\rtimes G$; and $(2)$ $a(1+\widetilde{D}^2)$ is compact for all $a\in A\rtimes G$.

For $(1)$, let $\ca{A}:=\bbra{a\in A\mid [a,D]\text{ is bounded}}$, and let $a\in C_c(G,\ca{A})$. We may assume that $a(h)$ is homogeneous and its degree is independent of $h$.
Then, for $e\in C_c(G,\dom(D))$,
\begin{align*}
[a,\widetilde{D}]e(g) &= \int_Ga(h)h.[D(e(h^{-1}g))]dh -(-1)^{\partial a}\int_G D[a(h)h.e(h^{-1}g)]dh \\
&=\int_Ga(h)[\{h(D)-D+D\}h.(e(h^{-1}g)]dh -(-1)^{\partial a}\int_G D[a(h)h.e(h^{-1}g)]dh \\
&=\int_G\bbra{a(h)(h(D)-D)+[a(h),D]}[h.(e(h^{-1}g))]dh.
\end{align*}
The last form is obviously bounded in $e$.

For $(2)$, we note the following computation:
\begin{align*}
(1+\widetilde{D}^2)^{-1} &= 1-\fra{b}(\widetilde{D})^2 \\
&= 1-\widetilde{\fra{b}(D)}^2.
\end{align*}
Since $(E,\fra{b}(D))$ is a $G$-equivariant bounded Kasparov module, $a(1-\fra{b}(\widetilde{D})^2)$ is compact for all $a\in A\rtimes G$.

\end{proof}

\subsection{Exterior tensor product issues}

A large advantage of unbounded Kasparov modules is that the exterior tensor product can be explicitly described at the level of modules. This feature is inherited to the equivariant setting. 
Let $G_1$ and $G_2$ be locally compact, second countable and Hausdorff group, and let $G:=G_1\times G_2$.

\begin{lem}\label{division of the problem 1}
Let $A_1$ and $B_1$ be $G_1$-$C^*$-algebras, and let $A_2$ and $B_2$ be $G_2$-$C^*$-algebras. Then $A_1\otimes_\bb{C} A_2$ and $B_1\otimes_\bb{C} B_2$ are $G$-$C^*$-algebras.

Suppose that an unbounded $G_1$-equivariant Kasparov $(A_1,B_1)$-module $(E_1,D_1)$ and an unbounded $G_2$-equivariant Kasparov $(A_2,B_2)$-module $(E_2,D_2)$ are given. Then, 
$$(E,D):=(E_1\otimes E_2,D_1\otimes \id+\id\otimes D_2)$$
is an unbounded $G$-equivariant Kasparov $(A,B)$-module.
\end{lem}
\begin{proof}
We only have to check the equivariance conditions.

$E$ is obviously $G$-equivariant $(A,B)$-module by the action $(g_1,g_2).(e_1,e_2):=(g_1.e_1,g_2.e_2)$. Moreover, $G$ preserves $\dom(D)$.
Since $g_i(\id)=\id$,
$$(g_1,g_2)(D)-D=[g_1(D_1)-D_1]\otimes \id+\id\otimes [g_2(D_2)-D_2]$$
is a bounded operator. This description also implies that the map $(g_1,g_2)\mapsto [(g_1,g_2)(D)-D]$ is continuous, by the assumption that $(E_i,D_i)$ are equivariant Kasparov modules.
Moreover,
\begin{align*}
&[D,D-(g_1,g_2)(D)]\cdot\frac{D}{1+\lambda+D^2} \\
&\ \ \ =
\bra{[D_1,D_1-g_1(D_1)]\otimes \id+ \id \otimes [D_2,D_2-g_2(D_2)]}\cdot\frac{D_1\otimes \id+\id\otimes D_2}{1+\lambda+D_1^2\otimes \id +\id\otimes D_2^2} \\
&\ \ \ =[D_1,D_1-g_1(D_1)]\cdot
\frac{D_1}{1+D_1^2}\otimes \id\cdot 
\frac{1+D_1^2\otimes \id}{1+\lambda+D_1^2\otimes \id +\id\otimes D_2^2} \\
&\ \ \ \ \ +[D_1,D_1-g_1(D_1)]\otimes  \frac{D_2}{1+D_2^2}
\cdot\frac{1+\id \otimes D_2^2}{1+\lambda+D_1^2\otimes \id +\id\otimes D_2^2} \\
&\ \ \ \ \ \ \ + \id\otimes [D_2,D_2-g_2(D_2)]\cdot
\frac{D_2}{1+D_2^2}\cdot 
\frac{1+\id\otimes D_2^2}{1+\lambda+D_1^2\otimes \id +\id\otimes D_2^2} \\
&\ \ \ \ \ \ \ \ \ - \frac{D_1}{1+D_1^2}\otimes [D_2,D_2-g_2(D_2)]\cdot 
\frac{1+D_1^2\otimes \id}{1+\lambda+D_1^2\otimes \id +\id\otimes D_2^2} .
\end{align*}
The last operator is bounded whose norm is bounded.
\end{proof}

This operation is compatible with the Kasparov product in the following sense.
Let $A_i$, $B_i$ and $C_i$ be $G_i$-$C^*$-algebras for $i=1,2$.

\begin{lem}
Let $(E_i,D_i)$ be unbounded $G_i$-equivariant Kasparov $(A_i,C_i)$-modules, and $(E'_i,D'_i)$ be unbounded $G_i$-equivariant Kasparov $(C_i,B_i)$-modules for $i=1,2$. Suppose that $(E^i,D^i):=(E_i\otimes E_i',D^i)$ satisfy the criterion to be a Kasparov product of $(E_i,D_i)$ and $(E'_i,D'_i)$, respectively in $i$. Then, the exterior product of $(E^1,D^1)$ and $(E^2,D^2)$:
$$(E^1\otimes_{\bb{C}} E^2 ,D^1\otimes \id+\id\otimes D^2)$$
is a Kasparov product of the exterior products $(E_1,D_1)\otimes_\bb{C}(E_1',D_1')$ and $(E_2,D_2)\otimes_\bb{C}(E_2',D_2')$.
\end{lem}
\begin{rmk}
At the level of $KK$-theory, this equality is obvious from the associativity of the Kasparov product, and the commutativity of the exterior tensor products: $(x_1\otimes_{C_1} y_1)\otimes_{\bb{C}} (x_2\otimes_{C_2} y_2)=(x_1\otimes_{\bb{C}}x_2)\otimes_{C_1\otimes C_2}(y_1\otimes_{\bb{C}}y_2)$. However, for our purpose, we need to prove that at the module level. We leave the proof to the reader.
\end{rmk}



The descent homomorphism is also compatible with the formula about the exterior tensor product. It can be proved by the same way in Lemma 3.15 in \cite{T3}.

\begin{lem}
Let $(E_i,D_i)$ be unbounded $G_i$-equivariant Kasparov $(A_i,B_i)$-modules for $i=1,2$, respectively. Let $D$ be $D_1\otimes \id+\id\otimes D_2$.
Then,
$$j_{G_1\times G_2}(E_1\otimes E_2,D_1\otimes \id+\id\otimes D_2)=
j_{G_1}(E_1,D_1)\otimes j_{G_2}(E_2,D_2).$$
\end{lem}

These results will be used in Section 3 to simplify the problem.

\subsection{Clifford algebras and modules}

In this subsection, we describe some algebraic features of Clifford algebras and their modules.

The Clifford algebra is defined as follows: For a Euclidean space $V$, let $\Cl_-(V):=T(V)\otimes\bb{C}/\left< v^2+|v|^2\right>$, where $T(V)$ is the tensor algebra. We will also encounter $\Cl_+$, whose definition is $T(V)\otimes\bb{C}/\left< v^2-|v|^2\right>$.
In fact, these are isomorphic via the homomorphism generated by a non-canonical isomorphism $v\mapsto iv$ for $v\in V$, but it is better to distinguish them. This is because there are various different choices of isomorphisms: $i\id_{V_1}\oplus (-i\id_{V_1^\perp})$ generates an isomorphism between $\Cl_+$ and $\Cl_-$, where $V_1$ is a subspace of $V$.

These two algebras turn out to be $C^*$-algebras.

\begin{lem}
The metric of $V$ induces metrics on $\Cl_\pm(V)$. By the left multiplications $\Cl_\pm(V)\to\End(\Cl_\pm(V))$, and the $C^*$-algebra structures on $\End(\Cl_\pm(V))$, we find that $\Cl_\pm(V)$ are $C^*$-algebras. The involution on $\Cl_\pm(V)$ are given by $v^*=\pm v$ for $v\in V$, respectively.
\end{lem}
\begin{proof}
We deal with only $\Cl_-(V)$. Let $v\in V$ be a unit vector, and consider a vector $v{\bf x}+{\bf y}:=vv_1\cdots v_k+w_1\cdots w_l$, where $v\perp v_i$, $v\perp w_j$, $v_i\perp v_{i'}$ for $i\neq i'$, and $w_j\perp w_{j'}$ for $j\neq j'$. Similarly for $v{\bf x}'+{\bf y}'$. Then,
\begin{align*}
\inpr{v\cdot (v{\bf x}+{\bf y})}{v{\bf x}'+{\bf y}'}{} &= \inpr{-{\bf x}+v{\bf y}}{v{\bf x}'+{\bf y}'}{} \\
&= -\inpr{{\bf x}}{{\bf y}'}{}+\inpr{v{\bf y}}{v{\bf x}'}{} \\
&=-\inpr{v{\bf x}}{v{\bf y}'}{}+\inpr{{\bf y}}{{\bf x}'}{} \\
&=-\inpr{v{\bf x}+{\bf y}}{v\cdot (v{\bf x}'+{\bf y}')}{}.
\end{align*}
\end{proof}

Thanks to this result, we find that $C_0(X,\Cl_+(T^*X))$ is a $C^*$-algebra. It is denoted by $Cl_\tau(X)$.

In this paper, we will use $\Cl_-$ for the Clifford multiplication, and hence the Dirac operator is formally self-adjoint. Then, $\Cl_+$ appears when we describe the Clifford symbol class. This is because a left $\Cl_-$-module automatically admits a right Hilbert $\Cl_+$-module structure, as proved below. 

A $\bb{Z}_2$-graded Hermite vector space $S$ is called a {\bf Clifford module of $V$}, if it is equipped with an even $*$-homomorphism $c:\Cl_-(V)\to \End(S)$. In other words, it is a $\bb{Z}_2$-graded Hermite vector space equipped with a linear map $c:V\to \End(S)$ such that $c(v)$ is a skew-adjoint and odd operator for each $v\in V$, and satisfies that $c(v)c(w)+c(w)c(v)=-2\inpr{v}{w}{V}$. A Clifford module of $V$ is called a {\bf Spinor} if it is an irreducible representation of $\Cl_-(V)$.  It is known that an even-dimensional vector space $V$ has two Spinors up to equivalence. Fix one of them, and the fixed Spinor is denoted by $(S,c)$. It is known that the Clifford multiplication $c$ gives an isomorphism as $\bb{Z}_2$-graded $C^*$-algebras $\Cl_-(V)\cong \End(S)$.

We can define an irreducible representation of $\Cl_+(V)$ from $\Cl_-(V)\circlearrowright S$ as follows.

\begin{lem}
The dual space $S^*$ admits a structure of a $*$-representation space of $\Cl_+(V)$, defined by
$$c^*(v)\cdot f:=(-1)^{\partial f}f\circ c(v)$$
for $f\in S^*$ and $v\in V$.
\end{lem}

\begin{rmk}
Let $\tau:\Cl_-(V)\to\Cl_+(V)$ be the linear extension of $\tau(v_1\cdots v_k):=(-1)^{\sum_{j=1}^{k-1}j}v_k\cdots v_1$ and $\tau(1)=1$. This is well-defined: 
$\tau(v^2+\|v\|^2)=-v^2+\|v\|^2=0$. Moreover, this is an anti-homomorphism $\tau(\alpha \beta)=(-1)^{\partial \alpha\cdot \partial \beta}\tau(\beta)\tau(\alpha)$. It can be proved as follows: Let $\alpha:=a_1\cdots a_k$, $\beta:=b_1\cdots b_l$, and then $\tau(\alpha\beta)=(-1)^{\sum_i^{k+l-1}i}b_l\cdots b_1a_k\cdots a_1=(-1)^{lk}\tau(\beta)\tau(\alpha)$.
Using $\tau$, we can rewrite the action as $c^*(\alpha)\cdot f=(-1)^{\partial f\cdot \partial \alpha}f\circ c(\tau^{-1}(\alpha))$ for $f\in S^*$ and $\alpha\in \Cl_+(V)$.

\end{rmk}

\begin{proof}
$c^*$ induces a map from the tensor algebra $T(V)\to \End(S^*)$. We check that $c^*(v)c^*(w)+c^*(w)c^*(v)=2\inpr{v}{w}{V}$. Let $s\in S$.
\begin{align*}
\bigl[\bigl(c^*(v)c^*(w)+c^*(w)c^*(v)\bigr)f\bigr](s) &= c^*(v)[c^*(w)f](s) + c^*(w)[c^*(v)f](s) \\
&= (-1)^{\partial f+1}\bbbra{c^*(w)f(c(v)s) + c^*(v)f(c(w)s)} \\
&= -f(c(w)c(v)s+c(v)c(w)s) \\
&= 2\inpr{v}{w}{V}f(s).
\end{align*}
Therefore the map $c^*$ from $\Cl_+(V)$ is defined. Let us verify that $c^*$ preserves the involution $*$. Note that $c^*(v)$ is given by ${}^tc(v)\circ \epsilon_{S^*}$.
$$[c^*(v)]^*=[{}^t\bbra{c(v)}\circ \epsilon_{S^*}]^*=\epsilon_{S^*}^*\circ ({}^tc(v))^*$$
$$=\epsilon_{S^*}\circ {}^t(c(v)^*)=-\epsilon_{S^*}\circ{}^tc(v)={}^tc(v)\circ\epsilon_{S^*}=c^*(v),$$
where we used the fact that $c(v)$ is an odd operator, and the transpose commutes with the adjoint.
\end{proof}

Thanks to this result, $c^*:\Cl_+(V)\to \End(S^*)$ is a $*$-isomorphism. Using this isomorphism, we define a Hilbert $\Cl_+(V)$-module structure on $S$.

\begin{lem}\label{lem spinor is hilb. cl. module}
A left $\Cl_-(V)$-module $S$ equipped with $c:V\to \End(S)$, admits a right Hilbert $\Cl_+(V)$-module structure given by the followings: By regarding $S$ as $S^{**}$, we define
\begin{itemize}
\item For $s\in S$ and $v\in V$, $s\cdot v:= s\circ c^*(v)$.
\item $\inpr{s_1}{s_2}{\Cl_+}:=s_1^*\otimes s_2 \in \End(S^*)\cong \Cl_+(V)\otimes\bb{C}$.
\end{itemize}
\end{lem}
\begin{proof}
Non-trivial thing is the compatibility of these two operations: $\inpr{s_1}{s_2\cdot v}{\Cl_+}=\inpr{s_1}{s_2}{\Cl_+}\cdot v$. Let us check that. Let $f\in S^*$.
\begin{align*}
\inpr{s_1}{s_2\cdot v}{\Cl_+} (f) &= s_1^*\otimes (s_2\circ c^* (v)) (f)\\
&= s_1^*\otimes s_2(c^*(v)f) \\
&= \bra{\inpr{s_1}{s_2}{\Cl_+}\cdot v}( f).
\end{align*}
\end{proof}

This result is important to describe the Clifford symbol element.

\subsection{Kasparov index theorem}

In this subsection, we review the Kasparov index theorem \cite{Kas84,Kas15}. After defining the ingredients, we outline the proof, in the unbounded picture of equivariant Kasparov modules.

\begin{prob}\label{problem for Kasparov}
Study the index theory in the following situation: $X$ is an even-dimensional complete Riemannian $Spin^c$-manifold; $G$ is a locally compact, second countable and Hausdorff group;
$G$ acts on $X$ isometrically, properly and cocompactly; $(S,c)$ is a $G$-equivariant Spinor bundle of $X$; $E$ is a $\bb{Z}_2$-graded $G$-equivariant Hermite vector bundle on $X$; and $D$ is a $G$-equivariant Dirac operator acting on $C^\infty(X,E\otimes S)$.
\end{prob}

The following is the typical example: For a compact Riemannian $Spin^c$-manifold $M$, the universal cover $X:=\widetilde{M}$ and the fundamental group $G:=\pi_1(M)$ satisfy the above condition. The Spinor and the Dirac operator on $X$ can be defined by the pullback of ones on $M$.

The following is independent of $D$, $S$ and $E$. In this sense, they are universal.

\begin{dfn}\label{dfn of Dirac element for X}
$(1)$ For $\xi\in \Cl_+(V)$ and $e_k\in V$, let $\widehat{e_k}\xi:=(-1)^{\partial \xi}\xi\cdot e_k$. Note that $(\widehat{e_k})^2=-\|e_k\|^2\id$.
Define a differential operator $\partial$ on $L^2(X,\Cl_+(T^*X))$ by
$$\partial:=\sum_k \widehat{e_k}\frac{\partial}{\partial x_k}$$
at $x\in X$, where $\{x_1,x_2,\cdots,x_n\}$ is a coordinate around $x\in X$ such that $\bbra{\frac{\partial}{\partial x_1}(x),\frac{\partial}{\partial x_2}(x),\cdots, \frac{\partial}{\partial x_n}(x)}$ is an orthonomal basis, and $e_i:=dx_i$. This operator is independent of the choice of such coordinate. The pair $(L^2(X,\Cl_+(T^*X)),\partial)$ defines a $G$-equivariant $(Cl_\tau(X),\bb{C})$-module.
Let $[d_X]:=[(L^2(X,\Cl_+(T^*X)),\partial)]\in KK_G(Cl_\tau(X),\bb{C})$. It is called the {\bf Dirac element} (Definition 2.2 in \cite{Kas15}).

$(2)$ At each $x\in X$, let $H_x:=L^2(T^*_xX,\Cl_+(T^*_xX))$. Consider the pair $(H_x,\partial_x)$, for 
$$\partial_x:=i\sum_k e_k\cdot\frac{\partial}{\partial x_k},$$
where $\{e_k\}$ is an orthonormal basis of $T_x^*X$, and $e_k\cdot $ means the multiplication by $e_k$ from the left side.
Consider the continuous family
$$[d_\xi]:= \bbbra{\bra{C_0(X,\cup H_x),\{\partial_x\}_{x\in X}}}\in KK_G(C_0(T^*X),Cl_\tau(X)).$$
For $f\in Cl_\tau(X)$, its action on $H_x$ is given by the right multiplication of $f(x)\in \Cl_+(T^*_xX)$. The inner product is given by $\inpr{f_1\otimes v_1}{f_2\otimes v_2}{Cl_\tau(X)}:=\inpr{f_1}{f_2}{L^2}\cdot v_1^*v_2$. Note that $\inpr{f_1}{f_2}{L^2}$ is a scalar-valued continuous function.
This element is called the {\bf fiberwise Dirac element} (Definition 2.5 in \cite{Kas15}).

$(3)$ A cut-off function is a compactly supported continuous function $c:X\to \bb{R}_{\geq 0}$ satisfying that $\int_Gc(g^{-1}.x)dg=1$ for all $x\in X$. It defines a projection $[c](g,x):=\sqrt{c(x)c(g^{-1}.x)}$ in $C_0(X)\rtimes G$. The corresponding $KK$-element is also denoted by $[c]\in KK(\bb{C},C_0(X)\rtimes G)$. It is called the {\bf Mishchenko line bundle}.
\end{dfn}

Note that $[d_\xi]$ gives a $KK_G$-equivalence between $C_0(T^*X)$ and $Cl_\tau(X)$. 

From the given data $E$, $S$ and $D$, we can define several $KK_G$-elements.

\begin{dfn}
$(1)$ $[D]:=[(L^2(X,E\otimes S),D)]\in KK_G(C_0(X),\bb{C})$ is called the {\bf index element} (Lemma 3.7 in \cite{Kas15}).

$(2)$ Let $c_E$ be the Clifford multiplication on $E\otimes S$, and let $p:T^*X\to X$ be the canonical projection. Let 
$[\sigma_D]:=[(C_0(T^*X,p^*(E\otimes S)),ic_E)]\in KK_G(C_0(X),C_0(T^*X))$. It is called the {\bf symbol element}.

$(3)$ Let $[\sigma_D^{Cl}]:= [\sigma_D]\otimes_{C_0(T^*X)}[d_\xi]\in KK_G(C_0(X),Cl_\tau(X))$. It is called the {\bf Clifford symbol element}.
\end{dfn}

We would like to define the analytic index for this case. The definition is based on the following observation: On a separable Hilbert space $H$, consider the space of Fredholm operators on $H$ equipped with the norm topology, ${\rm Fred}(H)$. The set of connected components of ${\rm Fred}(H)$ is isomorphic to $\bb{Z}$, and the isomorphism is given by the Fredholm index. In this sense, {\it the Fredholm index of an operator is the homotopy class itself of the operator}. Therefore, the homotopy class of the following Hilbert module should be the analytic index, if it is a Kasparov module.


\begin{dfn}\label{dfn of the analytic index}
On $C_c(X,E\otimes S)$, define the right action of $C_c(G)$ by
$$s*b(x):=\int_Gg(s)(x)\cdot b(g^{-1})dg$$
for $b\in C_c(G)$ and $s\in C_c(X,E\otimes S)$, and the inner product by
$$\inpr{s_1}{s_2}{\bb{C}\rtimes G}(g):=\int_X\inpr{s_1(x)}{g(s_2)(x)}{(E\otimes S)_x}dx$$
for $s_1,s_2\in C_c(X,E\otimes S)$. We obtain a Hilbert $\bb{C}\rtimes G$-module by the completion of this pre-Hilbert module. The completion of $D$ is an unbounded, densely defined, regular, odd and self-adjoint operator on the obtained Hilbert $\bb{C}\rtimes G$-module. The pair of the Hilbert $\bb{C}\rtimes G$-module obtained by the  completion of $C_c(X,E\otimes S)$ and the operator $D$ is called the {\bf analytic index of $D$} and denoted by $\ind (D)$.

\end{dfn}

\begin{rmks}
$(1)$ One must notice that we used the {\bf compactly supported} sections to define the index. It means that, in order to generalize this procedure to infinite-dimensional manifolds, we need some another trick to find some ``nice'' subspace. In fact, we have used an algebraic technique in \cite{T2} to specify the dense subspace.

$(2)$ We used the fact that $D$ is actually $G$-equivariant to define the analytic index. For the general case ($D$ is not $G$-equivariant, but the symbol is still equivariant), we need an averaging procedure. For the details, see Proposition 5.1 in \cite{Kas15}.
\end{rmks}

We call the following the Kasparov index theorem \cite{Kas84,Kas15}.

\begin{thm}
$(1)$ $\ind(D)$ is an unbounded Kasparov $(\bb{C},\bb{C}\rtimes G)$-module.

$(2)$ The analytic index is completely determined by the $K$-homology element $[D]\in KK_G(C_0(X),\bb{C})$, by the formula $\ind(D)=[c]\otimes_{C_0(X)\rtimes G} j^G([D])$. 

$(3)$ On the other hand, $[D]$ is completely determined by the topological data $[\sigma_D^{Cl}]$, by the formula $[D]=[\sigma_D^{Cl}]\otimes [d_X]$.

Combining all of them, we find that $\ind(D)$ is completely determined by the homotopy class of the symbol, which is a generalization of the Atiyah-Singer index theorem.
\end{thm}

\begin{rmk}
In this paper, the correspondence $[D]\mapsto [c]\otimes j_G([D])$ is called the assembly map, and denoted by $\mu_G$.

\end{rmk}

In order to prove this result in our language without additional efforts, it is convenient to suppose the following differential geometrical condition. See also \cite{Roe}, for example.


\begin{asm}
Let $\nabla^S$ be a metric connection on $S$ and let $V,W$ be smooth vector fields on $X$. We identify vector fields with $1$-forms by the metric. We impose that $\nabla^S$ satisfies the following condition\footnote{Briefly speaking, $c$ is ``flat''.}: For $s\in C^\infty(X,S)$, 
$$\nabla^S_V(c(W)s)=c(\nabla^{\rm LC}_V(W))(s)+c(W)(\nabla^S_V(s)),$$
where $\nabla^{\rm LC}$ is the Levi-Civita connection.

Take a metric connection on $E$, denoted by $\nabla^E$. Then, a metric connection on $E\otimes S$ is induced, denoted by $\nabla^{E\otimes S}$. We impose that our Dirac operator must be of the following form:
$$D=\sum_k \id_E\otimes c(e_k)\otimes \nabla^{E\otimes S}_{e_k}+h,$$
where $h$ is a {\bf bounded} section of $\End(E\otimes S)$, and commutes\footnote{in the graded sense} with the Clifford multiplication.
\end{asm}

Let us outline the proof. For the first part, $D$ has a well-defined index, we need to check the criterion so that the module defined in Definition \ref{dfn of the analytic index} is a Kasparov $(\bb{C},\bb{C}\rtimes G)$-module. For this purpose, we need to check that $D$ is regular and $(1+D^2)^{-1}$ is $\bb{C}\rtimes G$-compact. They are proved in Theorem 5.8 in \cite{Kas15}. 

For the equality $\ind(D)=[c]\otimes j^G([D])$, note that $[c]$ is given by $[([c]* \bbbra{C_0(X)\rtimes G},0)]$ as a Kasparov module. Let us recall the map $q: C_c(G,C_c(X,E\otimes S))\to C_c(X,E\otimes S)$ defined in Section 5 in \cite{Kas15}.
Any element $[c]*f\in [c]* \bbbra{C_0(X)\rtimes G}$ can be rewritten as $[c]*([c]*f)$, and hence we find that $([c]*f)\otimes F$ can be identified with $q(f*F)\in \overline{C_c(X,E\otimes S)}$. Therefore the isomorphism holds at the module level.
It suffices to check the criterion to be the Kasparov product. In fact, the second and the third conditions are obvious. For the first one, one can verify that the commutator contains no differentials, thanks to the Leibniz rule.

For the equality $[D]=[\sigma_D^{Cl}]\otimes [d_X]$, which is a main interest of the present paper, we use the following computation. This is a more explicit version of Proposition 3.10 in \cite{Kas15}

\begin{pro}
$[\sigma_D^{Cl}]=[(C_0(X,E\otimes S),0)]\in KK_G(C_0(X),Cl_\tau(X))$. The Hilbert $Cl_\tau(X)$-module structure is given by the family version of Lemma \ref{lem spinor is hilb. cl. module}.
\end{pro}

\begin{proof}
The module for the Kasparov product is given by the space of continuous sections of the Hilbert bundle
$$\bigcup_{x\in X} E_x\otimes S_x\otimes L^2(T^*_xX)\otimes \Cl_+(T^*_xX)\cong 
\bigcup_{x\in X} E_x\otimes S_x\otimes L^2(T^*_xX)\otimes S_x^*\otimes S_x.$$
Let $\ca{H}_x$ be the fiber $E_x\otimes S_x\otimes L^2(T^*_xX)\otimes S_x^*\otimes S_x$. Consider the operator $\partial'_x$ on $\ca{H}_x$ defined by
$$\partial'_x:= \id_E\otimes
ic(\xi)\otimes \id\otimes \id +i\sum\id\otimes \id\otimes \frac{\partial}{\partial \xi_k}\otimes e_k\cdot.$$
$\xi\in T_x^*X$.
One can prove that $(C_0(X,\cup_{x\in X} \ca{H}_x),\{\partial_x'\}_{x\in X})$ is a $G$-equivariant unbounded $(C_0(X),Cl_\tau(X))$-module. 
In fact, $\{\partial'_x\}_{x\in X}$ is a $Cl_\tau(X)$-module homomorphism\footnote{Recall that $\Cl_+(T_x^*X)$ acts on $\ca{H}_x$ by the right multiplication.}, commutes with the $C_0(X)$-action, and $\{(1+(\partial_x')^2)^{-1}\}_{x\in X}$ is a family of compact operators. That means the pair is a Kasparov module. Moreover, the operator is actually $G$-equivariant, and hence the pair gives an equivariant cycle.

One can check that the above class is a Kasparov product of $[\sigma_D]$ and $[d_\xi]$, by a simple computation using the Leibniz rule.

We would like to compute $\{\ker(\partial_x')\}_{x\in X}$. For this aim, we may compute the square of the operator. In the following, $c^*(e_k)$ denotes the Clifford multiplication from the left: $e_k\cdot$. This notation is compatible with the original $c$, because $\Cl_+$ is isomorphic to $S^*\otimes S$ as $(\Cl_+,\Cl_+)$-bimodules.

\begin{align*}
(\partial_x')^2
& = \id\otimes \id\otimes \sum_k\xi_k^2\otimes \id\otimes \id -
\sum_k\id\otimes  c(e_k)\otimes \xi_k\cdot\frac{\partial}{\partial \xi_k}\otimes  c^*(e_k)\otimes \id \\
&\ \ \ +\sum_k\id\otimes  c(e_k)\otimes \frac{\partial}{\partial \xi_k}\cdot \xi_k\otimes  c^*(e_k)\otimes \id +\id\otimes \id\otimes \sum_k\bra{-\frac{\partial^2}{\partial \xi_k^2}}\otimes \id\otimes \id\\
&= \id\otimes \id\otimes \sum_k\bra{-\frac{\partial^2}{\partial \xi_k^2}+\xi_k^2}\otimes \id\otimes \id+\sum_k\id\otimes c(e_k)\otimes \id\otimes  c^*(e_k)\otimes \id.
\end{align*}

Under the natural identification $S_x\otimes S_x^*\cong \Cl_-(T_x^*X)$, the element $c(e_k)\otimes c^*(e_k)$ defines a map $\Cl_-(T_x^*X)\ni X\mapsto e_k\cdot (-1)^{\partial X} X\cdot e_k$. In fact,
$$c(e_k)\otimes  c^*(e_k)(s\otimes f)=(-1)^{\partial s}\bbbra{
c(e_k)(s)\otimes c^*(e_k)f} = (-1)^{\partial s+\partial f}\bbbra{e_k\cdot s\otimes f\cdot e_k}.$$
Then, $1\in \Cl_-(T_x^*X)$ is the lowest weight vector of $\sum_kc(e_k)\otimes c^*(e_k)$, whose weight is $-n$. Such vector is unique up to scalar multiplication. On the other hand, $\sum_k\bra{-\frac{\partial^2}{\partial \xi_k^2}+\xi_k^2}$ has the one-dimensional lowest weight space whose weight is $n$. Therefore the family of the kernel of $\partial'_x$'s can be naturally identified with $E\otimes S$.
\end{proof}

\begin{pro}
$[D]=[\sigma_D^{Cl}]\otimes[d_X]$.
\end{pro}
\begin{proof}
Since $S\otimes_{Cl}(S^*\otimes S)\cong S$, $L^2(X,E\otimes S)\cong C_0(X,E\otimes S)\otimes_{Cl_\tau(X)} L^2(X,\Cl_+(T^*X))$. Once we notice this isomorphism, we only have to check three conditions in Proposition \ref{Kucs criterion}. We leave it to the reader.
\end{proof}

Thanks to Section 2.2, we obtain a product formula, in our language. See Proposition 2.11 in \cite{T2} also.

\begin{pro}\label{product formula}
Suppose that $X_i,G_i,S_i,E_i$, and $D_i$ satisfy the conditions in Problem \ref{problem for Kasparov}, for $i=1,2$. Let $X:=X_1\times X_2$, $G:=G_1\times G_2$, $S:=S_1\boxtimes S_2$, $E:=E_1\boxtimes E_2$ and $D:=D_1\otimes \id+\id\otimes D_2$. Then, at the module level,
\begin{align*}
[D]&=[D_1]\otimes [D_2]; \\
[\sigma_D^{Cl}]&= [\sigma_{D_1}^{Cl}]\otimes [\sigma_{D_2}^{Cl}]; \\
[d_X]&= [d_{X_1}]\otimes [d_{X_2}] ; \\
[c_{X}]&= [c_{X_2}]\otimes [c_{X_2}].
\end{align*}
Moreover, $j_G([D])=j_{G_1}([D_1])\otimes j_{G_2}([D_2])$ and $\ind(D)=\ind(D_1)\otimes \ind(D_2)$.
\end{pro}

Consequently, the four equalities $[D_i]=[\sigma_{D_i}^{Cl}]\otimes [d_{X_i}]$ and $\ind(D_i)=[c_{X_i}]\otimes j_{G_i}([D_i])$ for $i=1,2$, imply the index theorem on $X$: $[D]=[\sigma_{D}^{Cl}]\otimes [d_{X}]$ and $\ind(D)=[c_{X}]\otimes j_{G}([D])$. We ``use'' this property to define the objects on the total space in Section 3.

\subsection{A reformulation of the Kasparov index theorem}

As we have pointed out in Introduction, the $C_0$-algebra for infinite-dimensional space is trivial. However, another noncommutative algebra which is related with $C_0$, can be generalized to infinite-dimensional situation. In fact, the HKT algebra is  such an algebra.

In order to use the HKT algebra, we reformulate the Kasparov index theorem. We will avoid using the $C_0$-algebra. We recall the following $KK_G$-equivalence.

\begin{lem}\label{C0 is KK-equiv to Cltau}
For an even-dimensional and $ G $-equivariantly $Spin^c$-manifold $X$, fix a $G$-equivariant Spinor bundle $S$. By definition, $S_x$ is a left $\Cl_-(T^*_xX)$-module at each $x\in X$. At the same time, it is a right $\Cl_+(T^*_xX)$-module. Considering the construction Lemma \ref{lem spinor is hilb. cl. module} fiberwisely, we obtain a right Hilbert $Cl_\tau(X)$-module $C_0(X,S)$ equipped with a left $C_0(X)$-action. It determines a $KK_G$-element 
$$[S]:=[(C_0(X,S),0)]\in KK_G(C_0(X),Cl_\tau(X)).$$
Similarly, we can define $[S^*]=[(C_0(X,S^*),0)]\in KK_G(Cl_\tau(X),C_0(X))$.

Then, $[S]\otimes_{Cl_\tau(X)}[S^*]= 1_{C_0(X)}$ and $[S^*]\otimes_{C_0(X)}[S]= 1_{Cl_\tau(X)}$. Consequently, $C_0(X)$ is $KK_G$-equivalent to $Cl_\tau(X)$.
\end{lem}

Recall one more operation. Suppose that $C$ is a $\bb{Z}_2$-graded separable $C^*$-algebra. For simplicity, we assume that $C$ is nuclear.

\begin{dfn}
The homomorphism $\tau_C:KK_ G (A,B)\to KK_ G (C\otimes A,C\otimes B)$ is given by
$$(E,D)\mapsto (C\otimes E,\id\otimes D).$$
\end{dfn}

In fact, $C$ is always $\scr{S}$ in this paper, where $\scr{S}$ is $C_0(X)$ equipped with the grading homomorphism $\epsilon$ given by $\epsilon(f)(x):=f(-x)$. Let $\SC(X)$ be the graded tensor product $\scr{S}\otimes Cl_\tau(X)$. The same symbol denotes the HKT algebra.

Let us reformulate the Kasparov index theorem by using these operations. Put
$$[\widetilde{D}]:=\tau_\scr{S}\bra{[S^*]\otimes [D]}\in KK_ G (\SC(X),\scr{S});$$
$$[\widetilde{\sigma_D^{Cl}}]:=\tau_\scr{S}\bra{[S^*]\otimes [\sigma_D^{Cl}]}\in KK_ G (\SC(X),\SC(X));$$
$$[\widetilde{d_X}]:=\tau_\scr{S}\bra{[d_X]}\in KK_ G (\SC(X),\scr{S});$$
$$[\widetilde{c}]:=\tau_\scr{S}\bra{[c]\otimes j^G ([S])}\in KK(\scr{S},\SC(X)\rtimes G).$$
Using these reformulated $KK$-elements, we can reformulate the index theorem.

\begin{pro}
The following two equalities hold:
$$\tau_\scr{S}(\ind(D))=[\widetilde{c}]\otimes_{\SC(X)\rtimes G }j^G ([\widetilde{D}]),$$
$$[\widetilde{D}]=[\widetilde{\sigma_D^{Cl}}]\otimes_{\SC(X)}[\widetilde{d_X}].$$
As a result, $\tau_\scr{S}(\ind(D))$ is determined by the topological data $[\widetilde{\sigma_D^{Cl}}]$.
\end{pro}
\begin{proof}
The proof is done by several formal computations in $KK_G$-theory.
\begin{align*}
[\widetilde{\sigma_D^{Cl}}]\otimes_{\SC(X)}[\widetilde{d_X}]
&=\tau_\scr{S}\bra{[S^*]\otimes_{C_0(X)} [\sigma_D^{Cl}]}\otimes_{\SC(X)} \tau_\scr{S}\bra{[d_X]} \\
&=\tau_\scr{S}\bra{[S^*]\otimes_{C_0(X)} [\sigma_D^{Cl}]\otimes_{Cl_\tau(X)} [d_X]} \\
&=\tau_\scr{S}\bra{[S^*]\otimes_{C_0(X)} [D]} \\
&= [\widetilde{D}].
\end{align*}
\begin{align*}
[\widetilde{c}]\otimes_{\SC(X)\rtimes G }j^G ([\widetilde{D}]) &=
\tau_\scr{S}\bra{[c]\otimes_{C_0(X)\rtimes  G } j^G ([S])}\otimes_{\SC(X)\rtimes G } j^G  \bra{\tau_\scr{S}\bra{[S^*]\otimes_{C_0(X)} [D]}} \\
&= \tau_\scr{S}\bra{[c]\otimes_{C_0(X)\rtimes  G } j^G ([S])\otimes_{Cl_\tau(X)\rtimes G}j^G  \bra{[S^*]\otimes_{C_0(X)} [D]}} \\
&=\tau_\scr{S}\bra{[c]\otimes_{C_0(X)\rtimes  G } j^G \bbra{[S]\otimes_{Cl_\tau(X)}[S^*]\otimes_{C_0(X)} [D]}} \\
&= \tau_\scr{S}\bra{[c]\otimes_{C_0(X)\rtimes  G } j^G \bra{[D]}} \\
&= \tau_\scr{S}(\ind(D)).
\end{align*}
\end{proof}

We can easily describe the modified Kasparov modules, by using $S^*\otimes S\cong \End(S^*)\cong \Cl_+$.

\begin{lem}\label{modified Kas modules}
The algebra bundle $\Cl_+(T^*X)$ acts on itself from the both sides. By this action,
$$[S^*]\otimes [D]=[(L^2(X,E\otimes \Cl_+(T^*X)),\id_{S^*}\otimes D)]\in KK_G(Cl_\tau(X),\bb{C});$$
$$[S^*]\otimes [\sigma_D^{Cl}]=[(C_0(X,E\otimes \Cl_+(T^*X)),0)]\in KK_G(Cl_\tau(X),Cl_\tau(X)).$$

\end{lem}

Then, the reformulated $KK$-elements $[\widetilde{D}]$ and $[\widetilde{\sigma_D^{Cl}}]$ can be easily described. Our construction in Section 5 is modeled on these formulas.

\subsection{Twisted equivariant matters}

In our problem \ref{Main problem}, a $U(1)$-central extension appears. In the finite-dimensional setting, it is possible to regard something $\tau$-twisted $G$-equivariant, as something $G^\tau$-equivariant satisfying a certain condition. On the other hand, in the infinite-dimensional setting, the situation is completely different: We have defined a substitute of the $\tau$-twisted group $C^*$-algebra of $LT$, but it seems to be too difficult to construct a substitute of untwisted one which should be isomorphic to ``$C_0(\widehat{LT})$'' because of the Pontryagin duality. Therefore, we must study $\tau$-twisted equivariant theory without making reference to $G^\tau$-equivariant one.

Let us recall the twisted equivariant $KK$-theory for special cases. For details, consult with \cite{T2}.
Let 
$$1\to U(1) \xrightarrow{i} G^\tau\xrightarrow{p} G\to 1$$
be a $U(1)$-central extension of $G$. The homomorphisms are supposed to be smooth. Let $A$ and $B$ be $G$-$C^*$-algebras. Through $p$, we can regard $A$ and $B$ as $G^\tau$-algebras.

\begin{dfn}
A $\tau$-twisted $G$-action on a vector space, a Hilbert module or a vector bundle, is a $G^\tau$-action $\rho$ satisfying that $\rho(i(z))=z\id$ for any $z\in U(1)$.

\end{dfn}

\begin{dfn}
Let $k\in \bb{Z}$.
A {\bf $k\tau$-twisted $G$-equivariant Kasparov $(A,B)$-module} is a $G^\tau$-equivariant Kasparov module $(E,F)$ such that $i(z)\in i(U(1))\subseteq G^\tau$ acts on $E$ as $z^k\id_E$. The set of homotopy classes of such Kasparov modules is an abelian group, and denoted by $KK_G^{k\tau}(A,B)$. The case $k=1$ is standard: We always assume that $k=1$ by replacing $\tau$ with its tensor power $k\tau$. This is the reason why we put $KK_G^{\tau}(A,B):=KK_G^{1\tau}(A,B)$.

For any $G^\tau$-equivariant Kasparov module $(E,F)$, we can consider the averaging procedure with respect to the $U(1)$-action; we may assume that $F$ is actually $U(1)$-equivariant. Decompose $E$ with respect to the weight of the $U(1)$-action as $E=\prod_n E_n$, and then $F$ preserves this decomposition. Put $F=\prod_n F_n$.
Consequently, $KK_{G^\tau}(A,B)$ is decomposed as $\bigoplus KK_G^{k\tau}(A,B)$.

One can define the unbounded picture version of it in the obvious way. We will use this picture in the following.
\end{dfn}

Needless to say, $KK_{G^\tau}$ has the Kasparov product: $KK_{G^\tau}(A,C) \times KK_{G^\tau}(C,B)\to KK_{G^\tau}(A,B)$. Then, how is the restriction to $KK_G^{k\tau}$'s? The answer is the following. See \cite{T2} also.

\begin{lem}
The Kasparov product of $x\in KK_G^{k\tau}(A,C)$ and $y\in KK_G^{l\tau}(C,B)$ takes value in $KK_G^{(k+l)\tau}(A,B)$. In particular, $KK^\tau_G(A,C)\otimes KK_G(C,B)\to KK^\tau_G(A,B)$ is defined.
\end{lem}

We can define the partial descent homomorphism as follows. For this aim, we need to introduce the twisted crossed products of algebras and modules.

\begin{dfn}[Definition 2.13 in \cite{T2}]
For a $G$-$C^*$-algbera $A$, let $A\rtimes_{k\tau}G$ be the completion of the subset $\bbra{a\in C_c(G^\tau, A)\mid a(zg)=z^ka(g)}$ in the $C^*$-algebra $A\rtimes G^\tau$. It is a direct summand as the $C^*$-algebras. Similarly, for a $\tau$-twisted $G$-equivariant $B$-Hilbert module $E$, let $E\rtimes_{k\tau}G$ be the completion of the subset $\bbra{e\in C_c(G^\tau, E)\mid e(zg)=z^ke(g)}$ in $E\rtimes G^\tau$ with respect to the $B\rtimes G^\tau$-valued inner product.
\end{dfn}

Let $(E,D)$ be a $k\tau$-twisted $G$-equivariant Kasparov $(A,B)$-module.
As proved in Lemma 2.16 in \cite{T2}, 
\begin{itemize}
\item $(A\rtimes_{m\tau}G)\cdot (E\rtimes_{n\tau}G)=0$ unless $m=n-k$.
\item $(E\rtimes_{n\tau}G)\cdot (B\rtimes_{m\tau}G)=0$ unless $m=n$.
\item For $e_1,e_2\in E\rtimes_{n\tau}G\subseteq E\rtimes G^\tau$, $\inpr{e_1}{e_2}{B\rtimes G^\tau}\in B\rtimes_{n\tau} G$.
\item $E\rtimes_{n\tau}G$ is orthogonal to $E\rtimes_{n'\tau}G$ if $n\neq n'$.
\end{itemize}

Therefore, the pair $(E\rtimes_{n\tau}G,\widetilde{D}|_{E\rtimes_{n\tau}G})$ turns out to be a Kasparov $(A\rtimes_{(n-k)\tau}G, B\rtimes_{n\tau}G)$-module. The following is actually defined at the level of modules.

\begin{dfn}
The {\bf partial descent homomorphism} is the correspondence
$$j_G^{n\tau}:KK_G^{k\tau}(A,B)\ni [(E,D)]\mapsto [(E\rtimes_{n\tau}G,\widetilde{D}|_{E\rtimes_{n\tau}G})]\in KK(A\rtimes_{(n-k)\tau}G, B\rtimes_{n\tau}G).$$
$j_G^\tau$ denotes $j_G^{1\tau}$.
\end{dfn}

The name ``partial descent homomorphism'' comes from the fact that $j_{G^\tau}=\sum_n j_G^{n\tau}$.

Let us describe the partial assembly map. Recall that $X$ is a complete $Spin^c$-manifold, and $G$ acts on $X$ isometrically, properly and cocompactly.
For this time, we assume that $E\otimes S$ is a $\tau$-twisted $G$-equivariant Clifford module bundle, and $D$ is a $G^\tau$-equivariant Dirac operator on $E\otimes S$. Then, $(L^2(X,E\otimes S),D)$ is a $\tau$-twisted $G$-equivariant $(C_0(X),\bb{C})$-module, and the analytic index is an element of $KK(\bb{C},\bb{C}\rtimes_\tau G)$. 

The Mishchenko line bundle $[c]$ is an element of $KK(\bb{C},C_0(X)\rtimes G)$, but it can be regarded as an element of $KK(\bb{C},C_0(X)\rtimes G^\tau)$ at the same time, by the projection onto the direct summand $C_0(X)\rtimes G^\tau \to C_0(X)\rtimes G$.
 Then we can define the Kasparov product of $[c]$ and $j_G^\tau([D])\in KK(C_0(X)\rtimes G,\bb{C}\rtimes_\tau G)$. Let us define the partial assembly map: $\mu_G^\tau([D]):=[c]\otimes_{C_0(X)\rtimes G}j_G^\tau([D]):KK_G^\tau(C_0(X),\bb{C})\to KK(\bb{C},\bb{C}\rtimes_\tau G)$. In our situation, this partial assembly map is the same with the assembly map for the whole group $G^\tau$.

\begin{pro}\label{partial is the whole}
The following diagram commutes:
$$\begin{xymatrix}{
 KK_G^{\tau}(C_0(X),\bb{C}) \ar@/_70pt/[dd]_{\mu_G^\tau}
\ar[d]^{j_G^{\tau}} \ar@{^{(}-_>}[r]&  KK_{G^\tau}(C_0(X),\bb{C}) \ar[d]_{j_{G^\tau}}  \ar@/^70pt/[dd]^{\mu_{G^\tau}}
\\
 KK( C_0(X)\rtimes G,\bb{C}\rtimes_{\tau} G) 
\ar[d]^{[c]\otimes_{C_0(X)\rtimes G}-} &  KK(C_0(X)\rtimes G^\tau,\bb{C}\rtimes G^\tau) \ar[d]_{[c]\otimes_{C_0(X)\rtimes G^\tau }-}\\
 KK(\bb{C},\bb{C}\rtimes_{\tau}G) \ar@{^{(}-_>}[r] &  KK(\bb{C},\bb{C}\rtimes G^\tau).
}\end{xymatrix}$$

\end{pro}

\begin{rmk}
When the central extension $G^\tau\to G$ is topologically trivial, we can define a $U(1)$-valued $2$-cocycle $\tau$ on $G$ by fixing a diffeomorphism $G^\tau\cong G\times U(1)$, that is, $\tau:G\times G\to U(1)$ satisfying that $(g,z_1)\cdot (g_2,z_2)=(g_1g_2,z_1z_2\tau(g_1,g_2))$. 
For a vector space, a vector bundle or a Hilbert module $E$, that $E$ is $\tau$-twisted $G$-equivariant, is equivalent to that a map $\rho:G\to {\rm Aut}(E)$ satisfying that $\rho(g_1)\rho(g_2)=\rho(g_1g_2)\tau(g_1,g_2)$, is given. Using this description, we can construct the $\tau$-twisted $G$-equivariant theory, without making reference to the extended group $G^\tau$. In fact, this new description is more convenient for some cases, and we will mainly use it, although the old description (with the extended group) is useful to compare with the untwisted cases, just like Proposition \ref{partial is the whole}.

\end{rmk}

\section{$LT$-equivariant $KK$-theory and the problem}

In this short section, we will introduce the $LT$-equivariant $KK$-theory and its twisted version. The general study on this topic is far from being enough. See Caution \ref{caution}. Then, we will set the problem and simplify it by dividing the manifold into two parts, just like \cite{T1,T2,T3}.

\subsection{$LT$-equivariant $KK$-theory}

Our loop group $LT=C^\infty(S^1,T)$ is equipped with the $C^\infty$-topology. In fact, this topology is regarded as the inverse limit of the ``$L^2_k$-topology'' in this paper.\footnote{$T$ is not a vector space, and ``$L^2_k$'' does not make sense. However, up to some finite-dimensional issue, $LT$ can be identified with $C^\infty(S^1,\fra{t})$ where we can define the concept of $L^2_k$.} In order to prove the continuity of a map from $LT$, we always check the continuity with respect to the ``$L^2_k$-topology'' for some appropriate $k$.

\begin{dfn}
Let $A$ and $B$ be $LT$-$C^*$-algebras. Kasparov $(A,B)$-modules and $LT$-equivariant $KK$-groups are defined, just like Definition \ref{dfn of equivariant KK}. The $LT$-equivariant $KK$-group is denoted by $KK_{LT}(A,B)$. The homotopy class of a Kasparov module $(E,F)$ is denoted by $[(E,F)]$, as usual.
\end{dfn}

Just like usual cases, $KK_{LT}(A,B)$ is a homotopy invariant of $A$ and $B$, which is covariant and contravariant in $B$ and $A$ respectively.

\begin{dfn}
Let $(E_1,F_1)$ and $(E_2,F_2)$ be an $LT$-equivariant Kasparov $(A,C)$-module and an $LT$-equivariant Kasparov $(C,B)$-module, respectively. An $LT$-equivariant Kasparov $(A,B)$-module $(E_1\otimes E_2,F)$ is a Kasparov product of $(E_1,F_1)$ and $(E_2,F_2)$, if and only if $(E_1\otimes E_2,F)$ is a Kasparov product of $(E_1,F_1)$ and $(E_2,F_2)$ after forgetting the $LT$-action.
\end{dfn}

\begin{cau}\label{caution}
We have defined the concept of Kasparov products at the level of Kasparov modules. However, we have {\bf not} proved the homotopy invariance of the product. Moreover, we have {\bf not} proved the existence nor the uniqueness of the Kasparov product. These are because we have not yet proved the Kasparov technical theorem for $LT$-equivariant cases. At least, the elegant proof in \cite{Kas88} is not valid for our cases.

Thus, we can say ``a Kasparov {\it module} $z$ is a Kasparov product of Kasparov {\it modules} $x$ and $y$'', but we {\bf cannot say} ``the Kasparov product of {\it $KK_{LT}$-elements} $[x]$ and $[y]$, is $[z]$''. In this sense, $[z]=[x]\otimes [y]$ is just a formal expression. It is perhaps better to write it as $z\in \{x\otimes y\}$ or something.
Such a problem will be interesting, but we do not study further general theory on $KK_{LT}$.
\end{cau}

We can also formulate unbounded Kasparov modules and the bounded transformations for $LT$-equivariant cases, in the obvious way. The criterion \ref{Kucs criterion} is still valid.

We can of course study the twisted version of the $LT$-equivariant $KK$-theory. We do not repeat the details.

\begin{dfn}
We can define bounded $\tau$-twisted $LT$-equivariant Kasparov modules, $\tau$-twisted $LT$-equivariant $KK$-theory, unbounded $\tau$-twisted $LT$-equivariant  Kasparov modules, the bounded transformation, Kasparov products, and the criterion to be a Kasparov product in the unbounded picture.
\end{dfn}

\subsection{The problem}

Let us set the problem precisely here. Firstly, we recall the gauge action of $LT$ on $L\fra{t}^*:=\Omega^1(S^1,\fra{t})$. This action is defined by $l.A=A-l^{-1}dl$ for $l\in LT$ and $A\in L\fra{t}^*$, where we identify $\Omega^1(S^1,\fra{t})$ with the set of connections on the trivial bundle $S^1\times T\to S^1$. Our manifold differs from $L\fra{t}^*$ by some ``compact-part''.

\begin{dfn}
An infinite-dimensional manifold $\ca{M}$ is a ``proper $LT$-space'' if it has an $LT$-action and a proper, equivariant and smooth map $\Phi:\ca{M}\to L\fra{t}^*$. 
\end{dfn}

$LT$ has specific subgroups: $T$ is the set of constant loops, $\Pi_T$ is the set of geodesic loops starting from the origin. The ``rest part'' defined below, is denoted by $U$. Each element $g$ of the identity component of $LT$, has a lift $f\in C^\infty(S^1,\fra{t})$ such that $g(\theta)=\exp(f(\theta))$. $U$ is the set of such $g=\exp(f)$ satisfying that $\int_{S^1}f(\theta)d\theta=0$. Then, the following canonical decomposition is defined:
$$LT=(T\times \Pi_T)\times U.$$
In this decomposition, the component $T$ is regarded as the set of averages of contractible loops. 

By this decomposition, the group action $LT\circlearrowright L\fra{t}^*$ is simplified: $T$ acts there trivially, and $U$ acts there freely. The constant $1$-forms $\fra{t}\subseteq L\fra{t}^*$ is a global slice of the $U$-action. $\Pi_T$ preserves $\fra{t}$, and acts there by translations.
Therefore, $L\fra{t}^*$ is isomorphic to $\fra{t}\times U$ as $(T\times \Pi_T)\times U$-spaces. Moreover, $\Phi^{-1}(\fra{t})=:\widetilde{M}$ is a global slice in $\ca{M}$ with respect to the $U$-action; $\ca{M}$ is, as $(T\times \Pi_T)\times U$-spaces, isomorphic to $\widetilde{M}\times U$.\footnote{In \cite{T1,T3}, we set $M:=\ca{M}/\Omega T$, where $\Omega T$ is the set of loops starting from the origin. $\widetilde{M}$ is a $\Pi_T$-covering space of $M$.}

We impose the following assumption on the $\widetilde{M}$-part in order to study the index theory.

\begin{asm}\label{Spinc condition}
$\Phi^{-1}(\fra{t})$ is even-dimensional and $T\times \Pi_T$-equivariantly $Spin^c$, that is, $\ca{M}$ is ``even-dimensional and $Spin^c$''.
\end{asm}

Our problem is the following.

\begin{prob}
For a proper $LT$-space $\ca{M}$ satisfying Assumption \ref{Spinc condition} and equipped with a $\tau$-twisted $LT$-equivariant line bundle $\ca{L}$, construct an $LT$-equivariant Spinor bundle $\ca{S}$, and study the $KK$-theoretical $LT$-equivariant index theory. More precisely, construct three Kasparov modules corresponding to the index element, the Clifford symbol element and the Dirac element. Then, study the assembly map.
\end{prob}

Just like \cite{T1,T2,T3}, we would like to divide the problem into two parts. For this aim, we recall Section 2.2.
Just like Lemma \ref{division of the problem 1}, we can prove the following.

\begin{lem}
Let $A_1$ and $B_1$ be $T\times \Pi_T$-$C^*$-algebras, and let $A_2$ and $B_2$ be $U$-$C^*$-algebras. Then, $A:=A_1\otimes A_2$ and $B:=B_1\otimes B_2$ are $LT$-$C^*$-algebras.

Suppose that an unbounded $T\times \Pi_T$-equivariant Kasparov $(A_1,B_1)$-module $(E_1,D_1)$ and unbounded an $U$-equivariant Kasparov $(A_2,B_2)$-module $(E_2,D_2)$ are given. Then, 
$$(E,D):=(E_1\otimes E_2,D_1\otimes \id+\id\otimes D_2)$$
is an $LT$-equivariant Kasparov $(A,B)$-module.
\end{lem}

The twisted version of this lemma is not obvious, but it is still valid, thanks to Proposition 2.27 in \cite{FHTII} (or Proposition 2.8 in \cite{T3}). In this proposition, an admissible $U(1)$-central extension of $LT$ is proved to be split into two parts: $LT^\tau= (T\times \Pi_T)^\tau\boxtimes_{U(1)}U^\tau$. We use the same symbol for the restrictions of $\tau$.


\begin{lem}
Under the isomorphism $\ca{M}\cong \widetilde{M}\times U$,
$\ca{L}=\ca{L}|_{\widetilde{M}}\boxtimes\ca{L}|_{U}$ including the group action. The line bundle $\ca{L}|_{U}$ is given by the standard construction $U^\tau\times_{U(1)} \bb{C}$.
\end{lem}

Thanks to these results, we may construct everything on $\widetilde{M}$ and $U$, separately. Since $\widetilde{M}$ is finite-dimensional, we can construct anything without serious problems. For the $U$-part, we set the following (simplified) problem. We will concentrate on it from now on.

\begin{prob}
For a positive definite central extension $\tau$ of $U$, construct three unbounded Kasparov modules which are representatives of the index element $[\widetilde{\Dirac}] \in KK_U^\tau(\SC(U),\scr{S})$, the Clifford symbol element $[\widetilde{\sigma_{\Dirac}^{Cl}}]\in KK_U^\tau(\SC(U),\SC(U))$ and the Dirac element $[\widetilde{d_U}]\in KK_U(\SC(U),\scr{S})$, respectively. Then, prove that the Kasparov module representing $[\widetilde{\Dirac}]$ is a Kasparov product of the one representing $[\widetilde{\sigma_{\Dirac}^{Cl}}]$ and the one representing $[\widetilde{d_U}]$.
Formally, prove the equality
$$[\widetilde{\Dirac}]=[\widetilde{\sigma_{\Dirac}^{Cl}}]\otimes_{\SC(U)} [\widetilde{d_U}].$$
Then, study the assembly map.
\end{prob}

If we can solve this problem, we can define the Kasparov modules on the whole manifold $\ca{M}$, ``using'' Section 2.2. Put $\SC(\ca{M}):= \SC(U)\otimes Cl_\tau(\widetilde{M})$, and let $\cancel{D}$ be the $(T\times \Pi_T)^\tau$-equivariant Dirac operator on $\widetilde{M}$.
We define them as follows. The following tensor product is indeed at the level of modules.
\begin{align*}
[\widetilde{\ca{D}}] & :=[\widetilde{\cancel{D}}]\otimes [\widetilde{\Dirac}];\\
[\widetilde{\sigma_{\ca{D}}^{Cl}}] & :=[\widetilde{\sigma_{\cancel{D}}^{Cl}}]\otimes [\widetilde{\sigma_{\Dirac}^{Cl}}];\\
[\widetilde{d_{\ca{M}}}]&:= [d_{\widetilde{M}}]\otimes [\widetilde{d_{U}}].
\end{align*}
Then, Section 2.2 imply the equality $[\widetilde{\ca{D}}]= [\widetilde{\sigma_{\ca{D}}^{Cl}}]\otimes [\widetilde{d_{\ca{M}}}]$.


As a final remark of this section, we recall some examples of proper $LT$-spaces. See \cite{T3}.

\begin{exs}
$(1)$ A Hamiltonian $LT$-space is automatically a proper $LT$-space.

$(2)$ In particular, a gauge orbit $U\times \Pi_T\subseteq L\fra{t}^*$ is a proper $LT$-space.

$(3)$ In general, for an even-dimensional $Spin^c$-manifold $M$ equipped with a $T$-action, suppose that it has a $T$-equivariant Spinor bundle, and a $T$-equivariant map $\phi:M\to T$, where $T$ acts on itself by the conjugate action, namely trivially. Then, the pullback construction $M\times_{T} L\fra{t}^*$ with respect to the holonomy map $L\fra{t}^*\to T$ gives a proper $LT$-space. The condition looks slightly complicated, but a symplectic manifold with an $T$-action always satisfies these condition, after perturbing the symplectic form. The map taking values in $T$ is the circle valued-moment map studied in \cite{McD}.
\end{exs}

\section{The HKT algebra}

In this section, we study the HKT algebra: its definition, the symmetry, and the representation. The new point is the connection with the the representation of the infinite-dimensional group $U$.

\subsection{Properties of $U$}

We begin with studies of $U$. 
$U$ is identified with $\bbra{f:S^1\to \fra{t}\mid \int_{S^1}f(\theta)d\theta=0}$.  Fix an inner product of $\fra{t}$, denoted by $\inpr{\bullet}{\bullet}{\fra{t}}$,
and then we can take finite-dimensional approximations of $U$:
$$\Lie(U_N):={\rm Span}\bbra{\frac{\cos\theta}{\sqrt{\pi}}{\bf 1},\frac{\sin\theta}{\sqrt{\pi}}{\bf 1},\cosi{2}{\bf 1},\sine{2}{\bf 1},\cdots,\cosi{N}{\bf 1},\sine{N}{\bf 1}},$$
where ${\bf 1}$ is a unit vector of $\fra{t}$. From now on, we fix ${\bf 1}$, and we omit it.
$\Lie(U)$ has a symplectic form $\omega(f_1,f_2):=\int \inpr{f_1}{\frac{df_2}{d\theta}}{\fra{t}}d\theta$. The above is a symplectic bases: $\omega(\cosi{k},\cosi{l})=\omega(\sine{k},\sine{l})=0$ and $\omega(\cosi{k},\sine{l})=\delta_{k,l}$.

The group $U$ is, as a Lie group, diffeomorphic to ``$\bb{R}^\infty$''.\footnote{Topology of ``$\bb{R}^\infty$'' is highly non-trivial.} The coordinate is often written as, when $U$ is regarded as a {\it space},
$$\exp\bra{\sum \bra{x_k\cosi{k}+y_k\sine{k}}}\mapsto (x_1,y_1,x_2,y_2,\cdots)$$ 
and the coordinate is often written as, when $U$ is regarded as a {\it group},
$$\exp\bra{\sum \bra{g_k\cosi{k}+h_k\sine{k}}}\mapsto (g_1,h_1,g_2,h_2,\cdots).$$

The completion of $U$ with respect to the norm on the Lie algebra
$$\|(g_1,h_1,g_2,h_2,\cdots)\|_{L^2_l}^2:=\sum_k(1+k^{2l})(g_k^2+h_k^2)$$
is denoted by $U_{L^2_l}$. An element of this group is called an ``$L^2_l$-loop''.
In order to define the HKT algebra, {\bf we use the $L^2_{1/2}$-metric}. With respect to this norm, $\omega(\bullet,\bullet)$ is a continuous bilinear form.

\begin{dfn}\label{central extension of LT}
The central extension $U^\tau$ is given by $U\times U(1)$ equipped with the multiplication
$$(\exp(f_1),z_1)\cdot (\exp(f_2),z_2):=(\exp(f_1+f_2),z_1z_2e^{\frac{i}{2}\omega(f_1,f_2)}).$$

$\Lie(U^\tau)\cong \Lie(U)\oplus \fra{u}(1)$.
Let $K$ be the generator of $\fra{u}(1)$. For $\exp(f)\in U$, $X\in \Lie(U)$ and $z\in U(1)$, the adjoint representation is given by
$$\Ad_{\exp(f)}(X) = X+\omega(f,X)K,$$
$$\Ad_z(X)=X,$$
and hence the Lie algebra structure is given by, for $f_1,f_2\in \Lie(U)$,
$$[f_1,f_2]=\omega(f_1,f_2)K, $$
$$[K,f]=0.$$
\end{dfn}

$LT$ has an $S^1$-symmetry $\theta.f(s):=f(s+\theta)$. When we deal with this action, we write $S^1$ as $\bb{T}_\rot$. A positive energy representation is a representation of $LT$ which reflects the $\bb{T}_\rot$-symmetry, and which satisfies a certain finiteness condition. We define it for $U$ in our language. Consult \cite{T3} (or of course \cite{PS,FHTII}) for the $T\times \Pi_T$-part.

\begin{dfn}
Let $V$ be a separable Hilbert space, and let $U(V)$ be the unitary group of $V$, which is topologized by the compact open topology.
A continuous map $\rho:U\to U(V)$ is a positive energy representation (PER for short) at level $\tau$ of $U$ if it satisfies the following condition:
\begin{itemize}
\item $\rho(\exp(f_1))\circ \rho(\exp(f_2))=\rho(\exp(f_1+f_2))e^{\frac{i}{2}\omega(f_1,f_2)}$. In other words, $\rho$ is a homomorphism from the extended group $U^\tau$ satisfying that $\rho(i(z))=z\id_V$.
\item It extends to $\rho:U\rtimes \bb{T}_\rot\to U(V)$.
\item When we decompose $V$ by the weight of $\bb{T}_\rot$-action by $\rho|_{\bb{T}_\rot}$ as $V=\oplus_nV_n$, each $V_n$ is finite-dimensional and $V_n=0$ for any sufficiently small $n$.
\end{itemize}
\end{dfn}

For the sake of computations, it is more convenient to study the infinitesimal version $d\rho:\Lie(U)\to \End(V)$, which satisfies the commutation relation $[d\rho(f_1),d\rho(f_2)]=i\omega(f_1,f_2)\id$.
Let $d$ be the infinitesimal generator of $\bb{T}_\rot$. Then, the Lie bracket $\Lie(U^\tau\rtimes \bb{T}_\rot)$ is given by $[d,f]=f'$ and $[d,K]=0$.

We can construct an irreducible PER as follows. It is known that there is only one PER up to isomorphism.

\begin{dfn}
By an $L^2$-function on $\bb{R}$ defined by $\frac{1}{\pi^{1/4}}e^{-\frac{x^2}{2}}$ which is a unit vector, define an isometric embedding 
$$I_N:L^2(\bb{R}^N)\ni f\mapsto f\otimes \frac{1}{\pi^{1/4}}e^{-\frac{x^2}{2}}\in L^2(\bb{R}^{N+1}),$$
and define $\ud{L^2(\bb{R}^\infty)}$ by the Hilbert space inductive limit $\varinjlim L^2(\bb{R}^N)$. The ``infinite tensor product'' $\frac{1}{\pi^{1/4}}e^{-\frac{x_1^2}{2}}\otimes \frac{1}{\pi^{1/4}}e^{-\frac{x_2^2}{2}}\otimes \cdots$ defines a unit vector denoted by ${\bf 1}_b$, where ``$b$'' comes from ``boson''.

The ``function'' $\frac{1}{\pi^{1/4}}e^{-\frac{x_N^2}{2}}\otimes \frac{1}{\pi^{1/4}}e^{-\frac{x_{N+1}^2}{2}}\otimes \cdots\otimes \frac{1}{\pi^{1/4}}e^{-\frac{x_{M}^2}{2}}$ is denoted by $({\bf 1}_b)_N^M$, where $1\leq N\leq M\leq \infty$. Note that $M=\infty$ is allowed. In this notation, ${\bf 1}_b=({\bf 1}_b)_1^\infty$.

Let $L^2(\bb{R}^N)_\fin$ be the subspace which is algebraically spanned by functions of the form ``polynomial $\times$ $e^{-\frac{\|x\|^2}{2}}$''. $\ud{L^2(\bb{R}^\infty)_\fin}$ denotes the algebraic inductive limit $\varinjlim^\alg L^2(\bb{R}^N)_\fin$. Needless to say, ${\bf 1}_b$ is an element of $\ud{L^2(\bb{R}^\infty)_\fin}$.
\end{dfn}

In a diagram, our Hilbert spaces are organized in the following commutative diagram
$$\xymatrix{
L^2(\bb{R}) \ar[r]^{I_1} \ar[rrd]_{J_1} & L^2(\bb{R}^2) \ar[r]^{I_2} \ar[rd]^{J_2} & \cdots \ar[r]^{I_{N-1}}& L^2(\bb{R}^N) \ar[r]^{I_N} \ar[ld]_{J_N} & L^2(\bb{R}^{N+1}) \ar[r]^{I_{N+1}} \ar[lld]^{J_{N+1}}  & \cdots \\
& & \ud{L^2(\bb{R}^\infty)} & & }
$$
Every Hilbert space $L^2(\bb{R}^N)$ has a dense subspace $L^2(\bb{R}^N)_\fin$. $I_N$'s and $J_N$'s preserve these dense subspaces.

\begin{dfn}
On $\ud{L^2(\bb{R}^\infty)}_\fin$, let
$$d\rho\bra{\cosi{k}}:=\frac{\partial}{\partial x_k};\ d\rho\bra{\sine{k}}:=ix_k\times.$$
Strictly speaking, these operations are defined as follows: 
For $f\in \ud{L^2(\bb{R}^\infty)_\fin}$, choosing $\widetilde{f}\in L^2(\bb{R}^N)$ such that $J_N(\widetilde{f})=f$ and $N\geq k$, we define 
$$d\rho\bra{\cosi{k}}(f):= J_N\bra{\frac{\partial \widetilde{f}}{\partial x_k}},$$
and similarly for $d\rho\bra{\sine{k}}(f)$.

\end{dfn}

These operators are skew-adjoint operators, and hence they generate a unitary operators. Moreover, they satisfy the positive energy condition. In order to prove that, we introduce the complex basis:
$$z_k:=\frac{1}{\sqrt{2}}\bra{\cosi{k}+i\sine{k}};$$
$$\overline{z_k}:=\frac{1}{\sqrt{2}}\bra{\cosi{k}-i\sine{k}}.$$
These are unit vectors. The infinitesimal generator of the rotation action is given by
$$d\rho(d):=-i\sum_kkd\rho(z_k)d\rho(\overline{z_k}).$$
This is well-defined and satisfies $[d\rho(d),f]=d\rho\bra{\frac{df}{d\theta}}$, as computed in \cite{T1}. Moreover, each eigenvalue of $d\rho(d)/i$ is a non-negative integer, and its multiplicity is the number of partitions. In fact, $\ud{L^2(\bb{R}^\infty)_\fin}$ is the set of finite linear combinations of eigenvectors of $d\rho(d)$.

We can define the dual representation in the usual way. To clarify the notation, we describe it.

\begin{dfn}
Let $\ud{L^2(\bb{R}^\infty)^*}$ be the dual space of $\ud{L^2(\bb{R}^\infty)}$, which is anti-linearly isomorphic to $\ud{L^2(\bb{R}^\infty)}$.
For $g\in U$, let $\rho^*(g)\in U(\ud{L^2(\bb{R}^\infty)^*})$ be the operator defined by $[{}^t\rho(g^{-1})]$. Then, $\rho^*(\exp(f_1))\circ \rho^*(\exp(f_2))=\rho^*(\exp(f_1+f_2))e^{-\frac{i}{2}\omega(f_1,f_2)}$, that is, $\rho^*$ is a ``$-\tau$-twisted'' representation. For $\theta\in \bb{T}_\rot$, let $\rho^*(\theta):={}^t\rho(-\theta)$. This operation extends $\rho$ to a homomorphism from $U^\tau\rtimes \bb{T}_\rot$\footnote{The action of $i(z)\in i(U(1))\subseteq U^\tau$ is given by $z^{-1}\id$.}. It satisfies the ``negative energy condition''. The set of finite linear combinations of eigenvectors of $d\rho^*(d)$, is denoted by $\ud{L^2(\bb{R}^\infty)^*_\fin}$.
Just like $\ud{L^2(\bb{R}^\infty)}$, we use the symbol ${\bf 1}_b^*$ and $({\bf 1}_b^*)_N^M$.
\end{dfn}

As proved in Theorem 3.15 of \cite{T1}, $L^2(\bb{R}^N)\otimes L^2(\bb{R}^N)^*$ is isomorphic to $L^2(U_N)$ as twisted representation spaces, by the Peter-Weyl type homomorphism\footnote{It is more natural to regard $L^2(U_N)$ as the Hilbert space consisting of $L^2$-sections of a line bundle $\ca{L}\to U_N$ defined by $U_N^\tau\otimes_{U(1)}\bb{C}$. In this picture, several properties can be easily proved: two actions $L$ and $R$ clearly commute, thanks to the associativity of the group operation.
However, in the present paper, it is more convenient to forget the line bundle.}.
We only define the regular representations and describe the Peter-Weyl type homomorphism. For the details, see also section 3 in \cite{T1}. In the following, we identify $\Lie(U_N)$ with $U_N$ by the exponential map to simplify the notation. For $g=\exp(f)\in U$, $\omega(g,\bullet)$ means that $\omega(f,\bullet)$.

\begin{dfn}
For $f\in L^2(U_N)$ and $g\in U_N$, define
$$[L(g)(f)](x):=f(x-g)e^{\frac{i}{2}\omega(g,x)};$$
$$[R(g)(f)](x):=f(x+g)e^{\frac{i}{2}\omega(g,x)}.$$
\end{dfn}

Then, $L$ and $R$ are twisted representations at opposite level:
$$L(g_1)\circ L(g_2)=L(g_1+g_2)e^{\frac{i}{2}\omega(g_1,g_2)};$$
$$R(g_1)\circ R(g_2)=R(g_1+g_2)e^{-\frac{i}{2}\omega(g_1,g_2)}.$$

Let $e_k$'s and $f_k$'s be the unit vectors of $\Lie(U_N)$ corresponding to $\cosi{k}$'s and $\sine{k}$'s, respectively. Let $(x_1,y_1,\cdots,x_N,y_N)$ be the coordinate described above. Then, the infinitesimal action is computed as follows:
$$dL(e_k)=-\frac{\partial}{\partial x_k}+\frac{i}{2}y_k;\ \ dL(f_k)=-\frac{\partial}{\partial y_k}-\frac{i}{2}x_k;$$
$$dR(e_k)=\frac{\partial}{\partial x_k}+\frac{i}{2}y_k;\ \ dR(f_k)=\frac{\partial}{\partial y_k}-\frac{i}{2}x_k.$$

\begin{pro}
The Peter-Weyl type homomorphism $\Psi:L^2(\bb{R}^N)\otimes L^2(\bb{R}^N)^*\to L^2(U_N)$ defined by
$$\Psi(v\otimes f)(x):=\bra{\frac{1}{\sqrt{2\pi}}}^N\innpro{\rho(-x)(v)}{f}{}$$
is an isometric isomorphism as representation spaces of $U_N^\tau\times U_N^{-\tau}$.
\end{pro}

Seeing this result, we define the ``$L^2$-space of $U$'' as follows. See also Section 3.2 of \cite{T1}.

\begin{dfn}
$\ud{L^2(U)}$ is defined by $\ud{L^2(\bb{R}^\infty)}\otimes \ud{L^2(\bb{R}^\infty)^*}$. The representations $\rho\otimes \id$ and $\id\otimes \rho^*$ are also denoted by $L$ and $R$, respectively.

Let $\ud{L^2(U)_\fin}$ be the algebraic tensor product $\ud{L^2(\bb{R}^\infty)_\fin}\otimes^\alg \ud{L^2(\bb{R}^\infty)_\fin^*}$. Each element of this dense subspace is regarded as ``polynomial $\times$ $\bra{\frac{1}{\sqrt{2\pi}}}^\infty e^{-\frac{|g|^2}{4}}$''. We introduce notations
$$\vac:= \Psi({\bf 1}_b\otimes {\bf 1}_b^*);\ 
\vac_N^M:= \Psi(({\bf 1}_b)_N^M\otimes ({\bf 1}_b^*)_N^M).
$$
\end{dfn}

\subsection{The definition of the Higson-Kasparov-Trout algebra $\SC(U)$}

We recall the definition of the HKT algebra which can be regarded as the ``suspension of the Clifford algebra-valued function algebra of $U$''.
In fact, we will slightly modify the description so that it fits with our representation introduced in the next subsection. This is because our ``smooth functions'' must be ``approximately constant''. We will prove that the modified one is isomorphic to the original one as $C^*$-algebras thorough a ``diffeomorphism on $U$''.

We concentrate on $U$ which has a special property: $\bb{T}_\rot$ acts there linearly, and each weight space is real two-dimensional. These properties allow us to define a natural filtration of finite-dimensional subspaces thanks to the Fourier series theory
$$U_1\subseteq U_2\subseteq \cdots \subseteq U_N\subseteq \cdots \subseteq U.$$
We use only this filtration to construct the HKT algebra. This is the first (but small) different point from the original one.


We also modify the Clifford element. This is the main different point from the original one. Recall the coordinate on $U$: $\exp\bra{\sum_k\bra{x_k\cosi{k}+y_k\sine{k}}}\mapsto (x_1,y_1,x_2,y_2,\cdots)$, and recall that $e_k$ and $f_k$ denote the cotangent vectors $dx_k$ and $dy_k$, respectively.


\begin{dfn}\label{Clifford element in SC}
Let $l$ be a positive real number greater than $\frac{3}{2}$.
As a Clifford algebra-valued function on $U_M$, we introduce a Clifford element 
$$C_N^M:=\sum_{k=N}^Mk^{-l} (x_ke_k+y_kf_k)\in C^\infty(U_M,\Cl_+(T^*U_M))$$
for $N\leq M< \infty$. 

By an abuse of notation, $C_N^\infty$ (or simply $C_N$) means the {\bf formal} infinite sum
$\sum_{i=N}^\infty k^{-l} (x_ke_k+y_kf_k)$. We will prove that this formal infinite sum makes sense as an operator on $\ud{L^2(U)}$.
\end{dfn}

\begin{rmks}
$(1)$ Except for the statement of Lemma \ref{Ours = original one} and its proof, we fix $l$, and we omit $l$ in the symbol of the Clifford element. 

$(2)$ The original Clifford element is $\sum_{i=N}^M(x_ke_k+y_kf_k)\in C^\infty(U_M,\Cl_+(T^*U_M))$. If we consider (formally) the case when $N=1$ and $M=\infty$, this original Clifford element is invariant under all the isometric linear action. On the other hand, our Clifford element is much less symmetric.
\end{rmks}

Let us recall that a self-adjoint unbounded operator can be substituted for a function on the real line. For example, the heat kernel $e^{-tD^2}$ can be defined for a Dirac operator $D$. If the operator has compact resolvent and the function vanishes at infinity, the obtained operator is compact. Just like this, we can get an element of a $C^*$-algebra from an ``unbounded multiplier with compact resolvent'' and a function on the real line vanishing at infinity.

\begin{dfn-pro}[See \cite{Wor} or Appendix A.3. of \cite{HKT}]
Let $A$ be a $\bb{Z}_2$-graded $C^*$-algebra. An $A$-linear operator $D$ from a dense, $\bb{Z}_2$-graded, right $A$-submodule $\ca{A}$, into $A$ is an (odd) {\bf unbounded self-adjoint multiplier} if the following conditions are fulfilled:
\begin{itemize}
\item $(Dx)^*y=x^*(Dy)$ for all $x,y\in \ca{A}$.
\item The operator $D\pm i\id$ are isomorphisms from $\ca{A}$ to $A$.
\item $D$ reverses the grading.
\end{itemize}

For such $D$, there is an operator calculus homomorphism taking values in the multiplier algebra $M(A)$ of $A$,
$$\scr{S}\ni f\mapsto f(D)\in M(A)$$
mapping $(x\pm i)^{-1}$ to $(D\pm i\id)^{-1}$. $\scr{S}$ is graded by $[\epsilon(f)](t)=f(-t)$. With respect to this grading, the above homomorphism preserves the grading.

If $(D\pm i\id)^{-1}$ belongs to $A$ itself, this homomorphism takes value in $A$. Such $D$ is said to be with {\bf compact resolvent}, which is because the set of $A$-compact operators on the right $A$-module is $A$ itself, while the set of adjointable operators on $A$ is $M(A)$.
\end{dfn-pro}

On $\scr{S}$, the operator $X$ defined by $[Xf](x):=xf(x)$ is an unbounded self-adjoint multiplier with compact resolvent. This is a typical example of unbounded multipliers with compact resolvent.

Moreover, the Clifford element $C_{N,M}$ is also an unbounded self-adjoint multiplier on the $C^*$-algebra $\scr{C}(U_M\ominus U_{N-1}):= Cl_\tau(U_M\ominus U_{N-1})$. It can be proved as follows. Firstly, let $N$ be $1$. Then, $C_{1,M}$ is an unbounded multiplier with compact resolvent. This is because $(C_1^M+i)(C_1^M-i)=1+\sum k^{-2l}(x_k^2+y_k^2)$ is strictly positive. The general case is completely parallel to this  one.
Consequently, $X\otimes \id+\id\otimes C_{N}^M$ is an unbounded self-adjoint multiplier with compact resolvent of $\scr{S}\otimes\scr{C}(U_M\ominus U_{N-1})$.
Let $\SC(U_N)$ be the graded tensor product $\scr{S}\otimes \scr{C}(U_N)$

We are in the position to talk about the homomorphism connecting $\SC(U_N)$ and $\SC(U_M)$. Notice that $\SC(U_M)\cong \SC(U_M\ominus U_{N})\otimes \scr{C}(U_{N})$.

\begin{dfn}\label{dfn of beta}
A $*$-homomorphism $\beta_{N,M}:\SC(U_N)\to \SC(U_M)$ for $N\leq M$ is defined by
$$\beta_{N,M}(f\otimes h):=f(X\otimes \id+\id\otimes C_{N+1}^M)\otimes h.$$
\end{dfn}

Thanks to the equality $C_{N+1}^M+C_{M+1}^L=C_{N+1}^L$ and Proposition 3.2 in \cite{HKT}, we find that
$$\beta_{M,L}\circ \beta_{N,M}=\beta_{N,L}.$$
It enable us to take an inductive limit of the sequence
$$\SC(U_1)\xrightarrow{\beta_{1,2}} \SC(U_2)\xrightarrow{\beta_{2,3}}\cdots.$$

\begin{dfn}
By the $C^*$-algebra inductive limit, we define
$$\SC(U):=\varinjlim \SC(U_N).$$
We call this algebra the {\bf HKT algebra}.
\end{dfn}

We define a subalgebra of the HKT algebra consisting of ``rapidly decreasing functions''. It is convenient to adopt narrower one than the natural one. Let $f_e(x):=e^{-x^2}$ and $f_o(x):=xe^{-x^2}$. They generate $\scr{S}$, thanks to the Stone-Weierstrass theorem. ``$e$'' and ``$o$'' come from even and odd, respectively.

\begin{dfn}
Let $\scr{S}_\fin\subseteq \scr{S}$ be the $*$-subalgebra algebraically generated by $f_e$ and $f_o$. 
Let $\scr{C}(U_N)_\fin$ be the dense subalgebra consisting of Schwartz functions. 
Let $\SC(U_N)_\fin$ be the algebraic tensor product of $\scr{S}_\fin$ and $\scr{C}(U_N)_\fin$.
The Bott map $\beta_{N,M}$ maps $\SC(U_N)_\fin$ into $\SC(U_M)_\fin$, and hence we can define a dense subalgebra by the algebraic inductive limit
$$\SC(U)_\fin:=\varinjlim^\alg \SC(U_N)_\fin.$$
The $*$-homomorphism defining the inductive limit $\SC(U_N)\to \SC(U)$ is denoted by $\beta_N$. For the special one $\beta_0:\scr{S}\to \SC(U)$, we use $\beta$.
\end{dfn}

Our HKT algebra is isomorphic to the original one as $C^*$-algebras, as follows. However, our description is much more convenient for (possibly only) our purpose.

\begin{lem}\label{Ours = original one}
Let ${}^l\beta_{N,M}$ be the $*$-homomorphism in Definition \ref{dfn of beta} induced by the Clifford element ${}^lC_N^M:=\sum_{i=N}^Mk^{-l} (x_ke_k+y_kf_k)$. Fix a trivialization $T^*U_M\cong U_M\times \Lie(U_M)^*$, and we regard $C_N$ as, not a section of the Clifford algebra baundle, but a Clifford algebra-valued function.
If we define a $*$-isomorphism ${}^l\Phi:\scr{C}(U_N)\to \scr{C}(U_N)$ by
$${}^l\Phi_N(h)(x_1,y_1,x_2,y_2,\cdots,x_N,y_N):=h(x_1,y_1,2^{-l}x_2,2^{-l}y_2,\cdots,N^{-l}x_N,N^{-l}y_N),$$
the following diagram commutes:
$$\begin{CD}
\SC(U_N) @>\id\otimes{}^l\Phi_N>> \SC(U_N) \\
@V{}^0\beta_{N,M}VV @VV{}^l\beta_{N,M}V \\
\SC(U_M) @>\id\otimes{}^l\Phi_M>> \SC(U_M) 
\end{CD}
$$

Consequently, our HKT algebra is $*$-isomorphic to the original one.
\end{lem}
\begin{rmk}
Let ${}^l\phi$ be the diffeomorphism given by $(x_1,y_1,\cdots,x_N,y_N)\mapsto (x_1,y_1,\cdots,N^{-l}x_N,N^{-l}y_N)$. Note that ${}^l\Phi$ is {\bf NOT} the natural pullback $({}^l\phi)^*$, but just ${}^l\Phi(f\otimes c)=[({}^l\phi)^*(f)]\otimes c$, where $f$ is a {\it scalar}-valued function and $c\in \Cl_+(U_N)$. In this sense, our isomorphism is not very natural.

\end{rmk}
\begin{proof}
Once we prove the statement when $N=0$, the rest part is easy. All the homomorphisms are $*$-homomorphisms, and hence it suffices to check the commutativity for generators: $f_e$ and $f_o$. 
Let us compute the compositions at $f_e$ and $f_o$. See Proposition 3.2 in \cite{HKT} also. For $f_e$,
\begin{align*}
&\bigl[(\id\otimes{}^l\Phi_M)\circ {}^0\beta_{0,M}(f_e)\bigr] (x_1,y_1,x_2,y_2,\cdots,x_M,y_M) \\
&= (\id\otimes{}^l\Phi_M)[f_e(X)\otimes f_e({}^lC_1^M)] (x_1,y_1,x_2,y_2,\cdots,x_M,y_M)\\
&=f_e(X)\otimes e^{-\sum_{k=1}^Mk^{-2l}(x_k^2+y_k^2)}.
\end{align*}
Since ${}^l\Phi_0$ is the identity,
\begin{align*}
& {}^l\beta_{0,M}\circ(\id\otimes{}^l\Phi_0)(f_e)(x_1,y_1,x_2,y_2,\cdots,x_M,y_M) \\
&=f_e\bra{X\otimes \id+ \id\otimes \sum_{k=1}^Mk^{-l} (x_ke_k+y_kf_k)}\\
&=f_e(X)\otimes e^{-\sum_{i=1}^Mk^{-2l} (x_k^2+y_k^2)}.
\end{align*}

For $f_o$,
\begin{align*}
&(\id\otimes{}^l\Phi_M)\circ {}^0\beta_{0,M}(f_o)(x_1,y_1,x_2,y_2,\cdots,x_M,y_M) \\
&=Xf_e(X)\otimes f_e(x_1,y_1,2^{-l}x_2,2^{-l}y_2,\cdots,M^{-l}x_M,M^{-l}y_M)
+f_e(X)\otimes Yf_e(Y)|_{Y={}^lC_{1,M}} \\
&=Xf_e(X)\otimes e^{-\sum_{k=1}^Mk^{-2l}(x_k^2+y_k^2)}
+f_e(X)\otimes \sum_{k=1}^Mk^{-l} (x_ke_k+y_kf_k)e^{-\sum_{b=1}^Mb^{-2l}(x_b^2+y_b^2)}.
\end{align*}

\begin{align*}
& {}^l\beta_{0,M}\circ(\id\otimes{}^l\Phi_0)(f_o)(x_1,y_1,x_2,y_2,\cdots,x_M,y_M) \\
&=f_o(X\otimes \id+\id\otimes {}^lC_1^M)\\
&=Xf_e(X)\otimes e^{-\sum_{k=1}^Mk^{-2l}(x_k^2+y_k^2)}
+f_e(X)\otimes \sum_{k=1}^Mk^{-l} (x_ke_k+y_kf_k)e^{-\sum_{b=1}^Mb^{-2l}(x_b^2+y_b^2)}.
\end{align*}

\end{proof}

\begin{rmk}
This isomorphism suggests that our ``function on $U$'' is ``continuous'' even in a weaker topology.
\end{rmk}




\subsection{A group action and a representation}

In this subsection, we will see that $U$ acts on $\SC(U)$, and $\SC(U)$ has a representation on a Hilbert $\scr{S}$-module. It looks like
\begin{center}
the canonical $\scr{S}$-module $\scr{S}$ $\otimes$ ``representation of $\scr{C}(U)$'',
\end{center}
although the phrase ``representation of $\scr{C}(U)$'' does not make sense.

These two ingredients the ``group action'' and the ``representation'' are compatible. We will put off the proof until the next (main) section, because we will deal with two different $U$-actions on the representation space.

\begin{thm}\label{the U action}
$U$ acts on $\SC(U)$ continuously by $*$-automorphisms.

\end{thm}
\begin{proof}
In this proof, to be precise, $\sigma$ denotes the group action of $U$ on itself: $\sigma_g(x):=x+g$, where we identify $U$ with $\Lie(U)$, and we regard $U$ as a vector space. 

Let $U_\fin$ be the algebraic inductive limit of $U_N$: $U_\fin:=\cup_N U_N$. It acts on $\SC(U)_\fin$ in the natural way: Let $g\in U_N$ and $f\otimes h\in\SC(U_M)_\fin$. When $N\leq M$, $U_N\subseteq U_M$ and hence $U_N$ acts on $U_M$. The action is defined by $g.(f\otimes h):=f\otimes \sigma_{g^{-1}}^*h$. When $M<N$, replace $f\otimes h$ with $\beta_{M,N}(f\otimes h)$. 
This action is well-defined: $\beta_{M,L}[g.(f\otimes h)]=g.[\beta_{M,L}(f\otimes h)]$ for $L>M\geq N$. The map induced by $g$ clearly gives an isometric $*$-automorphism. 

We must check that the action extends to a $U$-action on $\SC(U)$. For this aim, we need some quantitative arguments.

Firstly, we check that the operator on $\SC(U)_\fin$ induced by $g\in U_\fin$ is continuous with respect to the norm of $\SC(U)$. Indeed, for $a\in \SC(U_M)$ and $g\in U_N$ (suppose that $M\geq N$),
\begin{align*}
\|\beta_{M}(g.a)\|_{\SC(U)}&=
\lim_{L\to \infty} \|\beta_{M,L}(g.a)\|_{\SC(U_L)}=\lim_{L\to \infty} \|g.[\beta_{M,L}(a)]\|_{\SC(U_L)} \\
&=\lim_{L\to \infty} \|\beta_{M,L}(a)\|_{\SC(U_L)}=
\|\beta_{M}(a)\|_{\SC(U)},
\end{align*}
where we used the fact that $U_N$ acts on $\SC(U_L)$ as isometric $*$-automorphisms for any $L\geq N$. Consequently, $U_\fin$ acts on $\SC(U)$ by $*$-automorphisms. The precise definition is as follows: For $a\in \SC(U)$ and $g\in U_\fin$, choose an approximating sequence $\{a_n\}\subseteq \SC(U)_\fin$, and define $g.a$ by $\lim_n g.a_n$. Since $g\in U_\fin$ gives an isometry, this limit exists and independent of the choice of $\{a_n\}$.


Secondly, we need to extend the $U_\fin$-action to a $U$-action.
For this purpose, it is natural to define the action by $g.a:=\lim_{n\to \infty} (g_n).a$. In order to justify this definition, we need to check the following: For any fixed $a\in \SC(U)$ and for any two sufficiently close elements $g,g'\in U_\fin$, $g.a-g'.a$ is small. In order to describe everything precisely, take a positive number $\varepsilon$, and choose $a_0\in \SC(U_N)$ so that $\|a-a_0\|<\varepsilon$. Then
$$\|g.a-g'.a\|\leq \|g.(a-a_0)\|+\|g.a_0-g'.a_0\|+\|g'.(a-a_0)\|$$
The first and the third terms are less than $\varepsilon$, because $g$ and $g'$ give isometries. 
For the second term, suppose that $g$ and $g'$ belong to $U_L$. 
We may assume that $a_0$ is of the form $\sum_{\text{finite sum}} f_i\otimes a_i$ for some $a_i$'s which are compactly supported and smooth, and $f_i$'s which are monomials of $f_e$ and $f_o$.
In a similar way to prove the Leibniz rule, we may deal with only $f_e\otimes a$ and $f_o\otimes a$. 
In fact, we may assume that $f_i\otimes a_i$ is of the form $(f_e)^n(f_o)^m(a')^{n+m-1}a''$ fow some suitable $n$, $m$, $a'$ and $a''$. 
It can be written as $\pm (f_e\cdot a')\cdot \cdots (f_o\cdot a')\cdot (f_o\cdot a'')$.\footnote{$\scr{S}\otimes \scr{C}(U_N)$ is the graded tensor product!} Then, if $g.(f_e\cdot a')-g'(f_e\cdot a')$'s are small enough, $\|g(f_i\otimes a_i)-g'(f_i\otimes a_i)\|<\varepsilon$.
Therefore we must prove the both of
\begin{center}
$g.\beta_{N,M}(f_e\otimes a)-g'.\beta_{N,M}(f_e\otimes a)\in \SC(U_M)$ and 
\end{center}
\begin{center}
$g.\beta_{N,M}(f_o\otimes a)-g'.\beta_{N,M}(f_o\otimes a)\in \SC(U_M)$
\end{center}
are arbitrary small uniformly in $M$, if $g$ is sufficiently close to $g'$. In fact, we need to estimate the limit of the norm of the above as $M$ tends to infinity, and so we may assume that $M$ is large enough so that $g,g'\in U_M$. Let us recall, from \cite{HKT}, that
$$\beta_{N,M}(f_e\otimes a)= f_e\otimes e^{-(C_{N+1}^M)^2}\otimes a\text{ and}$$
$$\beta_{N,M}(f_o\otimes a)= f_o\otimes e^{-(C_{N+1}^M)^2}\otimes a+
f_e\otimes C_N^Me^{-(C_N^M)^2}\otimes a.$$
Let us estimate the norm of the following functions of the variable $x\in U_M$:
$$ \bbbra{e^{-(C_{N+1}^M)^2}\otimes a}(x-g) -\bbbra{e^{-(C_{N+1}^M)^2}\otimes a}(x-g'),$$
$$ \bbbra{C_{N+1}^Me^{-(C_{N+1}^M)^2}\otimes a}(x-g) -\bbbra{C_{N+1}^Me^{-(C_{N+1}^M)^2}\otimes a}(x-g').$$
For $x=(x_1,y_1,x_2,y_2,\cdots,x_M,y_M)$, let 
$x|_N$ denote $(x_1,y_1,x_2,y_2,\cdots,x_N,y_N)\in U_N$. We use the coordinate $x=(x_1,y_1,x_2,y_2,\cdots,x_M,y_M)$, $g=(g_1,h_1,g_2,h_2,\cdots,g_M,h_M)$ and $g'=(g'_1,h'_1,g'_2,h'_2,\cdots,g'_M,h'_M)$.

\begin{align*}
&\left|\bbbra{e^{-(C_{N+1}^M)^2}\otimes a}(x-g) -\bbbra{e^{-(C_{N+1}^M)^2}\otimes a}(x-g')\right| \\
&=\left|
\exp\bra{-\sum_{N+1}^Mk^{-2l}\bbbra{(x_k-g_k)^2+(y_k-h_k)^2}}a(x-g)-\exp\bra{-\sum_{N+1}^Mk^{-2l}\bbbra{(x_k-g'_k)^2+(y_k-h'_k)^2}}\cdot a(x-g')\right|\\
&\leq \left| \exp\bra{-\sum_{N+1}^Mk^{-2l}\bbbra{(x_k-g_k)^2+(y_k-h_k)^2}}\right|\cdot \left|a\bra{x-g|_N}-a\bra{x-g'|_N}\right| \\
&\ \ \ +\left|
\exp\bra{-\sum_{N+1}^Mk^{-2l}\bbbra{(x_k-g_k)^2+(y_k-h_k)^2}}-\exp\bra{-\sum_{N+1}^Mk^{-2l}\bbbra{(x_k-g'_k)^2+(y_k-h'_k)^2}}\right||a(x-g')|\\
&\leq \sup_z |{\rm grad}(a)(z)|\cdot\|g-g'|_N\|_{l^2}\\
&\ \ \ 
+\left|
\exp\bra{-\sum_{N+1}^Mk^{-2l}\bbbra{(x_k-g_k)^2+(y_k-h_k)^2}}-\exp\bra{-\sum_{N+1}^Mk^{-2l}\bbbra{(x_k-g'_k)^2+(y_k-h'_k)^2}}\right||a(x-g')|.
\end{align*}

We need to estimate $|\exp(\bullet)-\exp(\bullet')|$. For this purpose, let ${\bf x}_2:=((N+1)^{-l}x_{N+1},(N+1)^{-l}y_{N+1},\cdots,M^{-l}x_M,M^{-l}y_M)$ and similarly for $g$ and $g'$. By abuse of notation, let $f_e(x_{N+1},y_{N+1},\cdots,x_M,y_M)$ be $e^{-\sum_{k=N+1}^M(x_k^2+y_k^2)}$ until the end of this proof. Then, $f_e(C_{N+1}^M)(x_{N+1},y_{N+1},\cdots,x_M,y_M)=f_e({\bf x}_2)$.

Let us continue the estimate. Thanks to the mean-value theorem (note that $f_e$ is real-valued),
\begin{align*}
|f_e({\bf x}_2-{\bf g}_2)-f_e({\bf x}_2-{\bf g}'_2)| &\leq \sup_{w_2}|{\rm grad}(f_e)(w_2)|\|({\bf x}_2-{\bf g}_2)-({\bf x}_2-{\bf g}'_2)\|,
\end{align*}
where $w_2=(u_{N+1},v_{N+1},u_{N+2},v_{N+2},\cdots,u_{M},v_{M})$ runs over $U_M\ominus U_N$. Since ${\rm grad}(f_e)$ is bounded, we obtain the desired result which is independent of $M$.

For the $f_o$-part, we need more subtle argument. In the estimate for $f_e$, we divide the problem into two parts: the $U_N$-part and the $U_M\ominus U_N$-part. We use the same technique, and we may assume $N=0$. What we would like to compute is the norm of
$$\sum k^{-l}(x_k-g_k)e^{-\sum_{k=1}^Nb^{-2l}\bbbra{(x_b-g_b)^2+(y_b-h_b)^2}}\otimes e_k-\sum k^{-l}(x_k-g_k')e^{-\sum_{k=1}^Nb^{-2l}\bbbra{(x_b-g_b')^2+(y_b-h_b')^2}}\otimes e_k$$
$$
+\sum k^{-l}(y_k-h_k)e^{-\sum_{k=1}^Nb^{-2l}\bbbra{(x_b-g_b)^2+(y_b-h_b)^2}}\otimes f_k-\sum k^{-l}(y_k-h_k')e^{-\sum_{k=1}^Nb^{-2l}\bbbra{(x_b-g_b')^2+(y_b-h_b')^2}}\otimes f_k.$$
For simplicity, we deal with only the former half. Put 
\begin{center}
$r^2:=\sum_{k=1}^Nb^{-2l}\bbbra{(x_b-g_b)^2+(y_b-h_b)^2}$ and $(r')^2:=\sum_{k=1}^Nb^{-2l}\bbbra{(x_b-g'_b)^2+(y_b-h'_b)^2}$. 
\end{center}
Then, the square of the pointwise norm is given as follows:
\begin{align*}
\left\|\sum k^{-l}(x_k-g_k)e^{-r^2}\otimes e_k-\sum k^{-l}(x_k-g_k)e^{-(r')^2}\otimes e_k\right\|^2 &= \sum k^{-2l}\left\|(x_k-g_k)e^{-r^2}-(x_k-g'_k)e^{-(r')^2}\right\|^2
\end{align*}
thanks to the $C^*$-condition and the formula in the Clifford algebra. 
What we need to prove is the following: The sup norm of the function $x\mapsto \left\|\sum k^{-l}(x_k-g_k)e^{-r^2}\otimes e_k-\sum k^{-l}(x_k-g_k)e^{-(r')^2}\otimes e_k\right\|^2$ is arbitrary small if $g'$ is sufficiently close to $g$.

\begin{align*}
& \sum k^{-2l}\|(x_k-g_k)e^{-r^2}-(x_k-g'_k)e^{-(r')^2}\|^2 \\
&= \sum k^{-2l}\|(x_k-g_k)e^{-r^2}-(x_k-g'_k)e^{-r^2}+(x_k-g'_k)e^{-r^2}-(x_k-g'_k)e^{-(r')^2}\|^2 \\ 
&\leq 2\sum k^{-2l}|g_k'-g_k|^2e^{-2r^2} + 2\sum k^{-2l}|x_k-g'_k|^2\cdot |e^{-r^2}-e^{-(r')^2}|^2 \\
&\leq 2\|g-g'\|_{L^2_{-l}}e^{-2r^2} + 2(r')^2\cdot |e^{-r^2}-e^{-(r')^2}|^2. 
\end{align*}

The first term is, clearly, arbitrary small if $g$ and $g'$ are sufficiently close. For the second term, we need more computations. Let $|\bullet|$ be the ``$L^2_{-l}$-norm''.
Note that $|r-r'|\leq |g-g'|$. Put $\delta:= |g-g'|$, and note that $ r+r'\leq 2r'+\delta $. Then, 
\begin{align*}
(r')^2\cdot |e^{-r^2}-e^{-(r')^2}|^2 &=(r')^2e^{-2(r')^2}|e^{(r')^2-r^2}-1|^2 \\
&\leq \begin{cases} (r')^2e^{-2(r')^2}(e^{(2r'+\delta)\delta}-1)^2 & r'\geq r \\
(r')^2e^{-2(r')^2}(1-e^{-(2r'+\delta)\delta})^2 & r'\leq r. \end{cases}
\end{align*}
The both are arbitrary small if $\delta$ is sufficiently small. We compute only the square root of the former one. Notice the following:
\begin{align*}
r'e^{-(r')^2}(e^{(2r'+\delta)\delta}-1) &=(r'-\delta)e^{-(r'-\delta)^2}\cdot e^{2\delta^2}-r'e^{-(r')^2}+\delta e^{-(r'-\delta)^2}\cdot e^{2\delta^2}
.\end{align*}
Since the function $t\mapsto te^{-t^2}$ is uniformly continuous and bounded, the last quantity can be arbitrary small if $\delta$ is sufficiently small.

\end{proof}

\begin{rmks}
$(1)$ The above action is, in fact, continuous with respect to the $L^2_{-l}$-metric. This fact might suggest that our HKT algebra is in some sense ``$\SC(U_{L^2_{-l}})$'' or something.

$(2)$ We will frequently encounter similar arguments in the present paper: Firstly, the corresponding statement on finite-dimensional objects is obvious. Secondly, to prove the statement for infinite-dimensional objects, we discuss some quantitative issues.
\end{rmks}

We would like to define a representation of $\SC(U)$ on a Hilbert $\scr{S}$-module. Seeing Lemma \ref{modified Kas modules} and Definition \ref{dfn of Dirac element for X}, we find that the ``natural'' representation of ``$\scr{C}(U)$'' should be on ``$L^2(U,\Cl_+(T^*U))$''.

Let us begin with the construction of the Spinor $S$.
Recall the complex basis of $\Lie(U)^*\otimes\bb{C}$: $z_k:=\frac{1}{\sqrt{2}}(e_k+if_k)$ and $\overline{z_k}:=\frac{1}{\sqrt{2}}(e_k-if_k)$.

\begin{dfn}[\cite{FHTII}]
Let $S_\fin:=\wedge^{\alg}\oplus^\alg_k\bb{C}\overline{z_k}$ be the algebraic exterior algebra of the anti-holomorphic part of $U_\fin\otimes \bb{C}$. It is naturally graded by the structure of the exterior algebra.
We introduce the Clifford multiplication by
$$\begin{cases}
\gamma(z_k):=-\sqrt{2}\overline{z_k}\rfloor \\
\gamma(\overline{z_k}):=\sqrt{2}\overline{z_k}\wedge.
\end{cases}$$
Let $S$ be the completion of $S_\fin$ with respect to the inner product induced by that of $\Lie(U)^*\otimes \bb{C}$.

$S_\fin$ has ``$1$'' as a unit vector. In order to distinguish from ${\bf 1}_b\in \ud{L^2(\bb{R}^\infty)}$, we write this specific element of $S_\fin$ as ${\bf 1}_f$, where ``$f$'' comes from ``fermion''. 
\end{dfn}

We have explained the dual Spinor is automatically defied by $S$. However, we will use a dense subspace of it in order to define a ``differential'' operator, and hence we should concretely describe it.

\begin{dfn}
Let $S_\fin^*:=\wedge^{\alg}\oplus^\alg_k\bb{C}z_k$, and let 
$$\begin{cases}
\gamma^*(z_k):=-\sqrt{2}z_k\wedge\circ \epsilon_{S^*} \\
\gamma^*(\overline{z_k}):=\sqrt{2}z_k\rfloor\circ \epsilon_{S^*}.
\end{cases}$$
As usual, $\epsilon_{S^*}$ is the grading homomorphism $\epsilon_{S^*}(z_{i_1}\wedge z_{i_2}\wedge \cdots \wedge z_{i_n})=(-1)^nz_{i_1}\wedge z_{i_2}\wedge \cdots \wedge z_{i_n}$.
The specific vector $1$ is also denoted by ${\bf 1}_f^*$.

The bilinear pairing between $S_\fin$ and $S^*_\fin$ is given by 
$$\innpro{z_{i_1}\wedge z_{i_2}\wedge \cdots \wedge z_{i_n}}{\overline{z_{j_1}}\wedge \overline{z_{j_2}}\wedge \cdots \wedge \overline{z_{j_n}}}{}=\delta_{i_1,j_1}\delta_{i_2,j_2}\cdots \delta_{i_n,j_n}$$ 
for $i_1<i_2<\cdots <i_n$ and $j_1<j_2<\cdots < j_n$, and 
$$\innpro{z_{i_1}\wedge z_{i_2}\wedge \cdots \wedge z_{i_n}}{\overline{z_{j_1}}\wedge \overline{z_{j_2}}\wedge \cdots \wedge \overline{z_{j_m}}}{}=0$$
for $n\neq m$.
\end{dfn}

We will introduce the ``mixed'' actions of the Clifford algebras on $S^*\otimes S$.

\begin{dfn}\label{mixed Clifford action}
On $S^*_\fin\otimes^\alg S_\fin$, define two Clifford multiplications for $X\in \Lie(U)$ by
$$c(X):=\frac{1}{\sqrt{2}}\bra{\id\otimes\gamma(X)-i\gamma^*(X)\otimes \id};$$
$$c^*(X):=\frac{1}{\sqrt{2}}\bra{\gamma^*(X)\otimes \id-i\id\otimes\gamma(X)}.$$
\end{dfn}

Note that $c$ is a homomorphism from $\Cl_-(\Lie(U))$, while $c^*$ is a homomorphism from $\Cl_+(\Lie(U))$. One can easily check that $c$ anti-commutes with $c^*$. The Clifford multiplication $c$ will appear as the ``symbol'' of our Dirac operator, and $c^*$ will appear as the ``potential''.

We will sometimes encounter Spinors of subspaces. $S_{V}$ and $S^*_V$ denotes the Spinor of a subspace $V\subseteq U$ corresponding to $S$ and $S^*$, respectively. In the following, as a subspace, $U_N$, $U_M\ominus U_N$, $U_N^\perp =U\ominus U_N$ will appear.

In order to define the representation, we need to introduce a Hilbert space which plays a role of ``$L^2(U,\Cl_+(T^*U))$''. Recall that $\Cl_+\cong \End(S^*)$ in finite-dimensional setting.

\begin{dfn}
Let $\ca{H}$ be the Hilbert space tensor product
$$\ud{L^2(U)}\otimes S^*\otimes S.$$
We sometimes use its dense subspace
$$\ca{H}_\fin:=\ud{L^2(U)_\fin}\otimes^\alg S^*_\fin\otimes^\alg S_\fin.$$

The specific vector $\vac\otimes {\bf 1}_f^*\otimes {\bf 1}_f$ is denoted by $\Vac$. 
\end{dfn}

The actual representation space is a Hilbert $\scr{S}$-module $\scr{S}\otimes \ca{H}$. We also deal with a dense subspace $\scr{S}_\fin\otimes^\alg \ca{H}_\fin$.


We would like to define a $*$-homomorphism $\Phi:\SC(U)\to \ca{L}_{\scr{S}}(\scr{S}\otimes \ca{H})$. 
For this purpose, it is essential to define an operator $\Phi(C_N)$'s corresponding to the Clifford element $C_N$'s. 
If such an operator is defined appropriately, we can define the representation. Let us briefly explain it here before the definition of $\Phi(C_N)$'s. Firstly, $\scr{C}(U_N)$ acts on $L^2(U_N)\otimes S^*_{U_N}\otimes S_{U_N}$: For $f\in C_0(U_N)$ and $v\in \Cl_+(U_N)$, $[f\otimes v]\cdot (\phi\otimes s\otimes s'):= (f\cdot \phi)\otimes c^*(v)[s\otimes s']$.
This action is denoted by $a\cdot v$ for $a\in \scr{C}(U_N)$ and $v\in L^2(U_N)\otimes S^*_{U_N}\otimes S_{U_N}$.
Then, $\Phi_N:\SC(U_N)\to \ca{L}_{\scr{S}}(\scr{S}\otimes \ca{H})$ is defined by $\Phi_N(f\otimes h):= f(X\otimes \id+\id\otimes \Phi(C_{N+1}))\otimes h$ through $\scr{S}\otimes \ud{L^2(U)}\otimes S^*\otimes S\cong \bbbra{\scr{S}\otimes \ud{L^2(U_N^\perp)}\otimes S^*_{U_N^\perp}\otimes S_{U_N^\perp}}\bigotimes \bbbra{L^2(U_N)\otimes S^*_{U_N}\otimes S_{U_N}}$.
After the following proposition, we will check that $\Phi_N$'s are compatible with $\beta_{\bullet,\bullet}$'s, and hence $\Phi:\SC(U)\to \ca{L}_\scr{S}(\scr{S}\otimes \ca{H})$ is well-defined.

In order to define the operator $\Phi(C_N)$, notice that the multiplication operator of the function $x_k$ is defined on $L^2(U_M)_\fin$, for any $M\geq k$. This operator commutes with $\beta_{M,L}$, and hence acts on $\ud{L^2(U)}$. In the representation theoretical language, it is defined by $i[dL(f_k)+dR(f_k)]$. Moreover, $e_k$ acts on $S^*\otimes S$ by $\frac{1}{\sqrt{2}}\bbbra{c^*(z_k)+c^*(\overline{z_k})}
=\frac{1}{2}[\gamma^*(z_k)\otimes \id-i\id\otimes \gamma(z_k)+\gamma^*(\overline{z_k})\otimes \id-i\id\otimes \gamma(\overline{z_k})]$. Consequently, each summand $x_ke_k+y_kf_k$ is defined. The non-trivial thing is whether the infinite sum converges or not in an appropriate sense. Note that each term $x_k\otimes e_k+y_k\otimes f_k$ is an unbounded operator.

\begin{pro}
The infinite sum of operators $C_N:=\sum_{k=N}^\infty k^{-l}(x_k\otimes c^*(e_k)+y_k\otimes c^*(f_k))$ defines an essentially self-adjoint operator from $\ud{L^2(U)_\fin}\otimes_\alg S_\fin\otimes_\alg S_\fin$ to $\ud{L^2(U)}\otimes S\otimes S$.
\end{pro}
\begin{proof}
Compare with \cite{T2} for the technique of the proof.

Since each summand $x_k\otimes c^*(e_k)+y_k\otimes c^*(f_k)$ is symmetric operator, so is the infinite sum, provided the infinite sum converges. Thus it is closable. To be essentially self-adjoint, we need to check that $\Phi(C_N)\pm i\id$ has dense range. Therefore we need two steps; $(1)$ The infinite sum $\Phi(C_N)$ converges in the strong sense; $(2)$ $\Phi(C_N)\pm i\id$ has dense range.

$(1)$ Fix an element $\phi\otimes s\otimes s'\in \ud{L^2(U)_\fin}\otimes_\alg S^*_\fin\otimes_\alg S_\fin$. We study the infinite sum $\sum_{k=N}^\infty k^{-l}(x_k\otimes c^*(e_k)+y_k\otimes c^*(f_k))(\phi\otimes s\otimes s')$.  By an abuse of notation, $\Phi(C_N)(\phi\otimes s)$ denotes this infinite sum. 
By definition of $\ud{L^2(U)_\fin}\otimes_\alg S_\fin^* \otimes_\alg S_\fin$, the vector $\phi\otimes s\otimes s'$ is a finite linear combination of $(\widetilde{\phi}\otimes\vac_M^\infty)\otimes (z_{j_1}\wedge z_{j_2}\wedge\cdots \wedge z_{j_m})\otimes (\overline{z_{i_1}}\wedge \overline{z_{i_2}}\wedge\cdots \wedge\overline{z_{i_n}})$'s for some $\widetilde{\phi}\in L^2(U_{M-1})_\fin$'s. We may choose $M$ so as to be greater than $i_n$ and $j_m$. We may write the vector as $[\widetilde{\phi}\otimes (z_{j_1}\wedge z_{j_2}\wedge\cdots \wedge z_{j_m})\otimes(\overline{z_{i_1}}\wedge \overline{z_{i_2}}\wedge\cdots \wedge\overline{z_{i_n}})]\otimes \Vac_M^\infty$. Since $\Phi(C_N^{M-1})$ is clearly a finite sum (it is unbounded, but it is well-defined on rapidly decreasing functions), and since $C_N=C_N^{M-1}+C_M$, it suffices to check that $\Phi(C_M)(\Vac_M^\infty)$ converges. We may assume that $M=1$.

On each piece $U_{N+1}\ominus U_N\cong \bb{R}^2$, the vacuum vector is $\frac{1}{\sqrt{2\pi}}e^{-\frac{1}{4}(x^2+y^2)}$. With the Gaussian integral formulas, we find that $\|\frac{1}{\sqrt{2\pi}}xe^{-\frac{1}{4}(x^2+y^2)}\|_{L^2}=1$. Thus
$$\left\|\sum_{k=1}^L\bra{k^{-l}x_k\otimes c^*(e_k)\Vac+k^{-l}y_k\otimes c^*(f_k)\Vac}\right\|^2=
\sum_{k=1}^L2k^{-2l}.$$
The limit as $L\to \infty$ exists, because $l$ is greater than $\frac{3}{2}$ (it is enough to suppose that $l>\frac{1}{2}$).

$(2)$ To see that $\Phi(C_L)\pm i\id$ has dense range, fix an element $\bra{\widetilde{\phi}\otimes s\otimes s'}\otimes \Vac_N^\infty\in \ud{L^2(U)_\fin}\otimes^\alg S^*_\fin\otimes^\alg S_\fin$. We need to find $v_\pm\otimes \Vac_M^\infty\in \ud{L^2(U)_\fin}\otimes^\alg S^*_\fin\otimes^\alg S_\fin$ such that $(\Phi(C_L)\pm i\id)(v_\pm\otimes \Vac_M^\infty)$ is close enough to $\bra{\widetilde{\phi}\otimes s\otimes s'}\otimes \Vac_N^\infty$. Take any real number $\varepsilon>0$. Since the infinite sum $\sum_k2k^{-2l}$ converges, we can choose $M$ satisfying that $\|\Phi(C_{M+1})(\Vac_{M+1}^\infty)\|^2=\sum_{k\geq M+1}2k^{-2l}< \varepsilon^2$. We may assume that $M$ is greater than $N$ and $L$.
Since $(\Phi(C_L^M)+i\id)(\Phi(C_L^M)-i\id)=(\Phi(C_L^M)-i\id)(\Phi(C_L^M)+i\id)=\sum_kk^{-2l}(x_k^2+y_k^2)+1$, we can find $v'_\pm\in L^2(U_M)\otimes S^*_{U_M}$ such that $(\Phi(C_L^M)\pm i\id)v'_\pm\otimes \Vac_{M+1}^\infty=\widetilde{\phi}\otimes s\otimes s'$. 
Since $L^2_\fin$ is dense in $L^2$, and since $\Phi(C_L)\pm i$ does not reduce the norm, we can find $v_\pm\in L^2(U_M)_\fin\otimes S_{U_M}^*\otimes S_{U_M}$ such that $\|(\Phi(C_L^M\pm i)v_\pm-(\Phi(C_L^M)\pm i)v'_\pm\|< \varepsilon.$
Therefore, $\|(\Phi(C_L)\pm i)v_{\pm}-\widetilde{\phi}\otimes s\otimes s'\otimes \Vac_L^\infty\|$ is less than $2\varepsilon$.

\end{proof}

This lemma enables us to define a $*$-homomorphism $\SC(U_N)\to \ca{L}_{\scr{S}}(\scr{S}\otimes\ca{H})$ as follows.

\begin{lem}
The following diagram commutes.
$$\xymatrix{
\SC(U_N) \ar[rr]^{\beta_{N,M}} \ar[rd]_{\Phi_N} && \SC(U_M) \ar[dl]^{\Phi_M} \\
& \ca{L}_{\scr{S}}(\scr{S}\otimes\ca{H}) &
}$$
\end{lem}

\begin{proof}
We may check the statement only for generators: $f_e\otimes a$ and $f_o\otimes a$. Recall that $\beta_{N,M}(f_e\otimes a)= f_e\otimes f_e(C_{N+1}^M)\otimes a$.
\begin{align*}
\Phi_M\circ \beta_{N,M}(f_e\otimes a) &= f_e(X\otimes \id+\id\otimes \Phi(C_{M+1}^\infty))\otimes e^{-(C_{N+1}^M)^2}\otimes a \\
&=f_e(X)\otimes e^{-\Phi(C_{M+1}^\infty)^2}\otimes e^{-(C_{N+1}^M)^2}\otimes a.
\end{align*}
On the other hand, $\Phi_N(f_e\otimes a)$ is given by $f_e(X\otimes \id+\id\otimes \Phi(C_{N+1}^\infty))\otimes a=f_e(X)\otimes e^{-\Phi(C_N^\infty)^2}\otimes a$.
Since $\Phi(C_{N+1}^\infty)=C_{N+1}^M +\Phi(C_{M+1}^\infty)$, $\Phi(C_{N+1}^\infty)^2=[C_{N+1}^M]^2 +\Phi(C_{M+1}^\infty)^2$, and $[C_{N+1}^M]^2$ commutes with $\Phi(C_{M+1}^\infty)^2$,
$$f_e(X)\otimes e^{-\Phi(C_N^\infty)^2}\otimes a=f_e(X)\otimes e^{-\Phi(C_{M+1}^\infty)^2}\otimes e^{-(C_{N+1}^M)^2}\otimes a.$$


For $f_o\otimes a$, use the parallel technique.

\end{proof}

Consequently, we obtain the following $*$-homomorphism.

\begin{dfn-thm}
There uniquely exists a $*$-homomorphism $\Phi:\SC(U)\to \ca{L}_{\scr{S}}(\scr{S}\otimes\ca{H})$ such that the following diagram commutes for all $N$:
$$\xymatrix{
\SC(U_N) \ar[rr]^{\beta_{N}} \ar[rd]_{\Phi_N} && \SC(U) \ar[dl]^{\Phi} \\
& \ca{L}_{\scr{S}}(\scr{S}\otimes\ca{H}) &
}$$

\end{dfn-thm}

From now on, we understand $\SC(U)$ acting on $\scr{S}\otimes \ca{H}$ always by $\Phi$, and we write$\Phi(a)(\phi)$ as $a\cdot \phi$ for $a\in \SC(U)$ and $\phi\in \scr{S}\otimes \ca{H}$.


\section{Three Kasparov modules and the index theorem}

This section is the main part of the present paper. We will introduce three unbounded $U$-equivariant Kasparov modules. One of them is untwisted, and the others are $\tau$-twisted. These modules correspond to the index element, the Dirac element and the Clifford symbol element. Then, we will prove the index theorem equality.

\subsection{Index element $[\widetilde{\Dirac}]$}

Recall that $\ud{L^2(U)}$ is defined by using the $\tau$-twisted representation theory of $U$. Thus $U$ naturally acts on $\scr{S}\otimes \ca{H}$ by
$$g.[f\otimes v\otimes v'\otimes s'\otimes s]:=f\otimes (g.v)\otimes v'\otimes s'\otimes s$$
with cocycle $\tau$. This action is modeled on the left regular representation.
Since $U$ preserves the inner product of $\ud{L^2(\bb{R}^\infty)}$ and does not touch the $\scr{S}$-part, this action preserves the $\scr{S}$-valued inner product, and commutes with the right $\scr{S}$-action.

We firstly prove that the left $\SC(U)$-module $\scr{S}\otimes \ca{H}$ is $\tau$-twisted $U$-equivariant.

\begin{lem}
For $g\in U$, $a\in \SC(U)$ and $\phi\in \scr{S}\otimes \ca{H}$,
$$g.(a\cdot \phi)=(g.a)\cdot (g.\phi).$$
\end{lem}

\begin{proof}
This equality seems to be obvious at first sight, because of the following (illegal!\footnote{The value at $x$ of $g.(a\cdot \phi)$, $g.(a\cdot \phi)(x)$, does not make sense.}) computation: $g.(a\cdot \phi)(x)= (a\cdot \phi)(g^{-1}.x)= a(g^{-1}.x)\cdot \phi (g^{-1}.x)=(g.a)\cdot (g.\phi)(x)$. However, we cannot use the language of functions without any tricks, and our problem is less clear than the first impression. 
Briefly speaking, we prove this equality by the continuity of the action, and the finite-dimensional approximation.

We need to check that
$$\|g.(a\cdot\phi)-(g.a)\cdot(g.\phi)\|< \varepsilon$$
for any $\varepsilon>0$.
If $g$, $a$ and $\phi$ come from finite-dimensional objects, the inequality is clear. For example, if $g\in U_N$, 
$a=f_1(X\otimes \id+\id\otimes C_{N+1}^\infty)\otimes \widetilde{a}$ and $\phi =f_2\otimes \Vac_{N+1}^\infty\otimes  \widetilde{\phi}\otimes s$, 
$$\left\|g.(a\cdot\phi)-(g.a)\cdot(g.\phi)\right\|=\left\|f_1(X\otimes\id+\id\otimes C_{N+1}^\infty)\cdot(f_2\otimes \Vac_{N+1}^\infty)\otimes\bbbra{ g.(\widetilde{a}\cdot \widetilde{\phi})-(g.\widetilde{a})\cdot (g.\widetilde{\phi})}
\otimes s\right\|,$$
which is clearly zero, because $g.(\widetilde{a}\cdot \widetilde{\phi})-(g.\widetilde{a})\cdot (g.\widetilde{\phi})$ is computed in $\scr{C}(U_N)$ as a function.

Let us go back to the general case.
It suffices to replace general $g$, $a$ and $\phi$ with such (better) ones.
Recall the following things: The $U$-action on $\ca{H}$ is unitary; The $U$-action on $\SC(U)$ is isometry; and $\Phi(a)$ is a bounded operator whose norm is bounded by $\|a\|$. Then,
$$\|g.(a\cdot\phi)-g.(a'\cdot\phi')\|=\|(a\cdot\phi)-(a'\cdot\phi')\|\leq \|a-a'\|\|\phi\|+\|a\|\|\phi-\phi'\|,$$
$$\|(g.a)\cdot(g.\phi)-(g.a')\cdot(g.\phi')\| \leq \|g.a-g.a'\|\|\phi\|+\|g.a'\|\|g.\phi-g.\phi'\|=\|a-a'\|\|\phi\|+\|a'\|\|\phi-\phi'\|
.$$
Thus we can find $a'\in \SC(U_N)_\fin$ and $\phi'=\widetilde{\phi'}\otimes s\otimes\Vac_N^\infty\in \bbbra{ L^2(U_N)\otimes^\alg S^*_{U_N}\otimes S_{U_N}} \otimes^\alg [\ud{L^2(U\ominus U_N)_\fin}\otimes^\alg (S^*_{U\ominus U_N})_\fin\otimes^\alg (S_{U_N})_\fin]$, such that the both of $\|g.(a\cdot\phi)-g.(a'\cdot\phi')\|$ and $\|(g.a)\cdot(g.\phi)-(g.a')\cdot(g.\phi')\|$ are smaller than $\varepsilon/4$, independently from $g$. Moreover, we can choose $g'\in U_M$ such that the both of $\|g.(a'\cdot\phi')-g'.(a'\cdot\phi')\|$ and $\|(g.a')\cdot(g.\phi')-(g'.a')\cdot(g'.\phi')\|$ are smaller than $\varepsilon/4$ once we fix $a'$ and $\phi'$, since the action is strongly continuous.\footnote{Note that we cannot choose the neighborhood of $g$ uniformly in $a$ or $\phi$. This is because the group action is not uniformly continuous, but strongly continuous.} Thanks to the homomorphism $\beta_{N,M}$ or the inclusion $U_M\hookrightarrow U_N$, we may assume that $M=N$. Now, the statement is clear by the finite-dimensional case.

\end{proof}

Thanks to this lemma, we obtain a $\tau$-twisted $U$-equivariant Hilbert $\scr{S}$-module $\scr{S}\otimes \ca{H}$ equipped with a left $\SC(U)$-module structure.
We must define an appropriate operator in order to define a Kasparov module.
We introduce two operators; one of them is $U$-equivariant with respect to the action $R$, and the other is $U$-equivariant with respect to the action $L$. The both look natural, but the operator we should deal with is the ``average'' of them. It is not actually equivariant, but almost equivariant such that the operator defines an equivariant Kasparov module. In fact, our operator is essentially the perturbed Bott-Dirac operator in \cite{HK}.

\begin{dfn}
On $\ud{L^2(\bb{R}^\infty)_\fin}\otimes^\alg \ud{L^2(\bb{R}^\infty)_\fin^*}\otimes^\alg S^*_\fin \otimes^\alg \otimes S_\fin$, let

$${}^L\Dirac:=\sum_k \sqrt{k}\bra{d\rho(z_k)\otimes \id\otimes \gamma^*(\overline{z_k})\otimes \id+d\rho(\overline{z_k})\otimes \id\otimes \gamma^*(z_k)\otimes \id},$$
$${}^R\Dirac:=\sum_k \sqrt{k}\bra{\id\otimes d\rho^*(z_k)\otimes \id\otimes \gamma(\overline{z_k})+\id\otimes d\rho(\overline{z_k})\otimes \id\otimes \gamma(z_k)}.$$

As an unbounded operator on $L^2(U_M\ominus U_{N-1})\otimes S_{U_M\ominus U_{N-1}}^*\otimes S_{U_M\ominus U_{N-1}}$, define 
$${}^L\Dirac_N^M:=\sum_{k=N}^M \sqrt{k}\bra{d\rho(z_k)\otimes \id\otimes \gamma^*(\overline{z_k})\otimes \id+d\rho(\overline{z_k})\otimes \id\otimes \gamma^*(z_k)\otimes \id}$$
for $1\leq N\leq M\leq \infty$, and similarly for ${}^R\Dirac$. Note that $M=\infty$ is allowed. They are defined on, at least, $L^2(U_M\ominus U_{N-1})_\fin \otimes^\alg S_{U_M\ominus U_{N-1},\fin}^*\otimes^\alg S_{U_M\ominus U_{N-1},\fin}$
\end{dfn}

These are well-defined as proved in \cite{FHTII}. For the details, see also Proposition 5.30 in \cite{T3}. Clearly ${}^L\Dirac$ is $U$-equivariant with respect to the action $R$, and ${}^R\Dirac$ is $U$-equivariant with respect to the action $L$. The operator we will study is the following. Note that ${}^L\Dirac$ is skew-symmetric, while ${}^R\Dirac$ is symmetric.

\begin{dfn}
On $\ud{L^2(\bb{R}^\infty)_\fin}\otimes^\alg \ud{L^2(\bb{R}^\infty)_\fin^*}\otimes^\alg S^*_\fin \otimes^\alg \otimes S_\fin$, define an operator $\Dirac$ by
$$\Dirac:=\frac{1}{\sqrt{2}}{}^R\Dirac +\frac{i}{\sqrt{2}}{}^L\Dirac.$$

Put $\Dirac_N^M:=\frac{1}{\sqrt{2}}{}^R\Dirac_N^M +\frac{i}{\sqrt{2}}{}^L\Dirac_N^M$.
\end{dfn}

\begin{rmk}
In the previous papers \cite{T1,T2,T3}, we essentially studied only ${}^R\Dirac$, although an operator which looks like ${}^L\Dirac$ appeared in some related $KK$-elements. Thus we wanted to prove that the pair $[(\scr{S}\otimes \ca{H},{}^R\Dirac)]$ is a $\tau$-twisted $U$-equivariant Kasparov $(\SC(U),\scr{S})$-module. However, it turns out to be false: the operator is not with locally compact resolvent.
That is essentially because each element of our $\SC(U)$ is too ``smooth'' to work as a ``compactly supported function''.

In order to overcome this problem, we modify the operator. The new one is, from the beginning, with compact resolvent. The cost we must pay is giving up the actual equivariance. It made us study the new equivariance condition on equivariant Kasparov modules in Section 2.1.
\end{rmk}

\begin{rmk}\label{our Dirac is Bott-Dirac}
Let us formally rewrite our Dirac operator using the function language, seeing Definition \ref{mixed Clifford action}.
Recall $dL(e_k)=-\partial_{x_k}+\frac{i}{2}y_k$ and so on. Firstly, note that
$$dL(z_k)\otimes \gamma^*(\overline{z_k})\otimes \id+dL(\overline{z_k})\otimes \gamma^*(z_k)\otimes \id
= dL(e_k)\otimes \gamma^*(e_k)\otimes \id+dL(f_k)\otimes \gamma^*(f_k)\otimes \id,$$
and similarly for ${}^R\Dirac$. Then, the $k$-th term of $\Dirac$ is given by
\begin{align*}
&\frac{i}{\sqrt{2}}\bra{dL(e_k)\otimes \gamma^*(e_k)\otimes \id+dL(f_k)\otimes \gamma^*(f_k)\otimes \id}+\frac{1}{\sqrt{2}}\bra{dR(e_k)\otimes \id\otimes \gamma(e_k)+dR(f_k)\otimes \id\otimes \gamma(f_k)} \\
&\ \ \ =\frac{i}{\sqrt{2}}\bra{\bra{-\partial_{x_k}+\frac{i}{2}y_k}\otimes \gamma^*(e_k)\otimes \id+\bra{-\partial_{y_k}-\frac{i}{2}x_k}\otimes \gamma^*(f_k)\otimes \id}\\
&\ \ \ \ \ \ +\frac{1}{\sqrt{2}}\bra{\bra{\partial_{x_k}+\frac{i}{2}y_k}\otimes \id\otimes \gamma(e_k)+\bra{\partial_{y_k}-\frac{i}{2}x_k}\otimes \id\otimes \gamma(f_k)}\\
&\ \ \ =\partial_{x_k}\otimes c(e_k)+\partial_{y_k}\otimes c(e_k)+\frac{x_k}{2}\otimes c^*(f_k)+\frac{y_k}{2}\otimes c^*(-e_k).
\end{align*}
By introducing an operator $J$ defined by $e\mapsto f$ and $f\mapsto -e$, we can formally rewrite the Dirac operator;
$$\Dirac=\sum\sqrt{k}\bbbra{
\partial_{x_k}\otimes c(e_k)+\partial_{y_k}\otimes c(e_k)+\frac{x_k}{2}\otimes c^*(J(e_k))+\frac{y_k}{2}\otimes c^*(J(e_k))}.$$
If we identify $c$ as the left multiplication of the Clifford algebra on itself, and $c^*\circ J$ as the right multiplication, $\Dirac$ looks like the Bott-Dirac operator.

Notice that this formula is illegal: We exchanged the order of operators infinitely many times; Even on finite energy vectors, the infinite sum does not converges.
\end{rmk}

\begin{rmk}
In order to reduce symbols, $\Dirac$ sometimes stands for the operator $\id\otimes \Dirac$ on $\scr{S}\otimes \ca{H}$.

\end{rmk}

We would like to prove that the pair $(\scr{S}\otimes \ca{H},\Dirac)$ is a $\tau$-twisted $U$-equivariant Kasparov module. For this purpose, we check the followings: $\Dirac$ is self-adjoint and regular; $\SC(U)_\fin$ preserves $\dom(\Dirac)$; $\SC(U)_\fin$-action commutes with $\Dirac$ modulo bounded operators; $(1+\Dirac^2)$ is compact; $U$ preserves $\dom(\Dirac)$;  and $\Dirac$ satisfies the conditions to be an equivariant Kasparov module.

\begin{lem}
$\Dirac$ is self-adjoint and regular on $\scr{S}\otimes \ca{H}$.
\end{lem}
\begin{proof}
It is easy to see that $\Dirac$ is symmetric.
It suffices to check that $\Dirac\pm i:\ca{H}\circlearrowright$ has dense range. In fact, $\Dirac\pm i$ give unbounded bijections on $\ca{H}_\fin$.
We concentrate on $\Dirac+i$ in order to simplify the notation. 
By exchanging the order, see $\ca{H}$ as $[\ud{L^2(\bb{R}^\infty)}\otimes S^*]\otimes[\ud{L^2(\bb{R}^\infty)^*}\otimes S]$. 
Then, $\Dirac$ can be written as $\frac{i}{\sqrt{2}}{}^L\Dirac\otimes \id+ \frac{1}{\sqrt{2}}\id\otimes {}^R\Dirac$.
Since ${}^L\Dirac$ and ${}^R\Dirac$ are anti-commutative (by definition of the graded tensor product), $\Dirac^2=\frac{1}{2}\bbbra{-\bra{{}^L\Dirac}^2\otimes \id+\id\otimes \bra{{}^R\Dirac}^2}$. Consider $(\Dirac+i)(\Dirac-i)=\Dirac^2+1$. Recall that $\bra{{}^L\Dirac}^2$ and $\bra{{}^R\Dirac}^2$ have discrete spectrum.\footnote{See the proof of Theorem 5.33 in \cite{T3} for the detail of the computation.} Take eigenvectors $v\in \ud{L^2(\bb{R}^\infty)}\otimes S^*$ and $v'\in \ud{L^2(\bb{R}^\infty)^*}\otimes S$ such that $\bra{{}^L\Dirac}^2(v)=\lambda v$ and $\bra{{}^R\Dirac}^2(v')=\lambda' v'$. Note that $\lambda\leq 0$ and $\lambda'\geq 0$.
Then,
$$[\Dirac^2+1](v\otimes v')=\bra{\frac{-\lambda+\lambda'}{2}+1}v\otimes v'.$$
Therefore, $(\Dirac+i)\bbbra{(\Dirac-i)\bra{\frac{2}{2-\lambda+\lambda'}(v\otimes v')}}=v\otimes v'$.
\end{proof}

It seems to be easy to prove that $[a,\Dirac]$ is bounded. This is because $\Dirac$ is a combination of first order derivatives and multiplication operators. The multiplication operator defined by ``smooth'' $a$ seems to commute with the multiplication with scalar-valued functions, and seems to commute, modulo bounded, with differential operator of first order. In fact, the formal computation of the commutator is not too difficult. We need only one trick to prove the convergence.
However, in order to {\bf define} the commutator, we must prove that $\Dirac\circ a$ is defined on some dense subspace, which is essential for this issue. This is because $\SC(U)_\fin$ does not preserve $\ca{H}_\fin$. Therefore we actually need to deal with, not $\ca{H}_\fin$, but $\dom(\Dirac)$. This is one of the most complicated and the most creative point of the present paper. Before that, we prepare a lemma to represent $e^{-C^2}\Vac$ explicitly 

\begin{lem}
$f(C_1^N)$ strongly converges to $f(C)$ on $\scr{S}_\fin\otimes^\alg\ca{H}_\fin$ as $N\to \infty$, for $f=f_e$ or $f_o$.
\end{lem}
\begin{proof}
This proof is inspired by the argument in Appendix A.4 in \cite{HKT}.

Define an orthogonal projection $P_k$ on $L^2(U_k\ominus U_{k-1})$ by the multiplication operator of the function given by
$$\chi_k(x,y):=\begin{cases} 1 & (x^2+y^2\leq k) \\ 0 & (x^2+y^2> k). \end{cases}$$
We would like to define an infinite tensor product 
$\bb{P}_M:=$``$\id\otimes \cdots\otimes \id\otimes P_M\otimes P_{M+1}\otimes \cdots$''.
For this aim, we define 
$$\bb{P}_M^N:=\id\otimes \cdots\otimes \id\otimes P_M\otimes P_{M+1}\otimes \cdots \otimes P_N\otimes \id\otimes \cdots,$$
and we will check that $\bb{P}_M^N$ strongly converges to some projection as $N\to \infty$. Then, we will define $\bb{P}_M$ by the the strong limit ``$\lim_{N\to \infty}\bb{P}_M^N$''.

Take $\phi\otimes s'\otimes s\in \ud{L^2(U)_\fin}\otimes^\alg S^*_\fin\otimes^\alg S_\fin$. We may assume that $\phi\in L^2(U_L)_\fin$, $s'\in S^*_{U_L}$ and $s\in S_{U_L}$. In order to prove that $\bb{P}_M^N[\phi\otimes s'\otimes s]$ converges as $N\to \infty$, we may assume that $N>L$. 
In order to simplify the notation, we assume that $M<L$. It is always possible, because of the embeddings $L^2(U_\bullet)\hookrightarrow L^2(U_{\bullet+1})$.
Then,
$$\bb{P}_M^N[\phi\otimes s'\otimes s]=\bb{P}_M^L[\phi\otimes s'\otimes s]\otimes \bb{P}_{L+1}^N\Vac_{L+1}^\infty.$$
Seeing this formula, we find that it suffices to prove that $\bb{P}_{L+1}^N[\Vac_{L+1}^\infty]$ converges as $N\to \infty$. For this aim, we may 
ignore the $U_L$-part. Let us regard $\bb{P}_{M}^N\Vac-\bb{P}_{M}^{N'}\Vac$ for $N<N'$, as
$$\bra{\bb{P}_{M}^N\Vac-\bb{P}_{M}^{N+1}\Vac}+\bra{\bb{P}_{M}^{N+1}\Vac-\bb{P}_{M}^{N+2}\Vac}+\cdots +\bra{\bb{P}_{M}^{N'-1}\Vac-\bb{P}_{M}^{N'}\Vac}.$$
Let us estimate $\|\bb{P}_{M}^{N+k}\Vac-\bb{P}_{M}^{N+k+1}\Vac\|$'s.
\begin{align*}
\|\bb{P}_{M}^{N+k}\Vac-\bb{P}_{M}^{N+k+1}\Vac\| &=
\|\bb{P}_{M}^{N+k}\Vac_{1}^{N+k}\| \cdot \|\Vac_{N+k}^{N+k}-P_{N+k}\Vac_{N+k}^{N+k}\| \cdot \|\Vac_{N+k+1}^\infty\| \\
&=\|\bb{P}_{M}^{N+k}\Vac_{1}^{N+k}\| \cdot 
\sqrt{\frac{1}{2\pi}\int_{r\geq \sqrt{N+k}}e^{-\frac{1}{2}r^2}dx_{N+k}dy_{N+k}} \\
&=\|\bb{P}_{M}^{N+k}\Vac_{1}^{N+k}\| \cdot e^{-\frac{N+k}{4}}.
\end{align*}
The coefficient $\|\bb{P}_{M}^{N+k}\Vac_{1}^{N+k}\|$ cannot be greater than $1$, because $\bb{P}$ is just an orthogonal projection. Therefore the sum we are estimating cannot exceed $\sum_{k=0}^{N'-N-1}e^{-\frac{N+k}{4}}\leq e^{-N/4}/(1-e^{-1/4})$. Thus the sequence $\bb{P}_M^N[\phi\otimes s'\otimes s]$ is a Cauchy sequence, and hence the limit $\bb{P}_M[\phi\otimes s'\otimes s]:=\lim_{N\to\infty}\bb{P}_M^N[\phi\otimes s'\otimes s]$ exists.

As the next step, we check that $\bb{P}_M$ strongly converges to $\id$ as $M\to \infty$ on $\ud{L^2(U)_\fin}\otimes^\alg S_\fin^*\otimes^\alg S_\fin$. In order to prove this, we may check that $\|\vac-\bb{P}_M\vac\|\to 0$ as $M\to \infty$. Note that
$$\|\vac-\bb{P}_M\vac\|\leq \|\vac-\bb{P}_M^N\vac\|+\|\bb{P}_M^N\vac-\bb{P}_M\vac\|.$$
For any $\varepsilon>0$, there exists $M$ satisfying the following: $\|\vac-\bb{P}_M^N\vac\|<\varepsilon/2$ for {\bf any} $N\geq M$, thanks to the estimate in the previous paragraph.
After that, we can find $N$ satisfying that $\|\bb{P}_M^N\vac-\bb{P}_M\vac\|<\varepsilon/2$. Consequently, $\bb{P}_M\to \id$ strongly, as $M\to \infty$.

Let us prove that $e^{-(C_1^N)^2}\to e^{-C^2}$ as $N\to \infty$. As usual, we may prove it for $\vac$.
\begin{align*}
&\|e^{-C^2}\vac - e^{-(C_1^N)^2}\vac\| \\
&\ \ \ \leq \|e^{-C^2}\bra{\vac-\bb{P}_M\vac}\| +
\|e^{-C^2}\bb{P}_M\vac-e^{-(C_1^N)^2}\bb{P}_M\vac\|+\|e^{-(C_1^N)^2}\bra{\bb{P}_M\vac-\vac}\| .
\end{align*}
One can choose $M$ such that the first and the third terms are small, because $\bb{P}_M\vac\to \vac$ as $M\to \infty$, and $e^{-C^2}$ and $e^{-(C_1^N)^2}$ are bounded operators whose norm is at most $1$. We would like to find $N$ such that the second term is small, and so we may assume that $N>M$. The second term can be rewritten as
$$\|e^{-(C_1^N)^2}\bb{P}_M^{N}\vac\|\cdot\|e^{-(C_{N+1}^\infty)^2}\bb{P}_{N+1}\vac_{N+1}^\infty-\bb{P}_{N+1}\vac_{N+1}^\infty\|.$$
$\|e^{-(C_1^N)^2}\bb{P}_M^{N}\vac\|$ is at most $1$. 
For $\|e^{-(C_{N+1}^\infty)^2}\bb{P}_{N+1}\vac_{N+1}^\infty-\bb{P}_{N+1}\vac_{N+1}^\infty\|$, we firstly recall that $C_{N+1}^\infty$ commutes with $\bb{P}_{N+1}$. Let $h$ be a function on $\bb{R}$ defined by $h(X):=e^{-X^2}-1$. Then,
$$e^{-(C_{N+1}^\infty)^2}\bb{P}_{N+1}-\bb{P}_{N+1}=h(C_{N+1}^\infty\bb{P}_{N+1})\bb{P}_{N+1}.$$
We prove that $C_{N+1}^\infty\bb{P}_{N+1}$ is a bounded operator whose operator norm is arbitrary small if $N$ is large enough.
On the $k$-th component, the operator norm of the multiplication operator by the function
$$(x_k\otimes e_k+y_k\otimes f_k)\chi_k$$
is precisely $\sqrt{k}$. Therefore, 
$$\|C_{N+1}^\infty\bb{P}_{N+1}\|\leq \sum_{k>N}k^{\frac{1}{2}-l}.$$
It means that $\|C_{N+1}^\infty\bb{P}_{N+1}\|$ can be arbitrary small when $N$ is large enough, since $l$ is greater than $\frac{3}{2}$. By the property of the operator calculus, $\|h(C_{N+1}^\infty\bb{P}_{N+1})\|\leq \sup_{|X|\leq \|C_{N+1}^\infty\bb{P}_{N+1}\|}h(X)$. Since $h$ is continuous and $h(0)=0$, we obtain the necessary estimate.

For $f_o$, we follow the same story. We would like to prove that $C_1^Ne^{-(C_1^N)^2}\to Ce^{-C^2}$ as $N\to \infty$ strongly. As usual, we may prove it for $\vac$. Firstly, for sufficiently large $M$,
\begin{align*}
&\|Ce^{-C^2}\vac - C_1^Ne^{-(C_1^N)^2}\vac\| \\
&\ \ \ \leq \|Ce^{-C^2}\bra{\vac-\bb{P}_M\vac}\| +
\|Ce^{-C^2}\bb{P}_M\vac-C_1^Ne^{-(C_1^N)^2}\bb{P}_M\vac\|+\|C_1^Ne^{-(C_1^N)^2}\bra{\bb{P}_M\vac-\vac}\| .
\end{align*}
It suffices to estimate the second term.
\begin{align*}
&\|Ce^{-C^2}\bb{P}_M\vac-C_1^Ne^{-(C_1^N)^2}\bb{P}_M\vac\| \\
&=
\left\|C_{N+1}^\infty e^{-C^2}\bb{P}_M\vac+
C_1^Ne^{-(C_1^N)^2}\bb{P}_M^N\vac\otimes\bbbra{e^{-(C_{N+1}^\infty)^2}\bb{P}_{N+1}\vac_{N+1}^\infty-\bb{P}_{N+1}\vac_{N+1}^\infty} \right\|\\
&\leq \|C_{N+1}^\infty e^{-C^2}\bb{P}_M\vac\|+
\|C_1^Ne^{-(C_1^N)^2}\bb{P}_M^N\vac\|\cdot \|e^{-(C_{N+1}^\infty)^2}\bb{P}_{N+1}\vac_{N+1}^\infty-\bb{P}_{N+1}\vac_{N+1}^\infty \|.
\end{align*}
We check that the both terms are small. 

For the second one, $\|C_1^Ne^{-(C_1^N)^2}\bb{P}_M^N\vac\|$ is bounded in $N$, because $C_1^Ne^{-(C_1^N)^2}$'s are given by the operator calculus of the same function $f_o$, and so they are bounded operators with the same operator norm. Moreover, $\|e^{-(C_{N+1}^\infty)^2}\bb{P}_{N+1}\vac_{N+1}^\infty-\bb{P}_{N+1}\vac_{N+1}^\infty \|$ is small if $N$ is large enough thanks to the argument for $f_e$. 

For $\|C_{N+1}^\infty e^{-C^2}\bb{P}_M\vac\|$, we rewrite it as $\|e^{-(C_1^N)^2}\bb{P}_M\vac\|\cdot \|C_{N+1}^\infty e^{-(C_{N+1}^\infty)^2}\bb{P}_M\vac\|$. The coefficient $\|e^{-(C_1^N)^2}\bb{P}_M^N\vac\|$ is bounded. Moreover, $\|C_{N+1}^\infty e^{-(C_{N+1}^\infty)^2}\bb{P}_{N+1}\vac\|$ can be rewritten as
$$f_o(C_{N+1}^\infty\bb{P}_{N+1})\vac.$$
We have proved the following thing: for any $\varepsilon>0$, if $N$ is large enough, $\|C_{N+1}^\infty\bb{P}_{N+1}\|_\op\leq \varepsilon$. Since $f_o(0)=0$ and $f_o$ is continuous, $f_o(C_{N+1}^\infty\bb{P}_{N+1})\vac$ is arbitrary small if $N$ is large enough.

\end{proof}

\begin{pro}
$\SC(U)_\fin$ give bounded operators on $\dom(\Dirac)$ with respect to the graph norm.

\end{pro}

\begin{proof}

It is enough to prove the statement for the generators $f_e\otimes \widetilde{a}$ and $f_o\otimes \widetilde{a}$. Let us consider
\begin{center}
$[f_e\otimes \widetilde{a}]\cdot f\otimes\phi\otimes s\otimes s'$ and $[f_o\otimes \widetilde{a}]\cdot f\otimes\phi\otimes s\otimes s'$ 
\end{center}
for given $ f\otimes\phi\otimes s\otimes s'\in \scr{S}_\fin\otimes^\alg
\ud{L^2(U)_\fin} \otimes^\alg S_\fin^*\otimes^\alg\otimes S_\fin$.
Suppose that $\widetilde{a}\in \scr{C}(U_N)_\fin$ and $\phi\otimes s\otimes s'\in L^2(U_N)_\fin\otimes^\alg S_{U_N}^*\otimes^\alg S_{U_N}$.


Let us recall that 
$$[f_e\otimes \widetilde{a}]\cdot  f\otimes\phi\otimes s\otimes s'=\pm f_e(X)f\otimes [f_e(C_{N+1}^\infty)\otimes \widetilde{a}]\cdot [\phi\otimes s\otimes s'].$$
Thus we may consider only $[f_e(C_{N+1}^\infty)\otimes \widetilde{a}]\cdot [\phi\otimes s\otimes s']$.
Thanks to the previous lemma, $[f_e(C_{N+1}^\infty)\otimes \widetilde{a}]\cdot [\phi\otimes s\otimes s']=\lim_{M\to\infty}[f_e(C_{N+1}^M)\otimes \widetilde{a}]\cdot [\phi\otimes s\otimes s']$. Since $[f_e(C_{N+1}^M)\otimes \widetilde{a}]\cdot [\phi\otimes s\otimes s']\in \dom(\Dirac)$, for our purpose, it suffices to check that $\Dirac\bra{[f_e(C_{N+1}^M)\otimes \widetilde{a}]\cdot [\phi\otimes s\otimes s']}$ converges. Let us verify that. Note that $\Dirac_{M+1}^\infty\bra{[f_e(C_{N+1}^M)\otimes \widetilde{a}]\cdot [\phi\otimes s\otimes s']}=0$.
\begin{align*}
&\Dirac\bra{[f_e(C_{N+1}^M)\otimes \widetilde{a}]\cdot [\phi\otimes s\otimes s']}\\
&= \Dirac_1^M\bra{[f_e(C_{N+1}^M)\otimes \widetilde{a}]\cdot [\phi\otimes s\otimes s']}\\
&= [\Dirac_1^M\ ,\ f_e(C_{N+1}^M)\otimes \widetilde{a}]\cdot [\phi\otimes s\otimes s'] + (-1)^{\partial \widetilde{a}}[f_e(C_{N+1}^M)\otimes \widetilde{a}]\cdot\bra{\Dirac_1^M
[\phi\otimes s\otimes s']} \\
&= [\Dirac_1^M\ ,\ f_e(C_{N+1}^M)\otimes \widetilde{a}]\cdot [\phi\otimes s\otimes s'] + (-1)^{\partial \widetilde{a}}[f_e(C_{N+1}^M)\otimes \widetilde{a}]\cdot\bra{\Dirac_1^N
[\phi\otimes s\otimes s']}.
\end{align*}
Let us compute the commutator $[\Dirac_1^M\ ,\ f_e(C_{N+1}^M)\otimes \widetilde{a}]$. We may use the function language: 
$$\Dirac_1^M=\sum_{k=1}^M\sqrt{k}\bra{\partial_{x_k}\otimes c(e_k)+\partial_{y_k}\otimes c(f_k)+\frac{x_k}{2}c^*(Je_k)+\frac{y_k}{2}c^*(Jf_k)}.$$
We may assume that $\widetilde{a}$ is of the form $A\otimes w$ for $A\in C^\infty(U_N)$ and $w\in \Cl_+(\Lie(U_N))$ with $\|w\|=1$.
\begin{align*}
&[\Dirac_1^M\ ,\ f_e(C_{N+1}^M)\otimes \widetilde{a}] \\
&= \sum_{k=1}^M\sqrt{k}\bbbra{\partial_{x_k}\otimes c(e_k)+\partial_{y_k}\otimes c(f_k)+\frac{x_k}{2}\otimes c^*(Je_k)+\frac{y_k}{2}\otimes c^*(Jf_k)\ ,\ e^{-\sum_{b=N+1}^Mb^{-2l}(x_b^2+y_b^2)}\otimes \bbra{A\otimes c^*(w)}} \\
&=\sum_{k=1}^M\sqrt{k}\bra{\bbbra{\partial_{x_k},e^{-\sum_{b=N+1}^Mb^{-2l}(x_b^2+y_b^2)}\otimes A}\otimes c(e_k)c^*(w)+\bbbra{\partial_{y_k},e^{-\sum_{b=N+1}^Mb^{-2l}(x_b^2+y_b^2)}\otimes A}\otimes c(f_k)c^*(w)} \\
&\ \ \ +\sum_{k=1}^M\sqrt{k}\bra{\frac{x_k}{2}e^{-\sum_{b=N+1}^Mb^{-2l}(x_b^2+y_b^2)}\otimes A[c^*(Je_k),c^*(w)]+
\frac{y_k}{2}e^{-\sum_{b=N+1}^Mb^{-2l}(x_b^2+y_b^2)}\otimes A[c^*(Jf_k),c^*(w)]
}\\ 
&=\sum_{k=1}^N\sqrt{k}\bra{e^{-\sum_{b=N+1}^Mb^{-2l}(x_b^2+y_b^2)}\otimes\frac{\partial A}{\partial x_k}\otimes c(e_k)c^*(w)
+e^{-\sum_{b=N+1}^Mb^{-2l}(x_b^2+y_b^2)}\otimes\frac{\partial A}{\partial y_k}\otimes c(f_k)c^*(w)} \\
&\ \ \ +\sum_{k=N+1}^M\sqrt{k}\bra{-2k^{-2l}x_ke^{-\sum_{b=N+1}^Mb^{-2l}(x_b^2+y_b^2)}\otimes A\otimes c(e_k)c^*(w)-2k^{-2l}y_ke^{-\sum_{b=N+1}^Mb^{-2l}(x_b^2+y_b^2)}\otimes A\otimes c(f_k)c^*(w)}\\
&\ \ \ \ \ \ +\sum_{k=1}^N\sqrt{k}\bra{e^{-\sum_{b=N+1}^Mb^{-2l}(x_b^2+y_b^2)}\otimes \frac{x_k}{2}A[c^*(Je_k),c^*(w)]+
e^{-\sum_{b=N+1}^Mb^{-2l}(x_b^2+y_b^2)}\otimes \frac{y_k}{2}A[c^*(Jf_k),c^*(w)]
} \\
&=: I_1^M+I_2^M+I_3^M.
\end{align*}
Note that $I_1^M$ can be written as $e^{-(C_{N+1}^M)^2}\otimes(\cdots)$, where $(\cdots)$ is a bounded operator preserving $\dom(\Dirac_1^N)$ which is independent of $M$. For $I_2^M$, note that the operator norm of the operator $-2k^{-2l}x_ke^{-\sum_{b=N+1}^Mb^{-2l}(x_b^2+y_b^2)}\otimes A\otimes c(e_k)c^*(w)$ is at most $2k^{-l}\cdot \sqrt{2}e^{-2}\cdot \|A\|_{L^\infty}$, which one can check by considering the maximum value of the function $k^{-2l}x_ke^{-\sum_{b=N+1}^Mb^{-2l}(x_b^2+y_b^2)}$.
Thus the norm of $I_2^M$ is at most 
$$4\sqrt{2}e^{-2}\cdot \|A\|_{L^\infty}\sum_{k=N+1}^Mk^{\frac{1}{2}-l}.$$
Note that $I_3^M$ is not $\sum^M(\cdots)$ but $\sum^N(\cdots)$. This is because $w\in \Cl_+(\Lie(U_N))$. Since $e_k$ and $f_k$ are orthogonal to $\Lie(U_N)$ if $k>N$, $c^*(Je_k)$ and $c^*(Jf_k)$ anti-commute with $c^*(w)$. Consequently, $I_3^M$ is of the form $e^{-(C_{N+1}^M)^2}\otimes(\cdots)$, where $(\cdots)$ is a bounded operator preserving $\dom(\Dirac_1^N)$ which is independent of $M$.

Let us prove that the sequence $\{\Dirac[e^{-(C_{N+1}^M)^2}\otimes \widetilde{a}\cdot \phi\otimes s\otimes s']\}_M$ converges. 
\begin{align*}
&\Dirac\bra{[f_e(C_{N+1}^M)\otimes \widetilde{a}]\cdot [\phi\otimes s\otimes s']}
-\Dirac\bra{[f_e(C_{N+1}^{M'})\otimes \widetilde{a}]\cdot [\phi\otimes s\otimes s']}
\\
&\ \ \ =[\Dirac_1^M\ ,\ f_e(C_{N+1}^M)\otimes \widetilde{a}]\cdot [\phi\otimes s\otimes s'] - [\Dirac_1^{M'}\ ,\ f_e(C_{N+1}^{M'})\otimes \widetilde{a}]\cdot [\phi\otimes s\otimes s']\\
&\ \ \ \ \ \ +
(-1)^{\partial \widetilde{a}}[f_e(C_{N+1}^M)\otimes \widetilde{a}]\cdot\bra{\Dirac_1^N
[\phi\otimes s\otimes s']}-(-1)^{\partial \widetilde{a}}[f_e(C_{N+1}^{M'})\otimes \widetilde{a}]\cdot\bra{\Dirac_1^N
[\phi\otimes s\otimes s']} \\
&\ \ \ =(I_1^M-I_1^{M'})[\phi\otimes s\otimes s'] +(I_2^M-I_2^{M'})[\phi\otimes s\otimes s'] +(I_3^M-I_3^{M'})[\phi\otimes s\otimes s'] \\
&\ \ \ \ \ \ +(-1)^{\partial \widetilde{a}}\bbbra{\bbra{f_e(C_{N+1}^M)-f_e(C_{N+1}^{M'})}\otimes \widetilde{a}}\Dirac_1^N[\phi\otimes s\otimes s'].
\end{align*}

The first, the third and the fourth terms are of the form $(e^{-(C_{N+1}^M)^2}-e^{-(C_{N+1}^{M'})^2})(\cdots)$. Thus the sequence converges.

If $M'>M$, $I_2^M-I_2^{M'}$ can be rewritten as
$$I_2^M-I_2^{M'}=I_2^M-e^{-(C_{M+1}^{M'})^2}I_2^M+e^{-(C_{M+1}^{M'})^2}I_2^M-I_2^{M'}.$$
$e^{-(C_{M+1}^{M'})^2}I_2^M-I_2^{M'}$ is given by
$$\sum_{k=M+1}^{M'}\sqrt{k}\bra{e^{-\sum_{b=N+1}^{M'}b^{-2l}(x_b^2+y_b^2)}\otimes \frac{x_k}{2}A[c^*(Je_k),c^*(w)]+
e^{-\sum_{b=N+1}^{M'}b^{-2l}(x_b^2+y_b^2)}\otimes \frac{y_k}{2}A[c^*(Jf_k),c^*(w)]
}. $$
This is arbitrary small, if $M$ and $M'$ are large enough. Moreover, $I_2^M-e^{-(C_{M+1}^{M'})^2}I_2^M=[1-e^{-(C_{M+1}^{M'})^2}]I_2^M$ converges strongly to zero, which is because $1-e^{-(C_{M+1}^{M'})^2}\to 0$ as $M,M'\to \infty$.
More precisely, the norm of $I_2^M-e^{-(C_{M+1}^{M'})^2}I_2^M$ does not exceed
$$4\sqrt{2}e^{-2}\cdot \|A\|_{L^\infty}\sum_{k=N+1}^Mk^{\frac{1}{2}-l}+4\sqrt{2}e^{-2}\cdot \|A\|_{L^\infty}\sum_{k=N+1}^{M'}k^{\frac{1}{2}-l}.$$
It is small if $M$ and $M'$ are large enough. 

Combining all of these estimates, we find that 
$$\|\Dirac[f_e(C)\otimes \widetilde{a}\cdot \phi\otimes s\otimes s']\|\leq \|[f_e(C)\otimes \widetilde{a}]\cdot \Dirac(\phi\otimes s\otimes s')\|+c\|
\phi\otimes s\otimes s'\|$$
for some positive number $c$. Therefore $f_e(C)\otimes \widetilde{a}$ gives a bounded operator on $\dom(\Dirac)$.

For the $f_o$-case, we may study $C_{N+1}^\infty f_e(C_{N+1}^\infty)\otimes\widetilde{a}$. The difference from the $f_e$-case lies only on the explicit computation of the commutator, and we leave it to the reader.
\end{proof}

We essentially computed the following, in the previous lemma.

\begin{pro}
For any element $a\in\SC(U)_\fin$, $[a,\Dirac]$ is bounded on $\scr{S}_\fin\otimes^\alg \ca{H}_\fin$. As a result, this operator extends to $\scr{S}\otimes \ca{H}$.
\end{pro}

\begin{proof}
Just like the beginning of the above proof, we may assume that $a\in \SC(U)_\fin$ is of the form $f\otimes \widetilde{a}$ for $f=f_e$ or $f_o$ and $\widetilde{a}\in \SC(U_N)_\fin$, and we may study $[\Dirac,e^{-(C_{N+1}^\infty)^2}\otimes \widetilde{a}]$ and $[\Dirac,C_{N+1}^\infty e^{-(C_{N+1}^\infty)^2}\otimes \widetilde{a}]$.

Since $\Dirac$ is a closed operator, for $\phi\otimes s\otimes s'\in \ca{H}_\fin \otimes^\alg S_\fin^*\otimes^\alg\otimes S_\fin$,
\begin{align*}
 [\Dirac,e^{-(C_{N+1}^\infty)^2}\otimes \widetilde{a}](\phi\otimes s\otimes s') & =\lim_{M\to\infty} [\Dirac,e^{-(C_{N+1}^M)^2}\otimes \widetilde{a}](\phi\otimes s\otimes s') \\
&=\lim_{M\to\infty} [\Dirac_1^M,e^{-(C_{N+1}^M)^2}\otimes \widetilde{a}](\phi\otimes s\otimes s') \\
&=\lim_{M\to\infty}I_1^M+\lim_{M\to\infty}I_2^M+\lim_{M\to\infty}I_3^M.
\end{align*}
As proved above, all of $\lim_{M\to\infty}I_i^M$'s (for $i=1,2,3$) are bounded by $c\cdot \|\phi\otimes s\otimes s'\|$ for some positive constant $c$.

For $f_o$, use the same argument.

\end{proof}

\begin{pro}
As an operator on $\ca{H}$,
$(1+\Dirac^2)^{-1}$ is a compact operator.

\end{pro}
\begin{proof}
Since ${}^L\Dirac$ and ${}^R\Dirac$ are anti-commutative, 
$$\Dirac^2=-\frac{1}{2}({}^L\Dirac)^2+\frac{1}{2}({}^R\Dirac)^2.$$
These two operators have the same spectrum, which consists of non-negative integers with finite multiplicity.
Therefore $\Dirac^2+1$ has discrete and positive spectrum. Its inverse is a compact operator.
\end{proof}

This fact tells us the following thing: On $\scr{S}\otimes \ca{H}$, the operator $[f\otimes a]\cdot (1+\id\otimes \Dirac^2)^{-1}$ is an $\scr{S}$-compact operator. Thus we have proved that the pair $(\scr{S}\otimes\ca{H},\Dirac)$ gives a Kasparov $(\SC(U),\scr{S})$-module. 

Let us move to the equivariance issues.
The following lemma looks obvious, at first sight. However, there are some difficulties: $U$ does not preserve $\ca{H}_\fin$; and $\Dirac$ is an infinite sum. 

\begin{lem}\label{U preserves dom of D}
$U$ preserves the domain of $\Dirac$.
\end{lem}

\begin{proof}
Recall that $\ud{L^2(U)}\otimes S^*\otimes S\cong \ud{L^2(\bb{R}^\infty)}\otimes S^*\otimes \ud{L^2(\bb{R}^\infty)^*}\otimes S$. Since ${}^R\Dirac$ is actually $U$-equivariant, we may ignore this part. Moreover, $\Dirac$ is trivial on the $\scr{S}$-part. As a result, we may concentrate on the operator ${}^L\Dirac$ acting on $\ud{L^2(\bb{R}^\infty)}\otimes S^*$.

We divide the problem into two parts: $(1)$ $g\in U_\fin$ preserves $\dom({}^L\Dirac)$ and $(2)$ the general case. The former is much easier than the latter, but it contains the key computation, which has essentially been mentioned in \cite{FHTII}.

$(1)$ What we need to check is the following: For any $v\in \dom({}^L\Dirac)$ whose approximating sequence is $\{v_n\}\subseteq \ud{L^2(\bb{R}^\infty)_\fin}\otimes^\alg S^*_\fin$, the sequence $\{g(v_n)\}$ converges in the graph norm of ${}^L\Dirac$. Since the $U_\fin$-action on $\ud{L^2(\bb{R}^\infty)}$ is unitary, the sequence converges in the norm of $\ud{L^2(\bb{R}^\infty)}\otimes S^*$. 
Thus, it suffices to check that $\{{}^L\Dirac(gv_n)\}$ gives a Cauchy sequence. 

Let us notice an elementary equality:
$$
{}^L\Dirac(gv_n) =g.\bbbra{g^{-1}\circ{}^L\Dirac\circ g(v_n)}=\bra{g\circ[g^{-1}({}^L\Dirac)]}(v).
$$
In the strong sense on $\ud{L^2(\bb{R}^\infty)_\fin}\otimes^\alg S^*_\fin$, the operator $g^{-1}({}^L\Dirac)$ can be written as
$$
g^{-1}\circ{}^L\Dirac\circ g=\sum_k\bbbra{ \rho(g^{-1})d\rho(z_k)\rho(g)\otimes\gamma^*(\overline{z_k})+\rho(g^{-1})d\rho(\overline{z_k})\rho(g)\otimes\gamma^*(z_k)}.
$$
This equality is proved as follows: Since $g\in U_\fin$, it preserves $\ker(d\rho(z_k)\otimes \gamma^*(\overline{z_k})+d\rho(\overline{z_k})\otimes \gamma^*(z_k))$ for sufficiently large $k$. Hence the infinite sum above is, vectorwisely, a finite sum. Let $g=\exp(f)$.
Then, as explained in Definition \ref{central extension of LT},
$$\rho(g^{-1})d\rho(z_k)\rho(g) =d\rho(\Ad_{g^{-1}}(z_k))=d\rho(z_k+\omega(-f,z_k)K).$$
Since $K$ gives $i\id$ on the representation space,
$g^{-1}\circ{}^L\Dirac\circ g={}^L\Dirac -i\gamma^*(Jf)$. 
This formula can be proved by considering the Fourier expansion $f=\sum_k(f_kz_k+\overline{f_k}\overline{z_k})$ and $\omega(z_k,\overline{z_k})=-i$.
Seeing this formula, we estimate $\|{}^L\Dirac(gv_n)-{}^L\Dirac(gv_m)\|$. 
\begin{align*}
\|{}^L\Dirac(gv_n)-{}^L\Dirac(gv_m)\| &=
\|g^{-1}({}^L\Dirac)(v_n)-g^{-1}({}^L\Dirac)(v_m)\| \\
&=\|{}^L\Dirac(v_n-v_m)-i\gamma^*(Jf)(v_n-v_m)\| \\
&\leq \|{}^L\Dirac(v_n-v_m)\|+\|\gamma^*(Jf)(v_n-v_m)\|.
\end{align*}
The first term is small, since $\{v_n\}$ converges in the graph norm of ${}^L\Dirac$. The second one is also small, because $\gamma^*(Jf)$ gives a bounded operator.


We need to add a comment before going to $(2)$. The above estimate enables us to generalize the formula
$g^{-1}({}^L\Dirac)={}^L\Dirac -i\gamma^*(Jf)$
on the $\ud{L^2(\bb{R}^\infty)_\fin}\otimes^\alg S^*_\fin$, to the same formula on $\dom({}^L\Dirac)$. Let $v\in \dom({}^L\Dirac)$, and $g\in U_\fin$. Suppose that $\{v_n\}\subseteq \ud{L^2(\bb{R}^\infty)}\otimes S_\fin^*$ converges to $v$ in the graph norm of ${}^L\Dirac$. Then $\{g.v_n\}$ converges to $gv$ in the graph norm of ${}^L\Dirac$.
\begin{align*}
g^{-1}\circ{}^L\Dirac\circ g(v) &= g^{-1}.\lim_k{}^L\Dirac (g.v_k) = \lim_kg^{-1}\circ{}^L\Dirac\circ g(v_k)  \\
&= \lim_k [{}^L\Dirac -i\gamma^*(Jf)](v_k) ={}^L\Dirac(v) -i\gamma^*(Jf)(v).
\end{align*}

$(2)$ Let us consider the general case. Assume that $g\in U$. Let $\{g_n\}\subseteq U_\fin$ be an approximating sequence of $g$. 
Take any $v\in \dom({}^L\Dirac)$.
Then, $\{g_nv\}$ converges to $g.v$.
We need to prove that $\{{}^L\Dirac(g_nv)\}$ is a Cauchy sequence. Put $g_n=\exp(f_n)$ and $g=\exp(f)$.
\begin{align*}
{}^L\Dirac(g_nv)-{}^L\Dirac(g_mv) &=
g_n.g_n^{-1}({}^L\Dirac)(v)-g_m.g_m^{-1}({}^L\Dirac)(v) \\
&= g_n.\bbbra{{}^L\Dirac-i\gamma^*(Jf_n)}(v)-g_m.\bbbra{{}^L\Dirac-i\gamma^*(Jf_m)}(v) \\
&= \Bigl[g_n.{}^L\Dirac(v)-g_m.{}^L\Dirac(v)\Bigr] - \Bigl[g_n.\bbbra{i\gamma^*(Jf_n-Jf_m)}(v)\Bigr] - \Bigl[ g_n.i\gamma^*(Jf_m)(v)- g_m.i\gamma^*(Jf_m)(v)\Bigr].
\end{align*}

Let us prove that all of the three terms are arbitrary small, if $n$ and $m$ are large enough. The first one is small, because the representation is continuous. The norm of the second term is given by $\|\gamma^*(Jf_n-Jf_m)(v)\|$. Since $[\gamma^*(X)]^*\gamma^*(X)=\|X\|^2\id$ by definition of the Clifford multiplication, and $\{g_n\}$ converges, the norm of the second term is small.

For the third one, we regard $g_n$ and $g_m$ as operators, and we write that as $(g_n-g_m).i\gamma^*(Jf_m)(v)$. We compute the square of the norm, noticing that $g_n-g_m$ commutes with $\gamma^*(Jf_m)$.
\begin{align*}
\|(g_n-g_m).i\gamma^*(Jf_m)(v)\|^2 &= \inpr{(g_n-g_m).i\gamma^*(Jf_m)(v)}{(g_n-g_m).i\gamma^*(Jf_m)(v)}{} \\
&=\inpr{(g_n-g_m).[\gamma^*(Jf_m)]^*\gamma^*(Jf_m)(v)}{(g_n-g_m).v}{} \\
&= \|Jf_m\|^2_{L^2}\inpr{(g_n-g_m).v}{(g_n-g_m).v}{}.
\end{align*}
Since $\|Jf_m\|^2_{L^2}$ converges to $\|Jf\|^2_{L^2}$, and since $\inpr{(g_n-g_m).v}{(g_n-g_m).v}{}$ converges to $0$, the quantity we are estimating is small.


\end{proof}

The following is almost a corollary of the above proof.

\begin{pro}\label{equivariance condition on D}
For any $g\in U$, $g(\Dirac)-\Dirac$ is bounded, 
$[\Dirac,\Dirac-g(\Dirac)]\frac{\Dirac}{1+\Dirac^2+\lambda}$ is uniformly bounded in $\lambda$,
and the map $g\mapsto g(\Dirac)-\Dirac$ is continuous on $U$.

\end{pro}
\begin{proof}

Let $g=\exp(f)\in U$.
We firstly extend the formula $g({}^L\Dirac)-{}^L\Dirac=i\gamma^*(Jf)$ for $U_\fin$ to the general case $g\in U$. Take $\{g_n\}\subseteq U_\fin$, such that $g_n=\exp(f_n)$ and $g_n\to g$. 
For $v\in \dom(\Dirac)$, 
$$g({}^L\Dirac)(v)= g.\lim_n\bbbra{{}^L\Dirac(g_n^{-1}.v)} = g.\lim_ng_n^{-1}.\bbbra{g_n\circ{}^L\Dirac\circ g_n^{-1}(v)} $$
$$=g.\lim_ng_n^{-1}.{}^L\Dirac(v)
+g.\lim_ng_n^{-1}.i\gamma^*(Jf_n)(v).$$ 
The first term converges to ${}^L\Dirac(v)$, clearly. The second one converges to $i\gamma^*(Jf)(v)$,  as follows:
$$
\|g_n^{-1}i\gamma(Jf_n)(v)-g^{-1}i\gamma(Jf)(v)\|\leq
\| g_n^{-1}.\bbbra{i\gamma(Jf_n)(v)-i\gamma(Jf)(v)}\|+
\|(g_n^{-1}-g^{-1})i\gamma^*(Jf)(v)\|.$$
The first term tends to zero, since $g_n^{-1}$ gives an isometry, and $Jf_n$ converges to $Jf$. The second one also tends to zero, because $g_n$ converges to $g$.

Since $i\gamma^*(Jf)$ is a bounded operator, $g(\Dirac)-\Dirac$ is also a bounded operator.

We would like to prove that $[\Dirac,\Dirac-g(\Dirac)]\frac{\Dirac}{1+\Dirac^2+\lambda}$ is uniformly bounded in $\lambda$. Firstly, we note that $\Dirac-g(\Dirac)$ is a bounded operator, but we have not proved that $\Dirac-g(\Dirac)$ preserves the domain of $\Dirac$ for general $g\in U$.\footnote{If $g\in U_\fin$, $g(\Dirac)-\Dirac$ preserves $\dom(\Dirac)$.} However, it is not a serious problem in our case. This is because $\Dirac-g(\Dirac)$ sends $\ca{H}_\fin$ to $\dom(\Dirac)$, which we will prove soon, and $\Dirac$ preserves $\ca{H}_\fin$. Therefore the operator we are considering is well-defined on at least $\ca{H}_\fin$. Once the operator is proved to be defined on $\ca{H}_\fin$, and if it is bounded, it extends to $\ca{H}$ itself uniquely.

We check that $\Dirac-g(\Dirac)$ maps $\ca{H}_\fin$ into $\dom(\Dirac)$ here. Let $\{g_n\}\subseteq U_\fin$ be an approximating sequence of $g$, and let $v\in \ca{H}_\fin$. 
We check that $\{\gamma^*(Jf_n)(v)\}$ converges in the graph norm of $\Dirac$. It means that $[\Dirac-g(\Dirac)](v)\in \dom(\Dirac)$. Since $g_n$ converges to $g$, the sequence converges in ``$L^2$''. Hence we need to prove that $\{\Dirac[\gamma^*(Jf_n)(v)]\}$ is a Cuachy sequence.

Let $Jf_n=\sum_l(a_lz_l+b_l\overline{z_l})$. Since $g_n\in U_\fin$, there is no problem to exchange the order of sum. 
In the following, we omit $(v)$ in order to simplify the notations.
%
\begin{align*}
\Dirac\circ\gamma^*(Jf_n) &= \frac{i}{\sqrt{2}}
[{}^L\Dirac,\gamma^*(Jf_n)]-\gamma^*(Jf_n)\circ \Dirac \\
&= \frac{i}{\sqrt{2}}\bbbra{\sum_{k}\sqrt{k}\bra{d\rho(z_k)\otimes \gamma^*(\overline{z_k})+d\rho(z_k)\otimes \gamma^*(\overline{z_k})},\sum_l(a_l\gamma^*(z_l)+b_l\gamma^*(\overline{z_l}))} - \gamma^*(Jf_n)\circ \Dirac \\
&= \frac{i}{\sqrt{2}}\sum_{k,l}\sqrt{k} \bbbra{
d\rho(z_k)\otimes \gamma^*(\overline{z_k})+d\rho(z_k)\otimes \gamma^*(\overline{z_k}),a_l\gamma^*(z_l)+b_l\gamma^*(\overline{z_l})} -\gamma^*(Jf_n)\circ \Dirac \\
&= \frac{i}{\sqrt{2}}\sum_k\sqrt{k} \bra{a_kd\rho(z_k)\otimes [\gamma^*(\overline{z_k}),\gamma^*(z_k)]+b_kd\rho(\overline{z_k})\otimes [\gamma^*(z_k),\gamma^*(\overline{z_k})] } -\gamma^*(Jf_n)\circ \Dirac\\
&= \sqrt{2}i\sum_k\sqrt{k} \bra{a_kd\rho(z_k)+b_kd\rho(\overline{z_k})} \otimes \id -\gamma^*(Jf_n)\circ \Dirac \\
&= \sqrt{2}id\rho(|d|^{1/2}Jf_n)\otimes \id - \gamma^*(Jf_n)\circ \Dirac,
\end{align*}
where we put $|d|^{1/2}Jf_n:=\sum_l\sqrt{l}(a_lz_l+b_l\overline{z_l})$.
We need to check that $\Dirac\circ\gamma^*(Jf_n)(v)$ converges. Thanks to the above computation,
$$\Dirac\circ\gamma^*(Jf_n)(v)=\sqrt{2}i[d\rho(|d|^{1/2}Jf_n)\otimes \id](v) - \gamma^*(Jf_n)\circ \Dirac(v).$$
The right hand side converges, if $g_n$ converges to $g$. Moreover, for any $g\in U$, we find that 
$[\Dirac,\gamma^*(Jf)]=\sqrt{2}id\rho(|d|^{1/2}Jf)\otimes \id.$

We are in the position to prove that the operator norm of $[\Dirac,\Dirac-g(\Dirac)]\frac{\Dirac}{1+\Dirac^2+\lambda}$ is uniformly bounded in $\lambda$. 
Formally, this is a kind of elliptic regularity: $[\Dirac,\Dirac-g(\Dirac)]$ and $\Dirac$ are first order, and $\frac{1}{1+\Dirac^2+\lambda}$ is the inverse of an elliptic operator of second order; the product is of ``$0$-th order''.

Let us prove that $\frac{\Dirac}{1+\Dirac^2+\lambda}$ maps $\ca{H}$ into $\dom(\Dirac)$. Take $v\in \ca{H}$. Choose an approximation $\{v_n\}\subseteq \ca{H}_\fin$, of $v$. Let $b_1(x):=\frac{x}{1+x^2+\lambda}$. Notice that $b_1$ is a bounded function, and hence $b_1(\Dirac)$ is a bounded operator. Consequently, $b_1(\Dirac)(v_n)=\frac{\Dirac}{1+\Dirac^2+\lambda}(v_n)$ converges to $b_1(\Dirac)v$. Put $b_2(x)=\frac{x^2}{1+x^2+\lambda}$. Then, $b_2(\Dirac)=\Dirac b_1(\Dirac)$. Since $b_2$ is also a bounded function, $b_2(\Dirac)(v_n)=\Dirac\bbbra{b_1(\Dirac)(v_n)}$ converges again. Since $\Dirac$ is a closed operator, $b_1(\Dirac)(v)\in \dom(\Dirac)$ and $\Dirac[b_1(\Dirac)(v)]=b_2(\Dirac)(v)$.

Let us prove that $d\rho(X)$ maps $\dom(\Dirac)$ to $\ca{H}$ continuously, for all $X\in \Lie(U)$. For this aim, it is sufficient to prove the following inequality:
$$\|d\rho(X)\otimes\id(v)\|^2 \leq C\bra{\|v\|^2+\|\Dirac(v)\|^2},$$
for some $C$ which is independent of $v$. Suppose that $X=\sum(a_kz_k+b_k\overline{z_k})$.
Note that
$$\inpr{d\rho(z_k)v}{d\rho(z_k)v}{}=\inpr{-d\rho(\overline{z_k})d\rho(z_k)v}{v}{}\leq  \frac{2}{k}\inpr{\Dirac^2v}{v}{}=\frac{2}{k}\|\Dirac(v)\|^2,$$
and similar for $d\rho(\overline{z_k})$. Then, we find that
$$\|d\rho(X)\otimes\id(v)\| \leq 2\sum\bra{
\frac{|a_k|}{\sqrt{k}}\|\Dirac(v)\|+\frac{|b_k|}{\sqrt{k}}\|\Dirac(v)\|}\leq C\bra{\|v\|^2+\|\Dirac(v)\|^2}$$
where $C$ can be taken as ``the $l^2_{1/2}$ norm of $X$'' times a constant. Let us prove that using the Cauchy-Schwarz inequality.
\begin{align*}
\bbbra{\sum\bra{\frac{|a_k|}{\sqrt{k}}+\frac{|b_k|}{\sqrt{k}}}} &= 
\sum\bra{\frac{1}{k}\cdot \sqrt{k}|a_k|+\frac{1}{k}\cdot \sqrt{k}|b_k|} 
\\
&\leq \sqrt{2\sum \frac{1}{k^2}}\cdot \sqrt{\sum k(|a_k|^2+|b_k|^2)}.
\end{align*}

Finally, we must prove the $U$-continuity. It is clear by the explicit description: $g(\Dirac)-\Dirac=\frac{i}{\sqrt{2}}\gamma^*(Jf)$.

\end{proof}

As a result of these propositions, we obtain the following index element.

\begin{dfn-thm}\label{dfn-thm index element}
The pair $(\scr{S}\otimes \ca{H},\Dirac)$ is an unbounded $\tau$-twisted $U$-equivariant Kasparov $(\SC(U),\scr{S})$-module. The corresponding $KK$-element is denoted by $[\widetilde{\Dirac}]\in KK_U^\tau(\SC(U),\scr{S})$, and is called the {\bf index element}.


\end{dfn-thm}

\begin{rmk}\label{another perturbation}
The operator ${}^R\Dirac_1^N+\bbbra{\frac{1}{\sqrt{2}}{}^R\Dirac_{N+1}^\infty+\frac{i}{\sqrt{2}}{}^L\Dirac_{N+1}^\infty}$ defines the same $KK$-element. 
This operator is actually $U_N$-equivariant.

\end{rmk}

\subsection{Dirac element $[\widetilde{d_U}]$}

The original Dirac element $[d_X]\in KK(Cl_\tau(X),\bb{C})$ for a Riemannian manifold is given by the pair $(L^2(X,\wedge^*T^*X),d+d^*)$. However, it is possible to add a potential term. Let us check that using the definition of unbounded equivariant Kasparov modules.




Let $X$ be a complete Riemannian manifold, and consider the above Kasparov module $[d_X]$. We write $d+d^*$ as $D$ in the following two propositions, in order to simplify the notation.

\begin{pro}
Let $h$ be a smooth bundle homomorphism of $\wedge^*T^*X$ satisfying the following conditions: $h(x)$ is odd, self-adjoint, commutes with the $\Cl_-(T^*_xX)$-action (or the symbol of $D$); and $\nabla(h)$ is a bounded section, where $\nabla$ is induced by the Levi-Civita connection. We also impose that $D+th$ is self-adjoint for $t\in [0,1]$.
Then, $(L^2(X,\wedge^*T^*X),D+h)$ defines a $K$-homology class of $Cl_\tau(X)$, and it is the same with $[d_X]$ as $KK$-elements. More precisely, they are operator homotopic to each other through unbounded Kasparov modules, by $(L^2(X,\wedge^*T^*X),D+th)$.
\end{pro}
\begin{proof}
It suffices to prove the following: $(L^2(X,\wedge^*T^*X),D+th)$ is unbounded Kasparov modules for each $t$. Let $D_t:=D+th$. 

Firstly, we check that $[a,D_t]$ is bounded for each smooth and compactly supported $a$ and each $t$. Since $\Cl_+(T^*X)$-action commutes with the symbol of $D$, the commutator $[a,D]$ is bounded. Moreover, since $[a,h]$ is a compactly supported smooth section of the bundle $\Cl_+(T^*X)$, it gives a bounded operator.

$a(1+D^2)^{-1}$ is supposed to be compact. In order to prove that $a(1+D_t^2)^{-1}$ is compact, we compare it with $a(1+D^2)^{-1}$.
\begin{align*}
a(1+D_t^2)^{-1}-a(1+D^2)^{-1}&= a(1+D^2)^{-1}(1+D^2-1-D_t^2)(1+D_t^2)^{-1} \\
&=-a(1+D^2)^{-1}(t(hD+Dh)+t^2h^2)(1+D_t^2)^{-1}.
\end{align*}
It is sufficient to prove that $(t(hD+Dh)+t^2h^2)(1+D_t^2)^{-1}$ is a bounded operator. We note that the operator $(1+D_t^2)^{-1}$ maps $L^2$ to $\dom(D_t^2)$. The graph norm of $D_t$, which is denoted by $\|\bullet\|_{D_t}$, is given by
\begin{align*}
\|\bullet\|_{D_t}^2&=\|\bullet\|_{L^2}^2+\|D_t\bullet\|_{L^2}^2 \\
&=\|\bullet\|_{L^2}^2+\|D\bullet\|_{L^2}^2+t\inpr{[D,h]\bullet}{\bullet}{L^2}+t^2\|h\bullet\|_{L^2}^2.
\end{align*}
Since $[D,h]$ is bounded (as proved soon), the norm $\|\bullet\|_{D_t}$ is equivalent to the norm $\|\bullet\|_{L^2}+\|D\bullet\|_{L^2}+t^2\|h\bullet\|_{L^2}$. Therefore $h$ maps $\dom(D_t)$ continuously to $L^2$. Consequently, $(t(hD+Dh)+t^2h^2)(1+D_t^2)^{-1}$ is a bounded operator.

Finally, we must check that $[D,h]$ is bounded.
\begin{align*}
Dh+hD &= \sum\partial_{x_i}\otimes(dx_i\wedge-dx_i\rfloor)\circ h+h\circ \sum\partial_{x_i}\otimes(dx_i\wedge-dx_i\rfloor) \\
&= -\sum\frac{\partial h}{\partial x_i}\otimes(dx_i\wedge-dx_i\rfloor).
\end{align*}
This is a bounded operator by the assumption that $\nabla(h)$ is bounded.

\end{proof}

For the equivariant $KK$-theory version, we need extra conditions.

\begin{pro}
Suppose that a locally compact, second countable and Hausdorff group $\Gamma$ acts on $X$, isometrically, properly and cocompactly. Suppose that $h$ satisfies the following additional assumptions:
\begin{itemize}
\item $g(h)-h$ is a bounded operator for each $g\in \Gamma$.
\item The function $\Gamma\ni g\to g(h)-h\in \ca{B}(L^2)$ is norm continuous.
\end{itemize}
Then, two unbounded $\Gamma$-equivariant Kasparov modules $(L^2(X,S),D+h)$ and $(L^2(X,S),D)$ are operator homotopic.
\end{pro}
\begin{proof}
We will follow the same story: We will verify that $D+th$ gives an unbounded $\Gamma$-equivariant Kasparov module.
For this aim, we need to check the followings: $(1)$ $g(D+th)-(D+th)$ is bounded for any $g$; $(2)$ $[D+th-g(D+th),D+th]\frac{D+th}{1+(D+th)^2}$ is bounded for any $g$; and $(3)$ $g\mapsto g(D+th)-(D+th)$ is continuous.

$(1)$ and $(3)$ are clear from the assumptions. Recall that $D$ is actually $\Gamma$-equivariant: $g(D)=D$.
For $(2)$, we need some computation. Put $D_t:=D+th$, and $L^2_{k,t}:=\dom(D_t^k)$. $L^2=L^2_{0,t}$ denotes the original $L^2$-space. The operator $\frac{D+th}{1+(D+th)^2}$ maps $L^2$ into $L^2_{1,t}$ continuously. Thus it suffices to check that $[D+th-g(D+th),D+th]$ maps $L^2_{1,t}$ into $L^2$ continuously.
\begin{align*}
&[D+th-g(D+th),D+th] \\
&= [D-g(D),D]+t\bra{[D-g(D),h]+[h-g(h),D]}+t^2[h-g(h),h] \\
&=[D-g(D),D]+t\bra{[D,h]-[g(D),h]+[h,D]-g[h,g^{-1}(D)]}+t^2\bbra{(h-g(h))h+h(h-g(h))}.
\end{align*}
Note that $D$ is $G$-equivariant. Then, first two terms themselves turn out to be bounded operators on $L^2$. For the third one, we note that $L^2_{1,t}=\dom(D)\cap \dom(th)$, which will be proved in the next paragraph. By the assumption, $h-g(h)$ is a bounded function, and hence it preserves $\dom(th)$. In particular, the operator $\bbra{(h-g(h))h+h(h-g(h))}$ maps $\dom(th)\supseteq L^2_{1,t}$ into $L^2$ continuously.

Let us prove that $L^2_{1,t}=\dom(D)\cap \dom(th)$. There is a common domain $C_c^\infty$. Take a vector $v\in L^2_{1,t}$ and an approximating sequence $\{v_n\}$ from $C_c^\infty$ in the $L^2_{1,t}$-norm. Then, $v_n$ tends to $v$ in the $L^2$-norm, and $\{D_t(v_n)\}$ gives a Cauchy sequence.
Hence, $\|D_t(v_n)-D_t(v_m)\|^2=\|D(v_n-v_m)\|^2+t\inpr{(Dh+hD)(v_n-v_m)}{v_n-v_m}{}+t^2\|h(v_n-v_m)\|^2$ goes to zero as $n,m\to\infty$. The second one can be negative, but it goes to zero. This is because $Dh+hD$ is a bounded operator. Therefore the both of $\{Dv_n\}$ and $\{hv_n\}$ are Cauchy sequences, and hence $v\in \dom(D)\cap \dom(h)$.

\end{proof}

These results justify the name of the following $KK$-element.

\begin{dfn}[See \cite{HK}]
On a finite-dimensional Euclidean space $W$, take an orthonormal basis $\{e_i\}$. The corresponding coordinate is denoted by $\{x_i\}$. Let us consider two Clifford multiplications on $\bigwedge^*T^*W$: $c_1(v)=v\wedge-v\rfloor$ and $c_1^*(v)=v\wedge+v\rfloor$. Define the Bott-Dirac operator as
$$D_W:=\sum_{i}\bra{c_1(e_i)\frac{\partial}{\partial x_i}+\frac{1}{2}c_1^*(e_i)x_i}.$$

$W$ acts on itself by the parallel transformation. The pair $(L^2(W,\bigwedge^*T^*W),D_W)$ determines an unbounded $W$-equivariant Kasparov $(Cl_\tau(W),\bb{C})$-module. $[d_W]$ denotes the corresponding $KK_W$-element, and it is called the {\bf Dirac element}. To distinguish it from the original Dirac element, we sometimes call it the {\bf Bott-Dirac element}.
\end{dfn}

What we can generalize is not the original Dirac element, but a perturbed version of the Bott-Dirac element. It has a strong advantage: The operator itself is Fradholm. The cost we must pay is that the operator is not actually equivariant. However, thanks to the flexibility of the definition of unbounded equivariant Kasparov modules, it it not a big problem.

We have pointed out that our Dirac operator is in fact the Bott-Dirac operator in Remark \ref{our Dirac is Bott-Dirac}. The pair $(\scr{S}\otimes \ca{H},\Dirac)$ defines a $\tau$-twisted $U$-equivariant Kasparov $(\SC(U),\scr{S})$-module. We would like to define the Dirac element by the same module and operator {\bf with the different action}.

The natural ($\tau$-twisted) action on our ``$L^2$-space'' looks a combination of the translation and the multiplication of the $2$-cocycle. On the other hand, the action to define the Dirac element must look like the translation without any cocycles. Such new action is the following.



\begin{dfn-thm}\label{dfn-thm Dirac element}
The group action $A$ is given by the following: For $g=\exp(f)\in U$,
$$A_g:=L_{\exp\bra{\frac{f}{2}}}R_{\exp\bra{-\frac{f}{2}}}.$$

The pair $(\scr{S}\otimes \ca{H},\Dirac)$ with the action $A$ gives an unbounded $U$-equivariant \footnote{Not $\tau$-twisted!} Kasparov $(\SC(U),\scr{S})$-module. The corresponding $KK$-element is called the {\bf Dirac element} and denoted by $[d_U]\in KK_U(\SC(U),\scr{S})$.

\end{dfn-thm}

\begin{rmks}Before going to the proof, we give several remarks.

$(1)$ Recall that ${}^R\Dirac$ is given by $\sum_k\sqrt{k}(\cdots)$. Thus it has already been ``perturbed'' from the original definition, unlike the Bott-Dirac operator of \cite{HK}.

$(2)$ The Bott-Dirac operator in \cite{HK} is Fredholm, but it is not with compact resolvent.
To overcome this serious problem, N. Higson and G. Kasparov introduced a perturbation family, whose ``boundary'' gives the desired ``Kasparov module''. This family gives an equivariant $E$-theory element by the central invariant. Through the isomorphism between $KK$ and $E$ for their cases, they obtained the desired $KK$-element. 

$(3)$ Comparing their construction, our definition seems to be too simple. Why is this OK? The reason is our group action is very special: The linear part of the action is trivial and hence it has already been ``diagonalized''. On the other hand, in Lemma 5.7 in \cite{HK}, they ``weakly diagonalized'' the linear part of a group action, in order to perturb the Bott-Dirac operator. Their action is not actually diagonalizable, and hence the perturbed operator does not give an equivariant Kasparov module; The family is merely ``asymptotically equivariant''. 

In order to generalize our result (to the $LG$-case for non-commutative $G$, for example), we may need to follow Higson-Kasparov's method. However, in the present paper, we concentrate on our special case.

\end{rmks}

\begin{proof}
We have already proved everything, except for the group action issues. 

We check the following things: $(0)$ $A$ is actually a group action which is compatible with the action $U\circlearrowright \SC(U)$ and the $\SC(U)$-module structure on $\scr{S}\otimes \ca{H}$; $(1)$ $U$ preserves the domain of $\Dirac$;
$(2)$ $g(\Dirac)-\Dirac$ is bounded; $(3)$ $[\Dirac,\Dirac-g(\Dirac)]\cdot\frac{\Dirac}{1+\lambda +\Dirac^2}$ is bounded; and $(4)$ the map $g\mapsto g(\Dirac)-\Dirac$ is continuous.

$(0)$  Let $\sqrt{g}$ be $\exp(f/2)$ for $g=\exp(f)$. Note that $L_{g_1}\circ L_{g_2}=L_{g_1\cdot g_2}e^{\frac{i}{2}\omega(f_1,f_2)}$ and $R_{g_1}\circ R_{g_2}=R_{g_1\cdot g_2}e^{-\frac{i}{2}\omega(f_1,f_2)}$. Then, since $L$ and $R$ are commutative,
\begin{align*}
A_{g_1}\circ A_{g_2} &= L_{\sqrt{g_1}}\circ R_{\sqrt{g_1}^{-1}}\circ L_{\sqrt{g_2}}\circ R_{\sqrt{g_2}^{-1}} \\
&= L_{\sqrt{g_1}}\circ L_{\sqrt{g_2}}\circ R_{\sqrt{g_1}^{-1}}\circ R_{\sqrt{g_2}^{-1}} \\
&= L_{\sqrt{g_1}\cdot \sqrt{g_2}}e^{\frac{i}{2}\omega(f_1/2,f_2/2)}
L_{\sqrt{g_1}^{-1}\cdot \sqrt{g_2}^{-1}}e^{-\frac{i}{2}\omega(f_1/2,f_2/2)} \\
&=A_{g_1\cdot g_2}.
\end{align*}

In order to prove the compatibility with the $\SC(U)$-module structure, we need to recall the definition of $U\circlearrowright \SC(U)$. Formally, this action is given by $g.a(x)=a(x-g)$. Just like Theorem \ref{the U action}, we can check that $R_g(a\cdot \phi)=[g^{-1}.a]\cdot R_g(\phi)$ for $a\in \SC(U)$ and $\phi\in \scr{S}\otimes \ca{H}$. Thus,
\begin{align*}
A_g(a\cdot \phi)&= L_{\sqrt{g}}\circ R_{\sqrt{g}^{-1}}(a\cdot \phi) \\
&= L_{\sqrt{g}}[\sqrt{g}.a\cdot R_{\sqrt{g}^{-1}}(\phi)] \\
&= [\sqrt{g}.\sqrt{g}.a]\cdot L_{\sqrt{g}}\circ R_{\sqrt{g}^{-1}}(\phi)\\
&= [g.a]\cdot A_g(\phi).
\end{align*}

$(1)$ Roughly speaking, $\dom(\Dirac)=\scr{S}\otimes\dom({}^L\Dirac)\otimes \dom({}^R\Dirac)$. 
As Lemma \ref{U preserves dom of D},
the $U$-action $L$ preserves $\dom({}^L\Dirac)$. In the same way, one can see that the action $R$ preserves $\dom({}^R\Dirac)$. Since $A$ is a combination of $L$ and $R$, the action $A$ preserves $\dom(\Dirac)$.

$(2)$ Let us compute $g(\Dirac)-\Dirac$ for the action $A$. In the following, $A_g(\Dirac)$ means $A_g\circ \Dirac \circ A_{g^{-1}}$, and similarly for $L$ and $R$.
\begin{align*}
A_g(\Dirac)-\Dirac &= 
R_{\sqrt{g}^{-1}}L_{\sqrt{g}}\bra{\frac{i}{\sqrt{2}}{}^L\Dirac+\frac{1}{\sqrt{2}}{}^R\Dirac}-\frac{i}{\sqrt{2}}{}^L\Dirac-\frac{1}{\sqrt{2}}{}^R\Dirac \\
&= \bbbra{L_{\sqrt{g}}\bra{\frac{i}{\sqrt{2}}{}^L\Dirac}-\frac{i}{\sqrt{2}}{}^L\Dirac}+ \bbbra{R_{\sqrt{g}^{-1}}\bra{\frac{1}{\sqrt{2}}{}^R\Dirac}-\frac{1}{\sqrt{2}}{}^R\Dirac},
\end{align*}
where we used the $R$- and $L$-invariance of ${}^L\Dirac$ and ${}^R\Dirac$, respectively. The first term is bounded, thanks to Proposition 4.19. 
For the second term, repeat the same argument for the action $R$.

$(3)$ Let us compute the commutator $[\Dirac,\Dirac-g(\Dirac)]$.
\begin{align*}
[\Dirac,\Dirac-g(\Dirac)] &= 
\bbbra{\frac{i}{\sqrt{2}}{}^L\Dirac+\frac{1}{\sqrt{2}}{}^R\Dirac\ ,\ \frac{i}{\sqrt{2}}{}^L\Dirac+\frac{1}{\sqrt{2}}{}^R\Dirac-R_{\sqrt{g}^{-1}}L_{\sqrt{g}}\bra{\frac{i}{\sqrt{2}}{}^L\Dirac+\frac{1}{\sqrt{2}}{}^R\Dirac}} \\
&= \frac{1}{2}
[i{}^L\Dirac+{}^R\Dirac\ ,\ i{}^L\Dirac+{}^R\Dirac-R_{\sqrt{g}^{-1}}L_{\sqrt{g}}(i{}^L\Dirac+{}^R\Dirac)] \\
&= \frac{1}{2}\bbra{
-[{}^L\Dirac,{}^L\Dirac-L_{\sqrt{g}}({}^L\Dirac)]+[{}^R\Dirac,{}^R\Dirac-R_{\sqrt{g}^{-1}}({}^R\Dirac)]}.
\end{align*}

Since ${}^R\Dirac$ is $L$-invariant, ${}^L\Dirac-L_{\sqrt{g}}({}^L\Dirac)=\Dirac-L_{\sqrt{g}}(\Dirac)$. Moreover, since ${}^R\Dirac$ anti-commutes with ${}^L\Dirac$
and $L_{\sqrt{g}}({}^L\Dirac)$, we find that $[{}^L\Dirac,{}^L\Dirac-L_{\sqrt{g}}({}^L\Dirac)]=[\Dirac,\Dirac-L_{\sqrt{g}}(\Dirac)]$. 
Similarly, $[{}^R\Dirac,{}^R\Dirac-R_{\sqrt{g}^{-1}}({}^R\Dirac)]=[\Dirac,\Dirac-R_{\sqrt{g}^{-1}}(\Dirac)]$.
Thanks to Proposition \ref{equivariance condition on D} and the $R$-action version of it, we obtain the desired boundedness.

$(4)$ It is clear from the same proposition and the same computation in $(2)$.

\end{proof}

\subsection{Clifford symbol element $[\widetilde{\sigma_{\Dirac}^{Cl}}]$}

We will define the reformulated Clifford symbol element for our case. For this aim, we recall the formula in Lemma \ref{modified Kas modules}. We would like to define the Clifford symbol element by $(\SC(U),0)$, with a certain {\bf $\tau$-twisted action} which looks like the left regular representation on the line bundle $\ca{L}$. Formally, our Clifford symbol element is ``$(\scr{S}\otimes C_0(U,\ca{L}\otimes \Cl_+(T^*U)),0)$'' including the group action.

Let us define such a twisted action. Formally, the action is written as 
$$g.a(\bullet)=a(\bullet-g)e^{\frac{i}{2}\omega(f,\bullet)},$$
for $g=\exp(f)\in U$. In order to justify this definition, we use the finite-dimensional approximation, as usual. Let $g\in U_N$ and $a\in \SC(U_M)_\fin$. It is possible to assume that $N=M$, by $\beta_{M,N}:\SC(U_M)\to \SC(U_N)$ (when $M<N$) or the inclusion $U_N\subseteq U_M$ (when $N<M$). Since $a$ is of the form $\sum f_i\otimes \widetilde{a_i}$, $g.a$ should be defined by $\sum f_i\otimes g.\widetilde{a_i}$, where $ g.\widetilde{a_i}$ is defined by $\widetilde{a_i}(\bullet-g)e^{\frac{i}{2}\omega(f,\bullet)}$. The problem is whether this $U_\fin$-action extends to $U$ or not, as usual.

\begin{lem}
This action extends to a continuous $U$-action on $\SC(U)$. It is compatible with the original untwisted $U$-action on the $C^*$-algebra $\SC(U)$. That is, the pair $(\SC(U),0)$ equipped with the new action is a $\tau$-twisted $U$-equivariant Kasparov $(\SC(U),\SC(U))$-module.
\end{lem}
\begin{rmk}
The above ``compatibility'' means the following. To be precise, we introduce some symbols used only here: $\sigma$ is the original untwisted action, and $\rho$ is the introduced twisted action. For $a,a',a_1,a_2\in \SC(U)$, we will prove that $\rho_g(a\cdot a_1\cdot a')=\sigma_g(a)\cdot \rho_g(a_1)\cdot \sigma_g(a')$, and $[\rho_g(a_1)]^*\rho_g(a_2)=\sigma_g(a_1^*a_2)$.

\end{rmk}

\begin{proof}
Once the action is extended, the properties listed above are clear from the formula for $U_\fin$. We check that the action is well-defined and continuous. Just like the proof of Theorem \ref{the U action}, the above action extends to the action on $\SC(U)$, since each $g$ gives an isometry. It is enough to check that the action is continuous.

To prove the continuity, we will use the usual technique. Let $a\in \SC(U)$ and $g,g'\in U_\fin$. For any $\varepsilon>0$, choose $\sum f_i\otimes \widetilde{a_i}\in \SC(U)_\fin$ such that $\|a-\sum f_i\otimes \widetilde{a_i}\|<\varepsilon$. 
We may assume that $a$ can be approximated by a single element $f\otimes \widetilde{a}$. Moreover, we may assume that the both of $f$ and $\widetilde{a}$ are real-scalar-valued, since the $U$-action is trivial on the fiber $\Cl_+$.
We may prove that $\|g.(f\otimes \widetilde{a})-g'.(f\otimes \widetilde{a})\|$ is less than $\varepsilon$ when $g$ and $g'$ are close enough. We prove this estimate when $f=f_e$ and $f_o$. For other functions, one can use the technique of the proof of Theorem \ref{the U action}. 

Suppose that $g,g'\in U_N$ and $\widetilde{a}\in \SC(U_M)$. Using $\beta_{M,L}$ for $L\geq N$, 
we identify $f_e\otimes \widetilde{a}$ with
$$f_e\otimes e^{-(C_{M+1}^L)^2}\otimes \widetilde{a}.$$
Since the above twisted action on $\Cl_\tau(U_N)$ is continuous, we may ignore the $U_N$-part, just like the proof of Theorem \ref{the U action}. Thus, we may assume that $M=0$. Let us estimate
$$\left| e^{-\sum_b [b^{-2l}(x_b-g_b)^2+(y_b-h_b)^2]+\frac{i}{2}\sum (g_by_b-h_bx_b)}-
e^{-\sum_b b^{-2l}[(x_b-g'_b)^2+(y_b-h'_b)^2]+\frac{i}{2}\sum (g'_by_b-h'_bx_b)}\right|.$$

Put $r^2:=\sum_b b^{-2l}[(x_b-g_b)^2+(y_b-h_b)^2]$, $(r')^2:=
\sum_bb^{-2l}[(x_b-g'_b)^2+(y_b-h'_b)^2]$, $\omega(g,x):=\sum (g_by_b-h_bx_b)$ and $\omega(g',x):=\sum (g'_by_b-h'_bx_b)$. Then,
\begin{align*}
| e^{-r^2+\frac{i}{2}\omega(g,x)}-
e^{-(r')^2+\frac{i}{2}\omega(g',x)}| 
&\leq | e^{-r^2}-e^{-(r')^2}|+e^{-r^2}|e^{\frac{i}{2}\omega(g,x)}-e^{\frac{i}{2}\omega(g',x)}|.
\end{align*}
The first term is small, thanks to Theorem \ref{the U action}. 
The second one can be written as $e^{-2r^2}|e^{\frac{i}{2}\omega(g-g',x)}-1|$. Roughly speaking, this is small, because the growth of $\omega(g-g',x)$ as $x\to \infty$ is slow enough. The rigorous proof is the following.

Take a positive real number $\varepsilon$. Since $e^{-2r^2}$ converges to $0$ as $r\to \infty$, we can find a large number $K$ satisfying the following: if $r \geq K$, $e^{-r^2} < \varepsilon/2$. Note that the condition is satisfied, if $\sum_b b^{-2l}(x_b^2+y_b^2)\geq \bra{K+\sqrt{\sum_bb^{-2l}(g_b^2+h_b^2)}}^2$. Consequently, if $\sum_b b^{-2l}(x_b^2+y_b^2)\geq \bra{K+\sqrt{\sum_bb^{-2l}(g_b^2+h_b^2)}}^2$, then $e^{-2(r')^2}|e^{\frac{i}{2}\omega(g-g',x)}-1|<\varepsilon$.

Next, we must estimate the quantity for ``small'' $x$. Firstly, choose a positive real number $\delta$ such that $|e^{\frac{i}{2}t}-1|< \varepsilon$ if $|t|< \delta$. 
Then, from the following computation using the Cauchy Schwarz type inequality, we can find positive number $\delta'$ such that $|\omega(g-g',x)| < \delta$ for all $x$ satsfying that $\sum_b b^{-2l}(x_b^2+y_b^2)\leq \bra{K+\sum_bb^{-2l}(g_b^2+h_b^2)}^2+1$, if $\sum_kk^{2l}[(g_k-g_k')^2+(h_k-h_k')^2] < \delta'$. 
\begin{align*}
|\omega(g-g',x)|^2&= \left| \sum_b[(g_b-g_b')y_b-(h_b-h_b')x_b]\right|^2 \\
&\leq \sum_bb^{2l}[(g_b-g_b')^2+(h_b-h_b')^2]\times \sum_bb^{-2l}(x_b^2+y_b^2).
\end{align*}

Combining the technique for the twisted action on $f_e$ and the untwisted action on $f_o$, one can deal with the twisted action on $f_o$. We leave it to the reader.

\end{proof}

Thanks to this lemma, the $(\SC(U),\SC(U))$-bimodule $\SC(U)$ admits a $\tau$-twisted $U$-action which is compatible with the original untwisted $U$-action on $\SC(U)$. We have reached the following definition-theorem.

\begin{dfn-thm}\label{dfn-thm symbol element}
The pair $(\SC(U),0)$ is a $\tau$-twisted $U$-equivariant Kasparov $(\SC(U),\SC(U))$-module. The corresponding $KK$-element is called the {\bf Clifford symbol element} of $\Dirac$, and denoted by
$$[\widetilde{\sigma_{\Dirac}^{Cl}}]:=[(\SC(U),0)]\in KK_U^\tau(\SC(U),\SC(U)).$$
\end{dfn-thm}

\subsection{The $KK$-theoretical index theorem}

The following is the main theorem of the present paper, but it is almost a corollary of three definition-theorems in this section.

\begin{thm}\label{Main theorem}
The Kasparov module defined in Definition-Theorem \ref{dfn-thm index element} is a Kasparov product of the one defined in Definition-Theorem \ref{dfn-thm symbol element} and the one defined in Definition-Theorem \ref{dfn-thm Dirac element}. Formally, in $KK_U^\tau(\SC(U),\scr{S})$,
$$[\widetilde{\Dirac}]= [\widetilde{\sigma_{\Dirac}^{Cl}}] \otimes_{\SC(U)}[\widetilde{d_U}].$$
\end{thm}
\begin{rmk}
As explained in Section 3, the $LT$-equivariant $KK$-theory has not been extensively studied, and we have not proved that the following thing: {\it For two given $LT$-equivariant $KK$-elements, there is a Kasparov product of them, and such modules are unique up to homotopy}. This theorem means, however, there is an example of two Kasparov modules which has a Kasparov product.
\end{rmk}

\begin{proof}
If we forget the group action, we find that $[\widetilde{\sigma_{\Dirac}^{Cl}}] =[\id]$ and $[\widetilde{\Dirac}]=[\widetilde{d_U}]$, and hence the statement is obvious. The non-trivial thing is only the following: The isomorphism $\SC(U)\otimes_{\SC(U)}[\scr{S}\otimes \ca{H}]\to \scr{S}\otimes \ca{H}$ is $\tau$-twisted $U$-equivariant, that is, the following diagram commutes:
$$\begin{CD}
\SC(U)\otimes_{\scr{S}}[\scr{S}\otimes \ca{H}] @>>> \scr{S}\otimes \ca{H}\\
@Vg\otimes gVV @VVgV \\
\SC(U)\otimes_{\scr{S}}[\scr{S}\otimes \ca{H}] @>>> \scr{S}\otimes \ca{H}.
\end{CD}$$

In order to prove this commutativity, recall the precise definition of the group actions and the left $\SC(U)$-action on $\scr{S}\otimes \ca{H}$: All of them are defined by the limit of the finite-dimensional approximations. Therefore, we only have to prove the equivariance on $U_N$ for all $N$. Remember that the $U$-actions on the index element and the Clifford symbol element have the same formula, and the $U$-action on the HKT algebra $\SC(U)$ and the Dirac element have the same formula. Once noticing these things, one can prove the statement immediately.
\end{proof}

\section{Assembly maps}

We should study the connection between the main result of the present paper and the previous one of \cite{T2}. In this paper, we studied the assembly map and the analytic index without the index element: We defined the ``crossed product of $C_0(U)$ by $U$'' $\ud{C_0(U)\rtimes U}$, the ``$\tau$-twisted group $C^*$-algebra of $U$'' $\ud{\bb{C}\rtimes_\tau U}$,
the ``image of the index element along the descent homomorphism'' $\ud{j_U^\tau([\Dirac])}\in KK(\ud{C_0(U)\rtimes U},\ud{\bb{C}\rtimes_\tau U})$, the ``Mishchenko line bundle'' $\ud{[c]}\in KK(\bb{C},\ud{C_0(U)\rtimes U})$, and the ``analytic index in the sense of \cite{Kas15}'' $\ud{\ind(\Dirac)}\in KK(\bb{C},\ud{\bb{C}\rtimes_\tau U})$. These $KK$-elements satisfy the relation $\ud{j_U^\tau([\Dirac])}=\ud{[c]}\otimes_{\ud{C_0(U)\rtimes U}} \ud{j_U^\tau([\Dirac])}$.\footnote{Unlike the main result of this paper, this Kasparov product makes rigorous sense.}

On the other hand, we have studied the index element, the Clifford symbol element and the Dirac element in the present paper, without index. 

In this section, we try constructing ``an assembly map'', whose value is the same with the analytic index in \cite{T2}. This  construction is not elegant\footnote{This is the reason why we add the quotation mark.}, but it might suggest the existence of a (much better) assembly map which is compatible with the latter half of the assembly map constructed in \cite{T2}. In the latter half of this section, we discuss the satisfactory form of the assembly maps and the crossed products.

\subsection{``An assembly map'' by the restrictions to finite-dimensional subgroups}
In this subsection, we define ``an assembly map'' in an indirect way. 
We begin with the definitions of several ingredients: the Bott elements 
and the $\tau$-twisted group $C^*$-algebras of $U_N$'s. 

\begin{dfn}
For any finite-dimensional Euclidean space $V$, $Cl_\tau(V)$ is $KK$-equivalent to $\bb{C}$ via the $KK$-element
$$[b_V]:=[Cl_\tau(V),c]\in KK(\bb{C},Cl_\tau(V)),$$
where $c$ is the Clifford multiplication. At $v\in V$, $c$ acts on the fiber $\Cl_+(T^*_vV)$ via $c(v)$.
\end{dfn}

The followings are well-known: $[b_V]\otimes_{Cl_\tau(V)}[d_V]=1_\bb{C}$ and $[d_V]\otimes_{\bb{C}}[b_V]=1_{Cl_\tau(V)}$. The direct proof is not difficult, thanks to Proposition \ref{Kucs criterion}. The Kasparov product of the former one is given by the harmonic oscillator, whose index is $1$. We will prove an infinite-dimensional analogue of the former one, in our language.

The reformulated Bott element is just $\tau_{\scr{S}}([b_V])$.
It is truly a module, but it has been realized as a $*$-homomorphism $\beta_V:\scr{S}\to \SC(V)$ in \cite{HKT}. This reformulated one has an infinite-dimensional analogue $[\beta_{U_N^\perp}]$ and so on. For the special one $[\beta_U]$, we use the symbol $[\beta]$. In the following, we forget the equivariant structure on $[\widetilde{d_U}]$.

\begin{dfn}
For a real Hilbert space $V$, the $*$-homomorphism $\beta_V:\scr{S}\to \SC(V)$ defines a $KK$-element $[\beta_V]\in KK(\scr{S},\SC(V))$. In particular, when $V$ is given by our infinite-dimensional Lie group $U$, the Bott element is denoted by $[\beta]\in KK(\scr{S},\SC(U))$.

\end{dfn}

\begin{pro}
$[\beta]\otimes_{\SC(U)} [\widetilde{d_U}]=[1_{\scr{S}}]$ in $KK(\scr{S},\scr{S})$.
\end{pro}
\begin{proof}
Since $\beta$ is a $*$-homomorphism,
$[\beta]\otimes_{\SC(U)}[\widetilde{d_U}]$ is given by $[(\scr{S}\otimes \ca{H},\Dirac)]$, where the left action of $\scr{S}$ is defined by the composition
$$\phi:\scr{S}\xrightarrow{\beta}\SC(U)\xrightarrow{\Phi}\End_\scr{S}(\scr{S}\otimes\ca{H});$$
$\phi(f)=f(X\otimes \id+\id\otimes C)$, where $C$ is the Clifford operator on $\ca{H}$.
In order to find much simpler description, we deform $\phi$. Let 
$\phi_t(f):=f(X\otimes \id+\id\otimes tC)$. Note that $\phi_1=\phi$, and $\phi_0(f)=f\otimes \id$. Then, two Kasparov modules $(\scr{S}\otimes\ca{H},\phi,\Dirac)$ and $(\scr{S}\otimes\ca{H},\phi_0,\Dirac)$ are homotopic by the homotopy
$$(\scr{S}\otimes \ca{H}[0,1],\{\phi_t\}_{t\in [0,1]},\Dirac).$$
We need to prove that the homotopy $(\scr{S}\otimes \ca{H}[0,1],\{\phi_t\}_{t\in [0,1]},\Dirac)$ is actually a Kasparov $(\scr{S},\scr{S}[0,1])$-module. We leave it to the reader.

The new representative $(\scr{S}\otimes\ca{H},\phi_0,\Dirac)$ coincides with the tensor product of the Kasparov $(\scr{S},\scr{S})$-module $[(\scr{S},0)]$ and the Kasparov $(\bb{C},\bb{C})$-module $[(\ca{H},\Dirac)]$. The latter is $1_\bb{C}$, because $\Dirac$ is a Fredholm operator with index $1$.

\end{proof}

Let us recall the structure of the $\tau$-twisted group $C^*$-algebra $\bb{C}\rtimes_\tau U_N$ from \cite{T1}. $U_N$ has only one $\tau$-twisted irreducible unitary representation on $L^2(\bb{R}^N)$, and so $\bb{C}\rtimes_\tau U_N$ is isomorphic to $\ca{K}(L^2(\bb{R}^N)^*)$. The representation space $L^2(\bb{R}^N)^*$ has a specific vector denoted by ${\bf 1}_b^*$ which is of the highest weight with respect to the infinitesimal rotation action $d\rho^*(d)$.

\begin{dfn}[Definition 4.12 in \cite{T1}]
The $*$-homomorphism $i_N$ from $\bb{C}\rtimes_\tau U_N$ to $\bb{C}\rtimes_\tau U_{N+1}$ is defined as follows: Let $P$ be the rank one projection onto $\bb{C}{\bf 1}_b^*\in L^2(\bb{R}^{N+1}\ominus \bb{R}^N)^*$, and let $i_N(k)$ be $k\otimes P$ for $k\in \ca{K}(L^2(\bb{R}^N)^*)$. Using these homomorphisms, define the $C^*$-algebra $\ud{\bb{C}\rtimes_\tau U}$ by
$$\ud{\bb{C}\rtimes_\tau U}:=\varinjlim \bb{C}\rtimes_\tau U_N.$$
$j_N:\bb{C}\rtimes_\tau U_N\to \ud{\bb{C}\rtimes_\tau U}$ denotes the canonical $*$-homomorphism to define the inductive limit. As a result, $\ud{\bb{C}\rtimes_\tau U}$ is isomorphic to $\ca{K}(\ud{L^2(\bb{R}^\infty)^*})$.
\end{dfn}

Let us define assembly maps, parametrized by $N\in \bb{N}$. 
Recall the following isomorphisms: $\SC(U)\cong \SC(U_N^\perp)\otimes \scr{C}(U_N)$; $\SC(U)\rtimes U_N\cong \SC(U_N^\perp)\otimes \bra{\scr{C}(U_N)\rtimes U_N}$.

\begin{dfn}
The $N$-assembly map ${}^N\mu_{U}^\tau:KK_{U}^\tau(\SC(U),\scr{S}) \to KK(\scr{S},\scr{S}\rtimes_\tau U)$ is the composition of the following sequence:
$$KK_{U}^\tau(\SC(U),\scr{S}) \xrightarrow{\text{forget}} KK_{U_N}^\tau(\SC(U),\scr{S}) \xrightarrow{j_{U_N}^\tau} KK(\SC(U)\rtimes U_N,\scr{S}\rtimes_\tau U_N)$$
$$ \xrightarrow{([\beta_{U_N^\perp}]\otimes_\bb{C} [\widetilde{c}_{U_N}])\otimes-}
KK(\scr{S},\scr{S}\rtimes_\tau U_N)\xrightarrow{-\otimes[j_{N}]} KK(\scr{S},\scr{S}\rtimes_\tau U).$$

For an element $[D]\in KK_{U}^\tau(\SC(U),\scr{S})$, suppose that ${}^N\mu_{U}^\tau([D])={}^{N+1}\mu_{U}^\tau([D])=\cdots$ for sufficiently large $N$. We say that such $[D]$ is {\bf assemblable}.
The value of the {\bf total assembly map} is defined by $\mu_U^\tau([D]):={}^N\mu^\tau_{U}([D])$ for sufficient large $N$.
\end{dfn}

The main theorem of this section is that our index element is assemblable. 
Firstly, we recall Remark \ref{another perturbation}. Our index element $[\widetilde{\Dirac}]$ has another representative defined by the operator 
$${}_M\Dirac:={}^R\Dirac_1^M+\frac{1}{\sqrt{2}}{}^R\Dirac_{M+1}^\infty+\frac{i}{\sqrt{2}}{}^L\Dirac_{M+1}^\infty,$$
which is actually $U_N$-equivariant for $N\leq M$.
The new representative is denoted by $[{}_M\widetilde{\Dirac}]$.

Note that 
$$[{}_M\widetilde{\Dirac}]=[{}_M\widetilde{\Dirac}_{N+1}^\infty]\otimes_\bb{C}[{}_M\Dirac_{1}^N]\in KK_{U_N^\perp}^\tau(\SC(U_N^\perp),\scr{S})\otimes KK_{U_N}^\tau(Cl_\tau(U_N),\bb{C}).$$
Let us compute $j_{U_N}^\tau([{}_M\widetilde{\Dirac}])$, after forgetting the $U_N^\perp$-action. It is easy to see that
$j_{U_N}^\tau([{}_M\widetilde{\Dirac}])=[{}_M\widetilde{\Dirac}_{N+1}^\infty]\otimes_\bb{C}j_{U_N}^\tau[{}_M\Dirac_{1}^N].$
We have computed this element. See \cite{T2} for the details.

\begin{pro}[Theorem 3.7 in \cite{T2}]
$j_{U_N}^\tau([{}_M\widetilde{\Dirac}_{1}^N])$ is given by
$$[(L^2(U_N)\otimes S\otimes [\bb{C}\rtimes_\tau U_N],D\otimes_2\otimes \id+\id\otimes_1 {}^L\Dirac_1^N)]\in KK(C_0(U_N)\rtimes U_N,\bb{C}\rtimes_\tau U_N).$$
\end{pro}

\begin{pro}[Proposition 3.9 in \cite{T2}]
$\mu_{U_N}^\tau([{}_M\Dirac_1^N])$ is given by
$$[c_{U_N}]\otimes_{C_0(U_N)\rtimes U_N}j_{U_N}^\tau([{}_M\widetilde{\Dirac}])=[(S_{U_N}\otimes (\bb{C}\rtimes_\tau U_N),{}^L\Dirac_1^N)].$$
This element is also represented as $[(L^2(\bb{R}^N)\widehat{\oplus} 0,0)]$.
\end{pro}

Let us compute the rest part of the $N$-assembly map in our language. Since the $KK$-element $[\widetilde{{}_M\Dirac_{N+1}^\infty}]$ does not depend on $M$ in $KK(\SC(U),\scr{S})$, we put $M=N$, and we remove $M$ in the symbol. What we should compute is $[b_{U_N^\perp}]\otimes_{\SC(U_N^\perp)}[\widetilde{\Dirac_{N+1}^\infty}]$. If we forget the group action, $[\widetilde{d_U}]$ coincides with $[\widetilde{\Dirac}]$. Thus, $[b_{U_N^\perp}]\otimes_{\SC(U_N^\perp)}[\widetilde{\Dirac}_{N+1}^\infty]=1$. Moreover, 
since $j_{N}$ maps $[(L^2(\bb{R}^N)\widehat{\oplus} 0,0)] \in KK(\bb{C},\bb{C}\rtimes_\tau U_N)$ to $[(\ud{L^2(\bb{R}^\infty)}\widehat{\oplus} 0,0)]\in KK(\bb{C},\bb{C}\rtimes_\tau U)$, we obtain the following theorem.

\begin{thm}\label{thm assembly map}
${}^N\mu_{U}^\tau([{}_M\Dirac]) = [(\ud{L^2(\bb{R}^\infty)}\widehat{\oplus} 0,0)]$ for all $N$. 
The RHS is independent of $N$, and our index element is assemblable; $\mu_{U}^\tau([\Dirac])=[(\ud{L^2(\bb{R}^\infty)}\widehat{\oplus} 0,0)]$.

This value coincides with $\ud{\ind({}^R\Dirac)}$ defined in \cite{T2}.
\end{thm}

\subsection{A satisfactory form of the assembly map}

The previous subsection might suggest that we can define the total assembly map without making reference to the finite-dimensional approximations: $\ud{j_U^\tau}([\widetilde{\Dirac}])$ should be ``$\varprojlim {j_{U_N}^\tau}([\widetilde{\Dirac}])$'' in some sense. We discuss such things in this subsection. There are no rigorous results here, but we will write several conjectures and observations on them.

\begin{conj}
We can define two $C^*$-algebras which play roles of ``$\SC(U)\rtimes U$'' and ``$\SC(U)\rtimes_\tau U$'', denoted by $\ud{\SC(U)\rtimes U}$ and $\ud{\SC(U)\rtimes_\tau U}$, respectively. 
\end{conj}

The $C^*$-algebra $\scr{S}\otimes \ca{K}(\ud{L^2(U)}\otimes S^*)$ is a strong candidate.
We have defined a $C^*$-algebra $\ud{ C_0(U)\rtimes U}$ in \cite{T2}, by ``believing'' that the well-known isomorphism $C_0(G)\rtimes G\cong \ca{K}(L^2(G))$ for any locally compact group $G$, is valid even for infinite-dimensional cases. Just like this, for even-dimensional $Spin^c$-Lie group $G$, $Cl_\tau(G)\rtimes G\cong \ca{K}(L^2(G,S_G))$ for the trivial Spinor bundle $S_G$. This is because $Cl_\tau(G)\cong C_0(G)\otimes \Cl_+(T^*_eG)$ and $\Cl_+(T^*_eG)\cong \End(S_e)$.

On the other hand, for an $N$-dimensional parallelizable manifolrd $X$, we may regard $Cl_\tau(X)$ as $C_0(X)\otimes \Cl_+(\bb{R}^N)$. Thus the $C^*$-algebra $\ca{K}(\ud{L^2(U)})\otimes \CAR$ is also a strong candidate, where $\CAR$ is the canonical anti-commutation relation algebra.\footnote{$\CAR$ is defined by the inductive limit $\varinjlim \Cl_+(\bb{R}^{2N})$. It is also regarded as the infinite tensor product $M_2(\bb{C})\otimes M_2(\bb{C})\otimes M_2(\bb{C})\otimes \cdots$.}

These two algebras are truly different: $K_0(\CAR)\cong\bb{Z}[\frac{1}{2}]$ while $K_0(\ca{K})\cong \bb{Z}$. We must choose the most appropriate $C^*$-algebra for the crossed products.

Suppose that the above conjecture is true. We need the descent homomorphisms.

\begin{conj}
We can also define the following homomorphisms:
\begin{itemize}
\item $\ud{j_{U}^\tau}:KK_{U}^\tau(\SC(U),\scr{S})\to KK(\ud{\SC(U)\rtimes U},\ud{\scr{S}\rtimes_\tau U})$.
\item $\ud{j_{U}^\tau}:KK_{U}^\tau(\SC(U),\SC(U))\to KK(\ud{\SC(U)\rtimes U},\ud{\SC(U)\rtimes_\tau U})$.
\item $\ud{j_{U}^{1}}: KK_{U}(\SC(U),\scr{S})\to KK(\ud{\SC(U)\rtimes_\tau U},\ud{\scr{S}\rtimes_\tau U})$.
\end{itemize}

These descent homomorphisms are related in the following sense: for $[x]\in KK_{U}^\tau(\SC(U),\scr{S})$, $[y]\in KK_{U}^\tau(\SC(U),\SC(U))$ and $[z]\in KK_{U}(\SC(U),\scr{S})$ satisfying that $[x]=[y]\otimes [z]$,
$$\ud{j_{U}^\tau}([x])=\ud{j_U^\tau}([y])\otimes \ud{j_{U}^1}([z]).$$

Moreover, the reformulated Mishchenko line bundle $\ud{[\widetilde{c}]}\in KK(\scr{S},\ud{\SC(U)\rtimes U})$ can be defined, and 
$$\tau_\scr{S}(\ud{\ind(\Dirac)})=\ud{[\widetilde{c}_U]}\otimes\ud{ j_{U}^\tau}([\widetilde{\Dirac}]).$$ 
\end{conj}

We have defined the $KK$-element $\ud{j_U^\tau([\Dirac])}$ in \cite{T2}. However, the above conjecture is still difficult. This is because we must define the images for all Kasparov modules. 
However, the construction of the ``crossed products'' will definitely give us some nice hints.

If all of these conjectures are correct, we will obtain the index formula using the symbol element.
$$\tau_\scr{S}(\ud{\ind(\Dirac)})=\ud{[\widetilde{c}_U]}\otimes \ud{j_{U}^\tau}([\widetilde{\sigma_{\Dirac}^{Cl}}])\otimes \ud{j_{U}^1}([\widetilde{d_U}]).$$

We hope to solve these conjectures in the immediate future!

\section*{Acknowledgements}
I am grateful to Nigel Higson for teaching me the important point of \cite{HK}. My operator $\Dirac$ (which looks like the perturbed Bott-Dirac operator) is inspired by that suggestion. 

I am supported by JSPS KAKENHI Grant Number 18J00019.

Graduate School of Mathematical Sciences, the University of Tokyo, 3-8-1 Komaba Meguro-ku Tokyo, Japan

E-mail address: {\tt dtakata@ms.u-tokyo.ac.jp}

\end{document}